\documentclass[final,3p]{elsarticle}

\usepackage[utf8]{inputenc}
\usepackage{graphicx}
\usepackage{subfigure}
\usepackage{hyperref}
\usepackage[usenames,dvipsnames]{color}

\usepackage{amsmath}
\usepackage{mathrsfs}
\usepackage{amsfonts}

\usepackage{amssymb}
\usepackage{amsthm}

\usepackage{xifthen}
\newtheorem{theorem}{Theorem}
\newtheorem{definition}[theorem]{Definition}
\newtheorem{lemma}[theorem]{Lemma}
\newtheorem{corollary}[theorem]{Corollary}
\newtheorem{remark}[theorem]{Remark}

\newtheorem{algorithm}[theorem]{Algorithm}
\newtheorem{proposition}[theorem]{Proposition}

\journal{Comput. Methods Appl. Mech. Engrg.
}

\newcommand{\hide}[1]{}
\newcommand{\rev}[2]{#2}
\newcommand{\skipifemptyarg}[1]{\ifthenelse{\isempty{#1}}{}{\left[#1\right]}}
\newcommand{\skipifscalar}[1]{\ifthenelse{\isempty{#1}}{}{;#1}}

\newcommand{\bs}[1]{\boldsymbol{#1}}
\newcommand{\scal}[2]{\bigl(#1,#2\bigr)}
\newcommand{\bildual}{^{-1}}
\newcommand{\bilf}[2]{a\ifthenelse{\isempty{#1}}{}{\bigl(#1,#2\bigr)}}
\newcommand{\bilfi}[2]{a\bildual\ifthenelse{\isempty{#1}}{}{\bigl(#1,#2\bigr)}}
\newcommand{\bilfG}[2]{\mathring{a}\ifthenelse{\isempty{#1}}{}{\bigl(#1,#2\bigr)}}
\newcommand{\bilfGi}[2]{\mathring{a}\bildual\ifthenelse{\isempty{#1}}{}{\bigl(#1,#2\bigr)}}
\newcommand{\bilfN}[2]{a_\VN\ifthenelse{\isempty{#1}}{}{\bigl(#1,#2\bigr)}}
\newcommand{\bilfNi}[2]{a_\VN\bildual\ifthenelse{\isempty{#1}}{}{\bigl(#1,#2\bigr)}}
\newcommand{\bilfDN}[2]{\M{a}_\VN\ifthenelse{\isempty{#1}}{}{\bigl(#1,#2\bigr)}}
\newcommand{\bilfDNi}[2]{\M{a}_\VN\bildual\ifthenelse{\isempty{#1}}{}{\bigl(#1,#2\bigr)}}

\newcommand{\set}[1]{\mathbb{#1}}
\newcommand{\M}[1]{\mat{#1}}
\newcommand{\MB}[1]{\boldsymbol{\mat{#1}}}
\newcommand{\mat}[1]{\mathsf{#1}} 
\newcommand{\V}[1]{\bs{#1}}
\newcommand{\T}[1]{\bs{#1}}
\newcommand{\ol}[1]{\overline{#1}}
\newcommand{\ul}[1]{\underline{#1}}
\newcommand{\oul}[1]{\overline{\underline{#1}}}

\newcommand{\mac}[1]{^{(#1)}}
\newcommand{\sub}[1]{^{#1}}
\newcommand{\incl}[1]{_{(#1)}}

\newcommand{\puc}{\mathcal{Y}}
\newcommand{\lpuc}{Y}
\newcommand{\pVN}{\meas{\VN}}
\newcommand{\ptVN}{\meas{\tVN}}
\newcommand{\per}{\#}
\newcommand{\eff}{{\mathrm{H}}}
\newcommand{\dime}{d}

\newcommand{\imu}{\mathrm{i}}		
\newcommand{\Tz}{\T{\xi}}
\newcommand{\cA}{c_A}
\newcommand{\CA}{C_A}

\newcommand{\dst}{{h}}

\newcommand{\Aeff}{\TA_{\eff}}
\newcommand{\Beff}{\TB_{\eff}}
\newcommand{\AeffN}{\TA_{\eff,\VN}}
\newcommand{\BeffN}{\TB_{\eff,\VN}}
\newcommand{\oAeffN}{\ol{\TA}_{\eff,\VN}}
\newcommand{\oBeffN}{\ol{\TB}_{\eff,\VN}}
\newcommand{\ouAeffN}{\ol{\ul{\TA}}_{\eff,\VN}}

\newcommand{\oAeffdst}{\overline{\TA}_{\eff,\dst}}
\newcommand{\oBeffdst}{\overline{\TB}_{\eff,\dst}}

\newcommand{\cb}[1]{\V{U}^{(#1)}}

\newcommand{\ep}{\ensuremath{\varepsilon}}
\newcommand{\del}{\ensuremath{\delta}}
\newcommand{\alp}{\ensuremath{\alpha}}
\newcommand{\bet}{\ensuremath{\beta}}

\newcommand{\xR}{\set{R}}
\newcommand{\xN}{\set{N}}
\newcommand{\xZ}{\set{Z}}
\newcommand{\xC}{\set{C}}
\newcommand{\xRd}{{\set{R}^{d}}}
\newcommand{\xCd}{\set{C}^{d}}
\newcommand{\xRdd}{\set{R}^{d\times d}}
\newcommand{\xCdd}{\set{C}^{d\times d}}
\newcommand{\xRdN}{\xR^{d\times \VN}}
\newcommand{\xRdtN}{{\xR^{d\times(\tVN)}}}

\newcommand{\xCdN}{\xC^{d\times \VN}}

\newcommand{\xRddspd}{\set{R}_{\mathrm{spd}}^{d\times d}}
\newcommand{\xMN}{\bigl[\set{R}^{d\times\VN}\bigr]^2}
\newcommand{\xMtN}{\bigl[\set{R}^{d\times(\tVN)}\bigr]^2}
\newcommand{\xhMN}{\bigl[\set{C}^{d\times\VN}\bigr]^2}

\newcommand{\xXN}{\xRdN}
\newcommand{\xXtN}{\xR^{d\times(\tVN)}}
\newcommand{\xhXN}{\set{C}^{d\times\VN}}
\newcommand{\xX}{\set{X}}
\newcommand{\hX}{\hat{\set{X}}}
\newcommand{\cH}{\mathscr{H}}

\newcommand{\xNd}{\set{N}^d}
\newcommand{\Zd}{\set{Z}^d}
\newcommand{\ZdmO}{\set{Z}^d \backslash \{\T{0}\}}

\newcommand{\ZN}{\set{Z}_{\VN}}
\newcommand{\ZNdr}{\mathring{\set{Z}}^d_{\VN}}
\newcommand{\ZtNd}{\set{Z}^d_{\tVN}}
\newcommand{\ZNd}{\set{Z}^d_{\VN}}
\newcommand{\ZNNd}{\xZ^d_{\tVN}}

\newcommand{\cU}{\mathscr{U}}
\newcommand{\cE}{\mathscr{E}}
\newcommand{\cJ}{\mathscr{J}}
\newcommand{\rcU}{\mathring{\mathscr{U}}}
\newcommand{\rcE}{\mathring{\mathscr{E}}}
\newcommand{\rcJ}{\mathring{\mathscr{J}}}
\newcommand{\cUN}{\cU_\VN}
\newcommand{\cEN}{\cE_\VN}
\newcommand{\cJN}{\cJ_\VN}
\newcommand{\cENapp}{\tilde{\cE}_\VN}
\newcommand{\cJNapp}{\tilde{\cJ}_\VN}

\newcommand{\Lper}[2]{L^{#1}_\#(\puc\skipifscalar{#2})}
\newcommand{\Ltp}{L^{2}_{\per}}
\newcommand{\Cper}[2]{C^{#1}_\#(\puc\skipifscalar{#2})}

\newcommand{\zmean}{\mean{0}}

\newcommand{\FT}[1]{\widehat{#1}}
\newcommand{\bfun}[1]{\varphi_{#1}}
\newcommand{\bfunN}[1]{\bfun{\T{N},#1}}

\newcommand{\xUN}{\set{U}_\VN}
\newcommand{\xEN}{\set{E}_\VN}
\newcommand{\xJN}{\set{J}_\VN}
\newcommand{\xENapp}{\tilde{\set{E}}_\VN}
\newcommand{\xJNapp}{\tilde{\set{J}}_\VN}
\newcommand{\cT}{\mathscr{T}}

\newcommand{\cTtNd}{\cT^d_\tVN}

\newcommand{\cTNd}{\cT_\VN^d}
\newcommand{\cTNtd}{\tilde{\cT}_\VN^d}

\newcommand{\xNk}{\Vx^{\Vk}_\VN}
\newcommand{\xNm}{\Vx^{\Vm}_\VN}

\newcommand{\TA}{\ensuremath{\T{A}}}
\newcommand{\rTA}{\ensuremath{\mathring{\T{A}}}}

\newcommand{\rTB}{\ensuremath{\mathring{\T{B}}}}
\newcommand{\TAi}{\ensuremath{\T{A}^{-1}}}
\newcommand{\TB}{\ensuremath{\T{B}}}
\newcommand{\mC}{\ensuremath{\T{C}}}
\newcommand{\mD}{\ensuremath{\T{D}}}
\newcommand{\mI}{\ensuremath{{\T{I}}}}
\newcommand{\mM}{\ensuremath{\T{M}}}

\newcommand{\mhG}{\ensuremath{\T{\hat{\Gamma}}}}
\newcommand{\hG}{\ensuremath{\hat{\Gamma}}}
\newcommand{\VN}{\ensuremath{{\V{N}}}}
\newcommand{\tVN}{\ensuremath{{2\VN-\V{1}}}}
\newcommand{\VM}{\ensuremath{{\V{M}}}}
\newcommand{\Ve}{\ensuremath{\V{e}}}
\newcommand{\Vj}{\ensuremath{\V{\jmath}}}
\newcommand{\rVe}{\ensuremath{\mathring{\Ve}}}
\newcommand{\tVe}{\ensuremath{\Ve}}
\newcommand{\tVj}{\ensuremath{\Vj}}
\newcommand{\rVj}{\ensuremath{\mathring{\Vj}}}
\newcommand{\VE}{\ensuremath{\V{E}}}
\newcommand{\VJ}{\ensuremath{\V{J}}}
\newcommand{\Vu}{\ensuremath{\V{u}}}
\newcommand{\Vv}{\ensuremath{\V{v}}}
\newcommand{\Vw}{\ensuremath{\V{w}}}
\newcommand{\Vx}{\ensuremath{\V{x}}}
\newcommand{\mx}{{\V{X}}}

\newcommand{\Vm}{\ensuremath{{\V{m}}}}
\newcommand{\Vn}{\ensuremath{{\V{n}}}}
\newcommand{\Vk}{\ensuremath{{\V{k}}}}
\newcommand{\Vl}{\ensuremath{{\V{l}}}}
\newcommand{\Vh}{\ensuremath{\V{h}}}
\newcommand{\Vo}{\ensuremath{{\V{0}}}}
\newcommand{\Vxi}{\ensuremath{\V{\xi}}}

\newcommand{\MBe}{\ensuremath{\MB{e}}}
\newcommand{\tMBe}{\ensuremath{\tilde{\MB{e}}}}

\newcommand{\MBj}{\ensuremath{\MB{j}}}
\newcommand{\tMBj}{\ensuremath{\tilde{\MB{j}}}}
\newcommand{\MBA}{\ensuremath{\MB{A}}}

\newcommand{\MBu}{\ensuremath{\MB{u}}}
\newcommand{\MBv}{\ensuremath{\MB{v}}}
\newcommand{\MBx}{\ensuremath{\MB{x}}}
\newcommand{\MBB}{\ensuremath{\MB{B}}}
\newcommand{\MBG}{\MB{G}}
\newcommand{\MBhG}{\ensuremath{\MB{\hat{G}}}}

\newcommand{\MhG}{\ensuremath{\M{\hat{G}}}}
\newcommand{\MBF}{\ensuremath{\MB{F}}}
\newcommand{\MBFi}{\ensuremath{\MB{F}^{-1}}}

\DeclareMathOperator{\rect}{rect}

\DeclareMathOperator*{\argmin}{arg\,min}

\DeclareMathOperator{\sinc}{sinc}
\DeclareMathOperator{\tr}{tr}
\newcommand{\IN}{\mathcal{I}_\VN}
\newcommand{\IM}{\mathcal{I}_\VM}
\newcommand{\INi}{\mathcal{I}_\VN^{-1}}
\newcommand{\ItN}{\mathcal{I}_\tVN}

\newcommand{\QN}{\mathcal{Q}_\VN}
\newcommand{\PN}{\mathcal{P}_\VN}
\DeclareMathOperator{\esssup}{ess\,sup}
\newcommand{\norm}[2]{\bigl\| #1 \bigr\|_{#2}}
\newcommand{\meas}[1]{|#1|}
\newcommand{\mean}[1]{\langle#1\rangle}
\newcommand{\D}[1]{\,{\mathrm d}#1}
\newcommand{\conj}[1]{\overline{#1}}
\DeclareMathOperator{\curl}{curl}
\let\div\undefined
\DeclareMathOperator{\div}{div}

\begin{document}

\begin{frontmatter}



\title{Guaranteed upper-lower bounds on homogenized properties by FFT-based Galerkin method
}

\author[label1]{Jaroslav Vond\v{r}ejc\corref{cor1}}\ead{vondrejc@gmail.com}
\author[label2,label3]{Jan Zeman}\ead{
zemanj@cml.fsv.cvut.cz}
\author[label4]{Ivo Marek}\ead{ marekivo@mat.fsv.cvut.cz}

\cortext[cor1]{Corresponding author}

\address[label1]{New Technologies for the Information Society, Faculty of Applied Sciences, University of West Bohemia, Univerzitn\'{i} 2732/8, 306 14 Plze\v{n}, Czech Republic.}

\address[label2]{Department of Mechanics, Faculty of Civil Engineering, Czech Technical  University in Prague, Th\'{a}kurova 7, 166 29 Prague 6, Czech Republic.}

    \address[label3]{Centre of Excellence IT4Innovations, V\v{S}B-TU Ostrava, 17.~listopadu 15/2172, 708 33 Ostrava-Poruba, Czech Republic.}

\address[label4]{Department of Mathematics, Faculty of Civil Engineering, Czech Technical University in Prague, Th\'{a}kurova 7, 166 29 Prague 6, Czech Republic.}
 
\begin{abstract}
Guaranteed upper-lower bounds on homogenized coefficients, arising from the periodic cell problem, are calculated in a scalar elliptic setting. Our approach builds on the recent variational reformulation of the Moulinec-Suquet (1994) Fast Fourier Transform (FFT) homogenization scheme by Vond\v{r}ejc et al. (2014), which is based on the conforming Galerkin approximation with trigonometric polynomials.
Upper-lower bounds are obtained by adjusting the primal-dual finite element framework developed independently by Dvo\v{r}\'{a}k (1993) and Wi\c{e}ckowski (1995) to the FFT-based Galerkin setting.
We show that the discretization procedure differs for odd and non-odd number of grid points. Thanks to the Helmholtz decomposition inherited from the continuous formulation, the duality structure is fully preserved for the odd discretizations. In the latter case, a more complex primal-dual structure is observed due to presence of the trigonometric polynomials associated with the Nyquist frequencies.
These theoretical findings are confirmed with numerical examples.
To conclude, the main advantage of the FFT-based approach over conventional finite-element schemes is that the primal and the dual problems are treated on the same basis, and this property can be extended beyond the scalar elliptic setting.
\end{abstract}

\end{frontmatter}


\section{Introduction}
\label{sec:introduction}
This work is dedicated to the determination of guaranteed upper-lower bounds on homogenized (effective) material coefficients originating from the theory of homogenization of periodic media. These bounds, which are essential for the development of reliable multi-scale simulations \cite{OBBH2003research-directions}, are calculated with an FFT-based Galerkin approach, a method introduced by the authors in \cite{VoZeMa2014FFTH} as a variational reformulation of the fast iterative scheme proposed by Suquet and Moulinec in~\cite{Moulinec1994FFT}. Since our objective is to develop a general methodology, we restrict our attention to scalar linear elliptic problems. Despite this limitation, we believe that our results are relevant to various FFT-based analyses of complex material systems e.g.~\cite[and references therein]{Montagnat2014multiscale,Sliseris2014,Stein2014fatigue}.

In this introduction, we briefly describe the basic framework of periodic homogenization leading to a cell problem, a variational problem that defines the homogenized matrix. We then discuss possible methods for its numerical treatment with an emphasis on FFT-based schemes and approaches and connect them to techniques for obtaining guaranteed bounds on the homogenized matrix. Finally, we introduce the structure of the paper.

\subsection{Periodic cell problem}
Using the notation introduced in Section~\ref{sec:preliminaries}, let us consider an open set $\Omega\subset\xRd$ with a Lipschitz boundary and a positive parameter $\ep>0$ denoting the characteristic size of microstructure. We search for the scalar quantity $u^\ep:\Omega\rightarrow\xR$,
$u^\ep\in H^1_0(\Omega)$,
satisfying the variational equation 
\begin{align}
\label{eq:weak_form_oscillating}
\int_\Omega \scal{\TA^{\ep}(\mx)\nabla u^{\ep}(\mx)}{ \nabla v(\mx)}_{\xRd}\D{\mx}
= F(v)
\quad\forall v\in H^1_0(\Omega),
\end{align}
where $\scal{\cdot}{\cdot}_{\xRd}:\xRd\times\xRd\rightarrow\xRd$ denotes the standard scalar product on $\xRd$.
The linear functional $F:H^1_0(\Omega)\rightarrow\xR$ covers both the prescribed source terms and various boundary conditions,
and $\TA^{\ep}:\Omega\rightarrow\xRdd$ represents the symmetric, uniformly elliptic, and bounded matrix field of material coefficients, i.e. $\TA^{\ep}\in L^{\infty}(\Omega;\xRddspd)$. For the purpose of this work, we focus on periodic media, for which
\begin{align*}
\TA^{\ep}(\mx) = \TA\left(\frac{\mx}{\ep}\right)\quad\text{for }\mx\in\Omega,
\end{align*}
where the symmetric and uniformly elliptic matrix $\TA:\xRd\rightarrow\xRdd$ is $\puc$-periodic with $\puc=\prod_{\alp=1}^d \bigl(-\frac{Y_\alp}{2},\frac{Y_\alp}{2}\bigr)\subset{\xRd}$ denoting the periodic cell, cf.~\eqref{eq:A}. Hence, the coefficients $\TA^\ep$ develop finer oscillations with a decreasing microstructural parameter~$\ep$.

The unique solution to \eqref{eq:weak_form_oscillating} exists thanks to the Lax-Milgram lemma, and thus it can be numerically approximated by, for example, the standard Finite Element Method (FEM). However, in order to obtain a satisfactory approximation, the element size must satisfy $h\ll \ep \meas{\puc} \ll 1$, which renders the direct approach infeasible due to excessive computational demands. 

Alternatively, the complexity of \eqref{eq:weak_form_oscillating} can be reduced by homogenization. It involves a limit process for $\ep\rightarrow 0$, leading to the decomposition of the problem into the macroscopic and the microscopic parts.
This limit passage can be performed by various techniques, such as formal asymptotic expansion \cite{Bensoussan1978per_structures}, 
two-scale convergence methods \cite{nguetseng1989general,allaire1992homogenization}, or periodic unfolding \cite{cioranescu2008unfolding}.

Irrespective of the method used, we find that the solutions $u^\ep$ converge weakly in $H^1_0(\Omega)$ to a limit state $u_\eff$ described by the macroscopic variational equation
\begin{align*}
\int_\Omega \scal{\TA_\eff\nabla u_\eff(\mx)}{\nabla v(\mx)}_{\xRd}\D{\mx}
= F(v)
\quad\forall v\in H^1_0(\Omega).
\end{align*}
Here, $\Aeff\in\xRddspd$ represents the homogenized matrix of  coefficients $\TA^{\ep}$ that is described by the microscopic variational formulation defined on the periodic cell $\puc$ only
\begin{align}
\label{eq:micro}
\scal{\Aeff \VE}{\VE}_{\xRd} 
&= \min_{v\in H_{\per,\zmean}^1(\puc)} \frac{1}{\meas{\puc}} \int_\puc \scal{\TA(\Vx)[\VE+\nabla v(\Vx)]}{[\VE+\nabla v(\Vx)]}_{\xRd}\D{\Vx},
\end{align}
where $H_{\per,\zmean}^1(\puc)$ denotes the space of $\puc$-periodic functions with square integrable gradients and zero mean, cf.~Section~\ref{sec:preliminaries}, and \eqref{eq:micro} must hold for any vector $\VE\in\xRd$.

\subsection{FFT-based homogenization methods}

The numerical solution of the cell problem \eqref{eq:micro}, particularly an approximation to the homogenized matrix $\Aeff$, can be carried out by various approaches such as Finite Differences~\cite{Flaherty1973,Panasenko1988,Garboczi:1998:FEFD},
Finite Elements~\cite{Guedes1990,Michel1999,Geers2010},
Boundary Elements \cite{Eischen1993BEM,Kaminski1999BEM,Prochazka2001BEM},
or Fast Multipole Methods~\cite{Greengard1998,Greengard2006,Helsing2011effective}.
Here, we focus on FFT-based methods, efficient solvers developed for cell problems with coefficients $\TA$ defined by general high-resolution images.

The original FFT-based formulation proposed by Moulinec and Suquet in \cite{Moulinec1994FFT} is based on an iterative solution to the integral Lippmann-Schwinger equation corresponding to~\eqref{eq:micro} by the Neumann series expansion. Efficiency of the algorithm is achieved by approximating and evaluating the action of the integral kernel by the Fast Fourier Transform~(FFT) algorithm in only $\mathcal{O}(N\log N)$ operations, as both the data of the problem and its solution are defined on a regular periodic grid.
A theoretical background to the original algorithm has been provided only recently by interpreting the method as a suitable Galerkin scheme and proving the convergence of approximate solutions to the continuous one. 
In particular, the work of Brisard and Dormieux~\cite{Brisard2010FFT,Brisard2012FFT} utilizes the Hashin-Shtrikman variational principles~\cite{hashin1962elast} combined with pixel or voxel-wise constant basis functions.
Our approach~\cite{VoZeMa2014FFTH} builds on standard variational principles and approximation spaces of trigonometric polynomials \rev{}{along with finite element method-based convergence results for smooth data}, which have been generalized for rough coefficients in \cite{Schneider2014convergence}.
Besides, several improvements of the original solver, leading to faster convergence \cite{Eyre1999FNS,Vinogradov2008AFFT,ZeVoNoMa2010AFFTH,Brown2002DFT} or higher robustness \cite{Michel2000CMB,Monchiet2012polarization}, have been proposed along with heuristic approaches to increase the accuracy of local fields based on the incorporation of the so-called shape functions \cite{Monchiet2012polarization,Monchiet2013conduct} or modification of the integral kernel \cite{willot2013fourier}.

The present work is based on our recent study \cite{VoZeMa2014FFTH}, which shows that the original Moulinec-Suquet scheme is equivalent to a Galerkin discretization of a weak solution to the cell problem \eqref{eq:micro}, when the approximation space is spanned by trigonometric polynomials and a suitable numerical quadrature scheme is used to evaluate the linear and bilinear forms.
We also demonstrated that the system of linear equations arising from the discretization can be efficiently solved with Krylov solvers, cf.~\cite{ZeVoNoMa2010AFFTH,VoZeMa2012LNSC}, and that the action of the system matrix can be efficiently evaluated by FFT.
To minimize technicalities, the analysis was restricted to the primal formulation and to grids with odd number of points along each coordinate.

\subsection{Upper and lower bounds on homogenized matrix}

The theory of rigorous bounds on the homogenized matrix has been the subject of many studies in analytical homogenization theories. These techniques employ the primal-dual formulations of \eqref{eq:micro} under limited --- and often uncertain --- information on the material coefficients $\TA$.
Specific examples include the Voigt \cite{voigt1910lehrbuch}, Reuss \cite{Reuss1929}, and Hashin-Shtrikman  bounds \cite{hashin1963variational}; see the monographs \cite{Milton2002TC,torquato2002random,cherkaev2000variational,dvorak2012micromechanics} for a more complete overview.
Because the bounds rely on limited data, their performance rapidly deteriorates for highly-contrasted media.

Relatively less attention has been given to the upper-lower bounds arising from an approximate solution to \eqref{eq:micro} obtained by a numerical method. To our knowledge, the pioneering work relevant to FEM has been made by Dvořák and Haslinger \cite{Dvorak1993master,Dvorak1995RNM,haslinger1995optimum}, who proposed a general framework for elliptic problems, developed unified primal-dual $p$-version solvers for the two-dimensional scalar equation, and  applied them later to the optimal design of matrix-inclusion composites. Error estimates and convergence rates of homogenized properties are provided there together with a reformulation using stream functions which leads to a dual formulation with the same structure as the primal one in the two-dimensional scalar setting.
In more general situations, mixed approaches are usually needed to approximate the dual formulation, as demonstrated by Wi\c{e}ckowski \cite{Wieckowski1995DFEM} for linear elasticity. 

Let us note that FFT-based bounds on a homogenized properties have also been investigated independently in \cite{kabel2012precisebounds,bignonnet2014fft}, utilizing the Brisard-Dormieux approach~\cite{Brisard2012FFT,Brisard2010FFT}. In this case, however, the evaluation of guaranteed bounds involves an infinite sum that converges very slowly, and a truncation of the sum violates the structure of the guaranteed bounds; see~\cite[Section~7]{VoZeMa2014FFTH} for a related discussion. This limitation is overcome here by a suitable integration rule developed in Section~\ref{sec:evaluation_bounds}.

\subsection{Content of the paper}

The aim of this paper is twofold. First, we demonstrate that the FFT-Galerkin method is directly applicable to the Dvořák-Haslinger setting~\cite{Dvorak1993master,Dvorak1995RNM,haslinger1995optimum} and that it naturally generates primal and dual problems with the same structure. We then extend our results from \cite{VoZeMa2014FFTH} to general grids by carefully treating the Nyquist frequencies. To this purpose, the paper is organized as follows.

Section~\ref{sec:preliminaries} summarizes useful facts on periodic functions, the Fourier transform, and the Helmholtz decomposition, which also play a fundamental role in the continuous and discretized primal-dual formulations analyzed throughout the paper. 

Section~\ref{sec:homog_duality_bounds} provides a continuous formulation of the homogenization problem together with their main properties. Then the results by Dvořák \cite{Dvorak1993master,Dvorak1995RNM} are employed. In particular, an abstract duality result is formulated here in order to cover both continuous and discrete problems, complemented with the theory for accurate upper-lower bounds based on conforming approximations to the homogenization problem.

Section~\ref{sec:trig_main} deals with the spaces of trigonometric polynomials \cite{SaVa2000PIaPDE}, which are used to approximate the homogenization problem. 
Our exposition follows the developments presented in \cite{VoZeMa2014FFTH} for an odd number of grid points and extends it to the general case. 

Section~\ref{sec:gani} is dedicated to discrete formulations arising from the Galerkin approximation with numerical integration. Here, the emphasis is again on the extension of results in \cite{VoZeMa2014FFTH} to general grids such that conforming approximations are obtained. The relations between the primal-dual formulations are investigated using the duality arguments from Section~\ref{sec:duality}.

Section~\ref{sec:evaluation_bounds} contributes to methodology for the evaluation of the upper-lower bounds on homogenized properties; the details are provided for general matrix-inclusion composites.

Section~\ref{sec:computational_issues} gathers several computational aspects with an emphasis on effective implementation.

Section~\ref{sec:numerical_experiments} contains numerical examples that confirm the theoretical findings on the structure of the upper-lower bounds and differences between discretization using odd and even grids. Performance of the method is demonstrated with a real-world material described by a high-resolution image.

Section~\ref{sec:conclusion} summarizes the most interesting results, while \ref{sec:appendix_duality} 
concludes the paper by proving the abstract duality result from Section~\ref{sec:duality}.

Let us remark that throughout the paper, we attempt to make a systematic distinction among infinite-dimensional variables, their finite-dimensional approximations, and fully discrete~(matrix) representations. Although this approach leads to a somewhat  more involved notation, we have found it to be very helpful in understanding the theoretical basis of FFT-based homogenization algorithms as well as connections among the many variants of FFT-based algorithms available in the literature.

\section{Notation and preliminaries}
\label{sec:preliminaries}
In this section, we introduce our notation and recall some useful facts related to matrix analysis, Section~\ref{sec:vectors_matrices}, and to spaces of periodic functions and the Fourier transform, Section~\ref{sec:periodic_functions_and_fourier}, used throughout the paper. Section~\ref{sec:helmoholtz} is dedicated to the Helmholtz decomposition of vector-valued periodic functions and its description by orthogonal projections, which will be essential for the duality arguments in both discrete and continuous settings.

In general, number spaces are denoted with double-struck symbols, e.g. $\xN$, $\xZ$, $\xR$, or $\xC$, operators are denoted with calligraphic letters, e.g. $\mathcal{I}$, $\mathcal{Q}$, $\mathcal{P}$, or $\mathcal{G}$, and function spaces are denoted in the standard way, e.g.  $\Lper{2}{}$, $\Cper{0}{\xRd}$, or using a script font, e.g. $\mathscr{U}$, $\mathscr{E}$, $\cH$, or $\cT$.

\subsection{Vectors and matrices}
\label{sec:vectors_matrices}

In the sequel, $d$ is reserved for the dimension of the model problem, assuming $d=2,3$.
To keep the notation compact, $\xX$
abbreviates the space of scalars, vectors, or matrices, i.e. $\xR$, $\xRd$, or $\xRdd$, and $\hX$ is used for their complex counterparts, i.e.~$\xC$, $\xCd$, or $\xC^{d \times d}$. Vectors and matrices are denoted by boldface letters, e.g.
$\Vu,\Vv \in \xRd$ or $\mM \in \xRdd$, with Greek letters used when referring to
their entries, e.g. $\mM = (M_{\alp\beta})_{\alp,\beta=1,\ldots,\dime}$.
Matrix $\mI = \bigl(\del_{\alp\beta}\bigr)_{\alp\beta}$ denotes the identity matrix where the symbol $\del_{\alp\beta}$ is reserved for the Kronecker delta, defined as
$\delta_{\alp\beta} = 1$ for $\alp = \beta$ and $\delta_{\alp\beta} = 0$ otherwise. 

As usual, matrix-matrix product $\T{L}\mM$, matrix-vector product $\mM \Vu$, and outer product $\Vu \otimes \Vv$ refer to 
\begin{align*}
(\T{L}\mM)_{\alp\beta} &= \sum_\gamma L_{\alp\gamma}M_{\gamma\beta}
&
(\mM \Vu)_\alp & = \sum_\beta M_{\alp\beta} u_\beta,
&
(\Vu \otimes \Vv)_{\alp\beta}
&=
u_\alp v_\beta,
\end{align*}
where we assume that $\alp$ and $\beta$ range from $1$ to $d$ for the sake of brevity.
Moreover, we endow the spaces with the standard inner products and norms, e.g. 
\begin{align}\label{eq:basic_matrix_vector_operations}
\scal{\Vu}{\Vv}_{\xC^{\dime}}
&= \sum_\alp u_\alp
\overline{v_\alp},
&
\|\Vu\|_{\xCd}^2 & = \scal{\Vu}{\Vu}_{\xCd},
&
\|\mM\|_{\xCdd} &= \max_{\Vu\neq\T{0}} \frac{\|\mM\Vu\|_{\xCd}}{\|\Vu\|_{\xCd}}.
\end{align}

The set $\xRddspd\subset{\xRdd}$ denotes the space of symmetric positive definite matrices satisfying
\begin{align*}
 M_{\alp\beta} &= M_{\beta\alp}\quad\text{for all }\alp,\beta,
 &
 \scal{\mM\Vu}{\Vu}_{\xRd}&>0\quad\text{for all }\Vu\in\xRd\text{ such that }\Vu\neq\Vo.
\end{align*}
In this space, the trace operator,
$\tr{\mM} = \sum_\alp M_{\alp\alp}\text{ for }\mM\in\xRdd$,
becomes an equivalent norm to \eqref{eq:basic_matrix_vector_operations} as it equals to the sum of eigenvalues, cf. \cite[Section 5.6]{horn2012matrix}.
The L\"{o}wner partial order, cf.~\cite[Section~7.7]{horn2012matrix}, of symmetric positive definite matrices will be found useful, i.e. for  $\T{L},\T{M}\in\xRddspd$ we write
\begin{align*}
 \T{L}\preceq\T{M}\quad\text{if}\quad\scal{\T{L} \Vu}{\Vu}_{\xRd}\leq \scal{\T{M} \Vu}{\Vu}_{\xRd}\quad\text{for all }\Vu\in\xRd.
\end{align*}
We also systematically use the inverse inequality property
\begin{align}\label{eq:spd_inverse_ineq}
 \T{L} \preceq \T{M}
 \quad \Longleftrightarrow\quad
  \T{M}^{-1}\preceq \T{L}^{-1}
\end{align}
for  $\T{L},\T{M}\in\xRddspd$, cf. \cite[Corollary~7.7.4.(a)]{horn2012matrix}.

\subsection{Periodic functions and Fourier transform}
\label{sec:periodic_functions_and_fourier}
We consider cells in the form $\puc = \prod_\alp \bigl(-\frac{\lpuc_\alp}{2},
\frac{\lpuc_\alp}{2}\bigr)$ for $\V{\lpuc}\in\xRd$ such that $\lpuc_\alp>0$. Then, a function $\Vu : \xRd \rightarrow \xX$ is $\puc$-periodic if
\begin{align*}
\Vu(\Vx + \sum\nolimits_\alp \lpuc_\alp k_\alp)
= \Vu(\Vx)\quad
\text{for all }
\Vx \in \puc
\text{ and all }
\Vk \in \Zd.
\end{align*}
The space $C_{\per}(\puc;\xX)$ collects all continuous $\puc$-periodic functions $\xR^d\rightarrow\xX$.
For $p \in \{2,\infty\}$,
\begin{align*}
\Lper{p}{\xX}
=
\left\{
\Vu :\puc\rightarrow\xX
: 
\Vu \text{ is $\puc$-periodic, measurable, and }\|\Vu\|_{\Lper{p}{\xX}}<\infty
\right\}
\end{align*}
denotes the Lebesgue spaces equipped with the norm
\begin{align*}
\norm{\Vu}{\Lper{p}{\xX}}
=
\begin{cases}
\esssup_{\Vx\in\puc} \norm{\Vu(\Vx)}{\xX} & \text{for } p=\infty, 
\\
\left(\meas{\puc}^{-1}
\displaystyle
\int_\puc
\norm{\Vu(\Vx)}{\xX}^2 \D{\Vx}
\right)^{1/2} & \text{for }p=2.
\end{cases}
\end{align*}
where $\meas{\puc} = \prod_\alp \lpuc_\alp$ denotes the Lebesgue measure of the cell $\puc$.

For the sake of brevity, we write $\Lper{p}{}$ instead of $\Lper{p}{\xR}$, and often shorten $\Lper{2}{\xRd}$ to $\Ltp$ when referring to the norms and the inner product.

The Fourier transform of $\Vu \in \Lper{2}{\xX}$ is given by
\begin{align}\label{eq:FT_def}
\FT{\Vu}( \Vk )
=
\conj{\FT{\Vu}(-\Vk)}
=
\frac{1}{\meas{\puc}}
\int_\puc
\Vu(\Vx)
\bfun{-\Vk}(\Vx)
\D{\Vx}
\in \FT{\xX}\quad
\text{for }
\Vk \in \Zd,
\end{align}
where the Fourier trigonometric polynomials,
\begin{align*}
\bfun{\Vk}(\Vx)
=
\exp{
\Bigl(
  2\pi\imu
    \scal{\Tz(\Vk)}{\Vx}_{\xRd} 
\Bigr)}\quad
\text{for }
\Vx \in \puc,
\Vk \in \Zd,
\text{ and with }
\Tz(\Vk) = (k_\alp / \lpuc_\alp)_\alp,
\end{align*}
form an orthonormal basis $\{\bfun{\Vk}\}_{\Vk \in \ZN}$ of
$\Lper{2}{}$, i.e.
\begin{align}
\label{eq:Fourier_basis_functions_orthogonality}
\scal{\bfun{\Vk}}{\bfun{\T{m}}}_{\Lper{2}{}}
&=
\delta_{\Vk \Vm}\quad\text{for }\Vk,\Vm \in \Zd,
\end{align}
cf.~\cite[pp.~89--91]{rudin1986real}.
Thus, every function $\Vu \in \Lper{2}{\xX}$ can be expressed in the form
\begin{align*}
\Vu(\Vx) 
=
\sum_{\Vk \in \Zd}
\FT{\Vu}( \Vk )
\bfun{\Vk}(\Vx)
\quad\text{for }
\Vx \in \puc.
\end{align*}

\rev{}{The spaces $\Lper{2}{\xX}$ and $\Lper{2}{\FT{\xX}}$ are also Hilbert spaces endowed with the standard inner products, e.g.
\begin{align*}
\scal{\Vu}{\Vv}_{\Lper{2}{\FT{\xX}}} &=
\frac{1}{\meas{\puc}}\int_{\puc} \scal{\Vu(\Vx)}{\Vv(\Vx)}_{\FT{\xX}}\D{\Vx}
= \sum_{\Vk\in\Zd} \scal{\hat{\Vu}(\Vk)}{\hat{\Vv}(\Vk)}_{\FT{\xX}},
\end{align*}
which can be expressed, thanks to Parseval's theorem, in both real and Fourier domains.} 
The mean value of function $\Vu\in \Lper{2}{\xX}$ over periodic cell $\puc$ is denoted as 
\begin{align*}
\mean{\Vu} = \frac{1}{|\puc|}\int_{\puc}\Vu(\Vx)\D{\Vx} = \hat{\Vu}(\Vo) \in \xX
\end{align*}
and corresponds to the zero-frequency Fourier coefficient.

\subsection{Helmholtz decomposition for periodic functions}
\label{sec:helmoholtz}
Operator $\oplus$ denotes the direct sum of mutually orthogonal subspaces, e.g. $\xRd=\cb{1}\oplus\cb{2}\oplus\dotsc\oplus\cb{d}$ for vectors  $\cb{\alp} = (\del_{\alp\beta})_\beta$.
By the Helmholtz decomposition \cite[pages~6--7]{Jikov1994HDOIF}, $\Lper{2}{\xRd}$ admits an orthogonal decomposition
\begin{align}\label{eq:Helmholtz_L2}
\Lper{2}{\xRd}=\cU\oplus\cE\oplus\cJ
\end{align}
into the subspaces of constant, zero-mean curl-free, and zero-mean divergence free fields
\begin{subequations}\label{eq:subspaces_UEJ}
\begin{align}
\cU &= \{\Vv\in \Lper{2}{\xRd}:\Vv(\Vx) = \mean{\Vv}\text{ for all }\Vx\in\puc\},
\\
\label{eq:cE_def}
\cE &= \{\Vv\in \Lper{2}{\xRd}:\curl\Vv = \Vo,\mean{\Vv}=\Vo\},
\\
\cJ &= \{\Vv\in \Lper{2}{\xRd}:\div\Vv = 0,\mean{\Vv}=\Vo\}.
\end{align}
\end{subequations}
Here, the differential operators $\curl$ and $\div$
are understood in the Fourier sense, so that
\begin{align*}
(\curl \Vu)_{\alp\beta}
& = \sum_{\Vk\in\Zd} 2\pi\imu \bigl(\xi_\beta(\Vk)\hat{u}_\alp(\Vk) - \xi_\alp(\Vk)\hat{u}_\beta(\Vk)\bigr) \varphi_\Vk,
&
\div \Vu
& = \sum_{\Vk\in\Zd} 2\pi\imu \scal{\Vxi(\Vk)}{\hat{\Vu}(\Vk)}_{\xCd} \varphi_\Vk,
\end{align*}
cf.~\cite[pp.~2--3]{Jikov1994HDOIF} and \cite{SaVa2000PIaPDE}. Furthermore, the constant functions from $\cU$ are identified with vectors from $\xRd$.

Alternatively, the subspaces arising in the Helmholtz decomposition \eqref{eq:subspaces_UEJ} can be characterized by the orthogonal projections introduced next.
\begin{definition}\label{def:projection}
Let $\mathcal{G}\sub{\cU}$, $\mathcal{G}\sub{\cE}$, and $\mathcal{G}\sub{\cJ}$ denote operators $\Lper{2}{\xRd}\rightarrow \Lper{2}{\xRd}$ defined via
\begin{align*}
\mathcal{G}\sub{\bullet} [\Vv](\Vx) &= 
\sum_{\Vk\in\Zd} \mhG\sub{\bullet}(\Vk) \hat{\Vv}(\Vk) \varphi_{\Vk}(\Vx)\quad\text{for }\bullet\in\{\cU,\cE,\cJ\},
\end{align*}
where the matrices of Fourier coefficients $\mhG\sub{\bullet}(\Vk)\in\xR^{d\times d}$ read
\begin{align*}
\mhG\sub{\cU}(\Vk) &=
\begin{cases}
\mI
\\
\Vo\otimes\Vo
\end{cases}
&
\mhG\sub{\cE}(\Vk) &=
\begin{cases}
\Vo\otimes\Vo
\\
\frac{\Vxi(\Vk)\otimes\Vxi(\Vk)}{\scal{\Vxi(\Vk)}{\Vxi(\Vk)}_{\xRd}}
\end{cases}
&
\mhG\sub{\cJ} (\Vk) &=
\begin{cases}
\Vo\otimes\Vo,&\text{for }\Vk = \Vo
\\
\mI - \frac{\Vxi(\Vk)\otimes\Vxi(\Vk)}{\scal{\Vxi(\Vk)}{\Vxi(\Vk)}_{\xRd}}&\text{for }\Vk\in\ZdmO
\end{cases}.
\end{align*}
\end{definition}
\begin{lemma}\label{lem:Helmholtz_decomp}
The operators $\mathcal{G}\sub{\cU}$, $\mathcal{G}\sub{\cE}$, and $\mathcal{G}\sub{\cJ}$ are mutually orthogonal projections with respect to the inner product on $\Lper{2}{\xRd}$, on $\cU$,$\cE$, and $\cJ$.
\end{lemma}
\begin{proof}
In \cite[Lemma~3.2]{VoZeMa2014FFTH}, we show in detail that $\mathcal{G}\sub{\cE}$ is an orthogonal projection onto $\cE$. The remaining cases follow from the mutual orthogonality of $\mhG\sub{\bullet}(\Vk):\xCd\rightarrow\xCd$ for all $\Vk\in\xZ^d$ and with $\bullet\in\{\cU,\cE,\cJ\}$, cf. \cite[Section~12.1]{Milton2002TC}.
\end{proof}
\section{Homogenization, duality, and upper-lower bounds}
\label{sec:homog_duality_bounds}

In the present section, we define homogenized matrices via variational problems and collect several useful facts about their evaluation in the primal and the dual formulations. The connection between the matrices is established in Section \ref{sec:duality} using duality arguments, which immediately provide their basic properties along with the Voigt-Reuss bounds in Section~\ref{sec:comments_on_homogenization}. Section~\ref{sec:upper-lower_bounds_general} is dedicated to the determination of accurate upper-lower bounds based on conforming primal-dual minimizers, following the earlier developments by Dvo\v{r}\'{a}k~\cite{Dvorak1993master,Dvorak1995RNM}.

Here and in the sequel, matrix field $\TA:\puc\rightarrow\xRddspd$ is reserved for material coefficients, which are required to be essentially bounded, symmetric, and uniformly elliptic
\begin{align}
\label{eq:A}
\TA \in \Lper{\infty}{\xRddspd},
&&
\cA \norm{\Vv}{\xRd}^2
\leq 
\scal{\TA(\Vx)\Vv}{\Vv}_{\xRd}
\leq \CA \norm{\Vv}{\xRd}^2
\end{align}
a.e.~in $\puc$ for all $\Vv \in \xRd$ with  $0 < \cA \leq \CA < +\infty$;
by \eqref{eq:spd_inverse_ineq} the inverse coefficients satisfy
\begin{align*}
\TAi \in 
\Lper{\infty}{\xRddspd},
&&
\frac{1}{\CA} \norm{\Vv}{\xRd}^2
\leq 
\scal{\TAi(\Vx)\Vv}{\Vv}_{\xRd}
\leq \frac{1}{\cA} \norm{\Vv}{\xRd}^2
\end{align*}
a.e.~in $\puc$ for all $\Vv \in \xRd$.
We will also consider bilinear forms $\bilf{}{}:\Lper{2}{\xRd}\times\Lper{2}{\xRd}\rightarrow\xR$ and $\bilfi{}{}:\Lper{2}{\xRd}\times\Lper{2}{\xRd}\rightarrow\xR$ provided by
\begin{align}\label{eq:bilinear_forms}
\bilf{\Vu}{\Vv} &:= \scal{\TA\Vu}{\Vv}_{\Lper{2}{\xRd}},
&
\bilfi{\Vu}{\Vv} &:= \scal{\TA^{-1}\Vu}{\Vv}_{\Lper{2}{\xRd}}
\end{align}
together with energetic norms
\begin{align*}
\|\Vu\|_{\TA} &:= \sqrt{\bilf{\Vu}{\Vu}},
&
\|\Vu\|_{\TAi} &:= \sqrt{\bilfi{\Vu}{\Vu}}.
\end{align*}

\begin{definition}[Homogenized matrices]
\label{def:homog_problem}
Let the coefficient $\TA$ satisfy \eqref{eq:A}. Then the primal and dual homogenized matrices $\Aeff, \Beff \in \xR^{d\times d}$ are defined as
\begin{subequations}
\label{eq:homog_problem}
\begin{align}\label{eq:P_homog_form}
\scal{\Aeff \VE}{\VE}_{\xRd} &= \min_{\Ve\in\cE} \bilf{\VE+\Ve}{\VE+\Ve}=\bilf{\VE+\tVe\mac{\VE}}{\VE+\tVe\mac{\VE}},
\\
\label{eq:D_homog_form}
\scal{\Beff\VJ}{\VJ}_{\xRd} &= \min_{\Vj\in\cJ} \bilfi{\VJ+\Vj}{\VJ+\Vj}=\bilfi{\VJ+\tVj\mac{\VJ}}{\VJ+\tVj\mac{\VJ}}
\end{align}
\end{subequations}
for arbitrary $\VE, \VJ\in\xR^d$.
\end{definition}
\begin{remark}The minimizers $\tVe\mac{\VE}$ and $\tVj\mac{\VJ}$, thanks to the Lax-Milgram lemma, exist, are unique for any $\VE,\VJ\in\xRd$, and satisfy the optimality conditions
\begin{align*}
 \bilf{\tVe\mac{\VE}}{\Vv} &= - \bilf{\VE}{\Vv}\quad\forall\Vv\in\cE,
 &
 \bilfi{\tVj\mac{\VJ}}{\Vv} &= - \bilfi{\VJ}{\Vv}\quad\forall\Vv\in\cJ.
\end{align*}
\end{remark}
\begin{remark}
Notice that the primal formulation \eqref{eq:P_homog_form} coincides with problem \eqref{eq:micro} introduced in Section~\ref{sec:introduction}, because the subspace $\cE$ from~\eqref{eq:cE_def} admits an equivalent characterization
$\cE = \{ \nabla f:f\in H^1_{\per,\zmean}(\puc) \},$
cf.~\cite[pp.~6--7]{Jikov1994HDOIF}.
\end{remark}

\subsection{Duality}
\label{sec:duality}
In this section, the homogenized matrices and their formulations \eqref{eq:homog_problem} are connected by standard duality arguments. These ideas are summarized into a proposition that is applicable to both the continuous homogenization problem \eqref{eq:homog_problem} and also to its discrete relatives \eqref{eq:FD_GaNi} and \eqref{eq:FD_GaNi_tilde}. In \ref{sec:appendix_duality}, in order to keep the exposition self-contained, we also provide its proof.
\begin{proposition}[Transformation to dual formulation]
\label{lem:transform2dual}
Let $\cH$ be a Hilbert space with a nontrivial orthogonal
decomposition $\cH = \mathring{\cU} \oplus \mathring{\cE} \oplus \mathring{\cJ}$, where $\mathring{\cU}$ is isometrically isomorphic to $\xRd$. Next, let bilinear forms $\bilfG{}{}:\cH\times\cH\rightarrow\xR$ and $\bilfGi{}{}:\cH\times\cH\rightarrow\xR$ be defined as
\begin{align*}
\bilfG{\Vu}{\Vv} &=
\scal{\rTA\Vu}{\Vv}_{\cH},
&
\bilfGi{\Vu}{\Vv} &= \scal{\rTA^{-1}\Vu}{\Vv}_{\cH}
\end{align*}
for symmetric, coercive, and bounded linear operator $\rTA:\cH\rightarrow\cH$, so that there exist $c_{\rTA}>0$ and $C_{\rTA}>0$ such that
\begin{align*}
c_{\rTA} \|\Vu\|_{\cH} \leq \scal{\rTA\Vu}{\Vu}_{\cH} \leq C_{\rTA} \|\Vu\|_{\cH}.
\end{align*}
Then matrices $\rTA_\eff,\rTB_{\eff}\in\xRdd$ defined as
\begin{subequations}
\label{eq:general_homog}
\begin{align}
\label{eq:general_homog_primal}
\scal{\rTA_\eff\VE}{\VE}_{\xRd} &= \min_{\rVe\in\mathring{\cE}} \bilfG{\VE+\rVe}{\VE+\rVe} = \bilfG{\VE+\rVe\mac{\VE}}{\VE+\rVe\mac{\VE}}
\\
\label{eq:general_homog_dual}
\scal{\rTB_{\eff}\VJ}{\VJ}_{\xRd} &= \min_{\rVj\in\mathring{\cJ}} \bilfGi{\VJ+\rVj}{\VJ+\rVj} = \bilfGi{\VJ+\rVj\mac{\VJ}}{\VJ+\rVj\mac{\VJ}}
\end{align}
\end{subequations}
for arbitrary $\VE,\VJ\in\xRd$ satisfy
\begin{align}\label{eq:duality_mutually_inverse_property}
\rTA_\eff = \rTB_{\eff}^{-1}.
\end{align}
Moreover, the minimizers $\rVe\mac{\VE}$ and $\rVj\mac{\VJ}$ of both formulations \eqref{eq:general_homog} are connected via
\begin{align}\label{eq:connection_primal_dual_fields}
\VJ+\rVj\mac{\VJ} &= \rTA[\VE+\rVe^{(\VE)}]
\quad\text{for }
\VJ = \mathring{\TA}_{\eff}\VE\text{ and }\VE\in\xRd.
\end{align}
\end{proposition}
\begin{remark}
The decomposition $\cH = \mathring{\cU} \oplus \mathring{\cE} \oplus \mathring{\cJ}$ fits either to the standard Helmholtz framework \eqref{eq:Helmholtz_L2} or to its fully discrete variants \eqref{eq:FD_Helmholtz_decomposition} and \eqref{eq:FD_Helmholtz_decomposition_E}.
Note that, to be defined properly, the bilinear forms \eqref{eq:general_homog} for $\VE,\VJ\in\xRd$ are understood as
\begin{align*}
 \bilfG{\VE+\rVe}{\VE+\rVe} &:= \bilfG{\mathcal{I}^{-1}[\VE]+\rVe}{\mathcal{I}^{-1}[\VE]+\rVe},
&
\bilfGi{\VJ+\rVj}{\VJ+\rVj} &:= \bilfGi{\mathcal{I}^{-1}[\VJ]+\rVj}{\mathcal{I}^{-1}[\VJ]+\rVj},
\end{align*}
with the help of the isometric isomorphism $\mathcal{I}:\rcU\rightarrow\xRd$, which is natural for spaces $\xRd$ and $\cU$, see also Remark~\ref{rem:identify_xRd_cUN} later in this paper.
\end{remark}
Properties of primal and dual homogenization problems \eqref{eq:homog_problem} now follow as a corollary to Proposition~\ref{lem:transform2dual}.
\begin{corollary}
\label{eq:continuous_transform2dual}
The homogenized matrices in \eqref{eq:P_homog_form} and \eqref{eq:D_homog_form} are mutually inverse
\begin{align*}
\Aeff=\Beff^{-1}.
\end{align*}
Moreover, the minimizers are connected by
\begin{align}\label{eq:continuous_conn_prim2dual}
 \VJ+\tVj\mac{\VJ} = \TA(\VE+\tVe\mac{\VE})\quad\text{for }\VJ=\Aeff\VE\text{ and }\VE\in\xRd.
\end{align}
\end{corollary}
\begin{proof}
The proof is a direct consequence of Proposition~\ref{lem:transform2dual} for
\begin{align*}
\begin{array}{ccccccc}
\cH & = &\rcU & \oplus & \rcE & \oplus & \rcJ \\
\rotatebox{90}{=} &  & \rotatebox{90}{=} &  & \rotatebox{90}{=} &  & \rotatebox{90}{=} \\
\Lper{2}{\xRd} & = &\cU & \oplus & \cE & \oplus & \cJ \\
\end{array}
\end{align*}
and
\begin{align*}
\bilfG{}{} &= \bilf{}{},
&
\bilfGi{}{} &= \bilfi{}{},
&
\mathring{\TA}_\eff &= \Aeff,
&
\mathring{\TB}_\eff &= \Beff,
&
\rVe\mac{\VE} &= \tVe\mac{\VE},
&
\rVj\mac{\VJ} &= \tVj\mac{\VJ}.
\end{align*}
\end{proof}
\subsection{Comments on the homogenized properties and their calculation}
\label{sec:comments_on_homogenization}
\begin{remark}\label{rem:Aeff_spd}
The homogenized matrix $\Aeff\in\xRdd$ is symmetric positive definite and thus regular, as follows from standard arguments in homogenization theory, e.g. \cite{Bensoussan1978per_structures,Jikov1994HDOIF,Cioranescu1999Intro2Homog}. Indeed, thanks to the coercivity of coefficients \eqref{eq:A}, the quadratic form in \eqref{eq:P_homog_form} is nonnegative and equals to zero only for $\Ve$ such that $(\VE+\Ve)\equiv 0$, which is impossible because the space $\rcJ$ does not contain constant fields. This implies the positive definiteness of matrix $\Aeff$, while its symmetry is inherited from the symmetry of  coefficients \eqref{eq:A} and consequently of the bilinear form $\bilf{}{}$, cf. \eqref{eq:Aeff_coefficients}.
In addition, the homogenized matrix satisfy
Voigt \cite{voigt1910lehrbuch} and Reuss \cite{Reuss1929} bounds 
\begin{align*}
\mean{\TAi}^{-1} \preceq \Beff^{-1} = \Aeff  &\preceq \mean{\TA}
\end{align*}
obtained from the equivalence \eqref{eq:spd_inverse_ineq} and the formulations in \eqref{eq:homog_problem} tested with $\Ve = \Vj = \Vo$. The lower bound also provides another proof of the positive definiteness of homogenized matrix $\Aeff$.
\end{remark}
Some additional notation is needed to analyze the homogenization problem \eqref{eq:homog_problem} in more detail. By linearity, the solutions to~\eqref{eq:homog_problem} can be fully characterized by solutions to $\dime$ auxiliary problems, obtained by successively setting $\VE$ and $\VJ$ equal to the basis vectors of $\xRd$.
\begin{definition}[Auxiliary problems]
\label{def:auxiliary_problems}
The auxiliary minimizers $\tVe^{(\alp)}\in\cE$ and $\tVj^{(\alp)}\in\cJ$ satisfy
\begin{subequations}
\label{eq:weak_forms}
\begin{align}
\label{eq:weak_form_primal}
\bilf{\tVe^{(\alp)}}{\Vv} &= - \bilf{\cb{\alp}}{\Vv}\quad\forall\Vv\in \cE,
\\
\label{eq:weak_form_dual}
\bilfi{\tVj^{(\alp)}}{\Vv} &= - \bilfi{\cb{\alp}}{\Vv}\quad\forall\Vv\in\cJ
\end{align}
\end{subequations}
with $\cb{\alp}=(\del_{\alp\beta})_\beta\in\xRd$.
\end{definition}
Now, the minimizers $\tVe\mac{\VE}\in\cE$ and $\tVj\mac{\VJ}\in\cJ$ for $\VE,\VJ\in\xRd$, recall Definition~\ref{def:homog_problem}, can be obtained from the auxiliary minimizers by linear superposition
\begin{align*}
\tVe\mac{\VE} &= \sum_{\alp} E_\alp\tVe\mac{\alp},
&
\tVj\mac{\VJ} &= \sum_{\alp} J_\alp\tVj\mac{\alp},
\end{align*}
and the components of the homogenized matrix can be expressed as
\begin{align}
\label{eq:Aeff_coefficients}
A_{\eff,\alp\beta} &= \bilf{\cb{\beta}+\tVe\mac{\beta}}{\cb{\alp}+\tVe\mac{\alp}},
&
B_{\eff,\alp\beta} &= \bilfi{\cb{\beta}+\tVj\mac{\beta}}{\cb{\alp}+\tVj\mac{\alp}}.
\end{align}
Using \eqref{eq:continuous_conn_prim2dual}, the dual auxiliary minimizer $\tVj\mac{\alp}$ can be expressed as a linear combination of primal ones $\Ve^{(\alp)}$, thus
\begin{align*}
\cb{\alp}+\tVj\mac{\alp} = \TA\sum_{\alp} E_\alp(\cb{\alp}+\tVe\mac{\alp})\quad
\text{where }\VE = \Aeff^{-1} \cb{\alp}.
\end{align*}

\subsection{Upper-lower bounds on the homogenized properties}
\label{sec:upper-lower_bounds_general}
Following Dvo\v{r}\'{a}k \cite{Dvorak1993master,Dvorak1995RNM}, the aim of the present section is to obtain guaranteed bounds on the homogenized matrix $\Aeff$ by utilizing a suitable conforming approximations
\begin{align}\label{eq:conf_minimizers}
\tVe^{(\alp)}_\dst\in\cE
\quad\text{and}\quad
\tVj^{(\alp)}_\dst\in\cJ,
\end{align}
as test fields in \eqref{eq:homog_problem}. Here, $\dst$ represents a discretization parameter related to the maximum element size for FEM or grid spacing for FFT-based methods.
\begin{definition}[Upper-lower bounds on homogenized matrix, \cite{Dvorak1993master}]\label{def:upper_lower_bounds}
Matrices $\oAeffdst,\oBeffdst\in\xRdd$ defined as
\begin{align}
\label{eq:upper_lower_bounds}
\ol{A}_{\eff,\dst,\alp\beta} &=
\bilf{\cb{\beta}+\tVe^{(\beta)}_{\dst} }{\cb{\alp}+ \tVe^{(\alp)}_{\dst} },
&
\ol{B}_{\eff,\dst,\alp\beta} &=
\bilfi{\cb{\beta}+\tVj^{(\beta)}_{\dst}}{\cb{\alp}+\tVj^{(\alp)}_{\dst} }
\end{align}
are guaranteed upper-lower bounds on the homogenized matrix $\Aeff$. The mean of guaranteed bounds with a guaranteed error stands for 
\begin{align}\label{eq:approx_homog_with_error}
 \oul{\TA}_{\eff,\dst} &= \frac{1}{2}(\oAeffdst + \oBeffdst^{-1}),
 &
 \mD_{\dst} &= \frac{1}{2}(\oAeffdst - \oBeffdst^{-1}).
\end{align}
\end{definition}
The correctness of this definition is demonstrated with the following lemma.
\begin{lemma}\label{lem:bounds}
The matrices from Definition \ref{def:upper_lower_bounds} are symmetric positive definite and satisfy the upper-lower bounds structure
\begin{align}\label{eq:Aeff_ineq}
\Aeff &\preceq \oAeffdst,
&
\Beff &\preceq \oBeffdst,
&
\oBeffdst^{-1} &\preceq \Beff^{-1} = \Aeff \preceq \oAeffdst.
\end{align}
Moreover, the previous bounds imply the element-wise bounds for diagonal components
\begin{align}\label{eq:estimate_diagonal}
\bigl(\ol{B}_{\eff,\dst}^{-1}\bigr)_{\alp\alp} &\leq \bigl(B_{\eff}^{-1}\bigr)_{\alp\alp} = \bigl(A_{\eff}\bigr)_{\alp\alp} \leq \bigl(\ol{A}_{\eff,\dst}\bigr)_{\alp\alp},
\end{align}
and for non-diagonal components, i.e. for $\alp\neq\beta$
\begin{align}\label{eq:estimate_nondiagonal}
\oul{A}_{\eff,\dst,\alp\beta} - D_{\dst,\alp\alp} - D_{\dst,\beta\beta} \leq A_{\eff,\alp\beta} \leq 
\oul{A}_{\eff,\dst,\alp\beta} + D_{\dst,\alp\alp} + D_{\dst,\beta\beta}.
\end{align}
\end{lemma}
\begin{proof}
The first two inequalities in \eqref{eq:Aeff_ineq} are the consequence of minimality properties of primal and dual homogenized matrices $\Aeff$ and $\Beff$ according to Definition~\ref{def:homog_problem}, tested with conforming approximations \eqref{eq:conf_minimizers}, i.e. $\Ve=\Ve_h\mac{\VE}\in\cE$ and $\Vj=\Vj_h\mac{\VJ}\in\cJ$.
The last inequality in \eqref{eq:Aeff_ineq} is a consequence of property \eqref{eq:spd_inverse_ineq}.

The symmetry of the upper-lower bounds $\TA_{\eff,\dst},\TB_{\eff,\dst}$ follows from the symmetry of bilinear forms in \eqref{eq:upper_lower_bounds}, and the positive definiteness is shown by \eqref{eq:Aeff_ineq} once recalling that $\Aeff\in\xRddspd$.

The estimate of the diagonal terms \eqref{eq:estimate_diagonal} results from the inequality \eqref{eq:Aeff_ineq} tested with $\cb{\alp}$. 
For the non-diagonal terms, we have
\begin{align*}
2A_{\eff,\alp\beta} &= \scal{\Aeff(\cb{\alp}+\cb{\beta})}{\cb{\alp}+\cb{\beta}}_{\xRd}  - A_{\eff,\alp\alp} - A_{\eff,\beta\beta}.
\end{align*}
The first inequality in \eqref{eq:Aeff_ineq} tested with $\cb{\alp}+\cb{\beta}$ provides
\begin{align*}
\scal{\Aeff(\cb{\alp}+\cb{\beta})}{\cb{\alp}+\cb{\beta}}_{\xRd} 
\leq \scal{\oAeffdst(\cb{\alp}+\cb{\beta})}{(\cb{\alp}+\cb{\beta})}_{\xRd},
\end{align*}
Utilizing the inequalities for diagonal components \eqref{eq:estimate_diagonal}, we obtain the upper estimate in
\eqref{eq:estimate_nondiagonal}. The lower bound follows by analogous arguments.
\end{proof}
Now, we establish the relations among auxiliary minimizers \eqref{eq:conf_minimizers}, homogenized matrices \eqref{eq:upper_lower_bounds}, and guaranteed  error \eqref{eq:approx_homog_with_error}.
\begin{lemma}[Estimates]
\label{lem:estimates}
The following relations hold
\begin{align}
\label{eq:energetic_estimates}
\|\tVe^{(\alp)}-\tVe^{(\alp)}_{\dst}\|^2_{\TA} &= \ol{A}_{\eff,\dst,\alp\alp} - A_{\eff,\alp\alp},
&
\|\tVj^{(\alp)}-\tVj^{(\alp)}_{\dst}\|^2_{\TA^{-1}} &= \ol{B}_{\eff,\dst,\alp\alp} - B_{\eff,\alp\alp}
\end{align}
and
\begin{align}
 2\tr \mD_{\dst} &\leq \sum_\alp \|\tVe^{(\alp)}-\tVe^{(\alp)}_{\dst}\|^2_{\TA} + (\tr \Aeff)^2 \|\tVj^{(\alp)}-\tVj^{(\alp)}_{\dst}\|^2_{\TA^{-1}}
\label{eq:estimate_guaranteed_error}
 \\
  & \leq \|\TA\|_{\Lper{\infty}{\xRdd}} \sum_{\alp}   \|\tVe^{(\alp)}-\tVe^{(\alp)}_{\dst}\|^2_{\Ltp} 
+ (\tr \Aeff)^2
\|\TA^{-1}\|_{\Lper{\infty}{\xRdd}}  \sum_{\alp} \|\tVj^{(\alp)}-\tVj^{(\alp)}_{\dst}\|^2_{\Ltp}.
\nonumber
\end{align}
\end{lemma}
\begin{proof}
The proof of the estimates \eqref{eq:energetic_estimates} is shown only for the primal formulation, the dual case proceeds by analogy. Denoting
$\breve{\Ve}^{(\alp)} := (\cb{\alp}+\tVe^{(\alp)})$ and $\breve{\Ve}^{(\alp)}_h := (\cb{\alp}+\tVe^{(\alp)}_h)$, we obtain
\begin{align*}
\|\Ve^{(\alp)}-\Ve^{(\alp)}_{\dst}\|^2_{\TA} &= \bilf{ \breve{\Ve}^{(\alp)}-\breve{\Ve}^{(\alp)}_{\dst} }{\breve{\Ve}^{(\alp)}-\breve{\Ve}^{(\alp)}_{\dst} }
= \bilf{\breve{\Ve}^{(\alp)} }{\breve{\Ve}^{(\alp)} } - 2\bilf{ \breve{\Ve}^{(\alp)} }{\breve{\Ve}^{(\alp)}_{\dst} } + \bilf{\breve{\Ve}^{(\alp)}_{\dst} }{\breve{\Ve}^{(\alp)}_{\dst} }
\\
&= \bilf{\breve{\Ve}^{(\alp)} }{\breve{\Ve}^{(\alp)} } - 2\bilf{ \breve{\Ve}^{(\alp)} }{\breve{\Ve}^{(\alp)} } + \bilf{\breve{\Ve}^{(\alp)}_{\dst} }{\breve{\Ve}^{(\alp)}_{\dst} }
= \ol{A}_{\eff,\dst,\alp\alp} - A_{\eff,\alp\alp},
\end{align*}
where we have incorporated the Galerkin orthogonality of auxiliary problem \eqref{eq:weak_form_primal} tested with $\tVe_h\mac{\alp}$ and $\tVe\mac{\alp}$, from which it follows
\begin{align*}
\bilf{\cb{\alp}+\tVe^{(\alp)} }{ \cb{\alp}+\tVe_h^{(\alp)} } = \bilf{ \cb{\alp}+\tVe^{(\alp)} }{\cb{\alp} } = \bilf{ \cb{\alp}+\tVe^{(\alp)}}{\cb{\alp}+\tVe^{(\alp)}}.
\end{align*}
The estimate for the guaranteed error \eqref{eq:estimate_guaranteed_error} utilizes the fact that
\begin{align*} 
0 \leq \tr(\mD-\mC) &\leq \tr [\mD(\mC^{-1}-\mD^{-1})\mC] \leq \tr \mD \tr(\mC^{-1}-\mD^{-1}) \tr\mC 
\\
&\leq(\tr \mD)^2\tr(\mC^{-1}-\mD^{-1})
\end{align*}
holding for $\mC,\mD\in\xR^{d\times d}_{\mathrm{spd}}$ such that $\mC\preceq \mD$.
This inequality and \eqref{eq:energetic_estimates} enable us to calculate
\begin{align*}
2\tr \mD_\dst &=
 \tr(\oAeffdst - \Aeff) + \tr(\Aeff -\oBeffdst^{-1})
\leq \tr(\oAeffdst - \Aeff) + (\tr\Aeff)^2\tr(\TB_{\eff} - \oBeffdst)
\\
&\leq  \sum_{\alp}   \|\tVe^{(\alp)}-\tVe^{(\alp)}_{\dst}\|^2_{\TA} 
+ (\tr \Aeff)^2
 \sum_{\alp} \|\tVj^{(\alp)}-\tVj^{(\alp)}_{\dst}\|^2_{\TA^{-1}},
\end{align*}
and the proof is completed with the H\"{o}lder inequality.
\end{proof}

\section{Trigonometric polynomials and their fully discrete counterparts}
\label{sec:trig_main}
This section provides an introduction to discretization of the homogenization problem \eqref{eq:homog_problem} using trigonometric polynomials defined on a regular grid with $\VN\in\xNd$ points, with $N_\alp$ points along each Cartesian axis.
Suitability of such approximations has been demonstrated in \cite{VoZeMa2014FFTH}, following the general framework of Saranen and Vainikko \cite{SaVa2000PIaPDE}, but only for the odd number of grid points
\begin{align}
\label{eq:N_odd}
\VN\in\xN^d \text{ and } N_\alp \text{ is odd for all }\alp.
\end{align}
This assumption is often referred to as \emph{odd grid}; \emph{non-odd} or \emph{even grids} are used accordingly.
Obviously, \eqref{eq:N_odd} is restrictive from the applications point of view, so in this section we extend our earlier results from \cite{VoZeMa2014FFTH} to the general case.
Note the difficulty in working with non-odd number of grid points was identified and partially solved in \cite[Section 2.4.2]{Moulinec1998NMC} by heuristic arguments. Here, we refine this result in a way to preserve the structure of upper-lower bounds on the homogenized matrix established in Section~\ref{sec:upper-lower_bounds_general}.

This section begins with a brief notation part in Section~\ref{sec:trig_notation} complemented with the basic properties of trigonometric polynomials in Section~\ref{sec:trig_defs}.
The fully discrete representation of trigonometric polynomials is introduced in Section~\ref{sec:FD_spaces_odd} and \ref{sec:FD_spaces_non_odd} for odd and general number of grid points, respectively.

\subsection{Notation}
\label{sec:trig_notation}
A multi-index notation is systematically employed, in which $\xX^{\VN}$ represents $\xX^{N_1 \times \cdots \times N_\dime}$ for $\VN\in\set{N}^d$.
Then the sets $\xXN$ and $\xMN$, or their complex counterparts $\xhXN$ and $\xhMN$, represent the spaces of vectors and matrices, e.g. 
$\MBv = \bigl(\M{v}_{\alp}^\Vk\bigr)_{\alp}^{\Vk\in\ZNd} \in \xXN$ and $\MB{M} = \bigl(\M{M}_{\alp\bet}^{\Vk\Vm}\bigr)_{\alp,\bet}^{\Vk,\Vm\in\ZNd} \in \xMN$
with an index set $\ZNd\subset \set{Z}^\dime$ introduced subsequently in~\eqref{eq:index_sets}.
The objects of these discrete spaces are indicated by bold serif font, e.g. $\MBu$ and $\MB{M}$, in order to distinguish them from scalars $u_\alp\in\xR$ for $\alp=1,\dotsc,d$, vectors $\Vu\in\xRd$, scalar-valued functions $\Vv\in\Lper{2}{}$, or vector-valued functions $\Vw\in\Lper{2}{\xRd}$.
    
Sub-vectors and sub-matrices are designated by superscripts, e.g. $\MBv^\Vk =
\bigl(\M{v}_{\alp}^\Vk\bigr)_{\alp} \in \xRd$ or $\MB{M}^{\Vk\Vm} =
\bigl(\M{M}_{\alp\bet}^{\Vk\Vm}\bigr)_{\alp,\bet} \in \xRdd$. The inner products on $\xXN$ and $\xhXN$
are defined as
\begin{align*}
\scal{\MB{u}}{\MBv}_{\xXN} 
&= 
\frac{1}{\meas{\VN}}
\sum_{\Vk\in\ZNd}
\scal{\MBu^{\Vk}}{\MBv^{\Vk}}_{\xRd},
&
\scal{\MB{u}}{\MBv}_{\xhXN} 
&= 
\sum_{\Vk\in\ZNd} 
\scal{\MBu^{\Vk}}{\MBv^{\Vk}}_{\xCd},
\end{align*}
where $|\VN |
= \prod_\alp N_\alp$ stand for the number of grid points.

Moreover, the matrix-vector or matrix-matrix multiplications follow from
\begin{align*}
(\MB{M}\MBv)^{\Vk} 
= 
\sum_{\Vm\in\ZNd}
\MB{M}^{\Vk\Vm}
\MBv^{\Vm}
\in \xRd
\quad\text{or}\quad
(\MB{M} \MB{L})^{\Vk\Vm} 
=
\sum_{\Vl\in\ZNd}
\MB{M}^{\Vk\Vl}\MB{L}^{\Vl\Vm}
\in \xRdd
\end{align*}
for $\Vk,\Vm \in \ZNd$ and $\MB{L} \in \xMN$.
The identity operator on $\xXN$ corresponds to a matrix
\begin{align*}
\MB{I} = \bigl(\del_{\alp\beta}\del_{\Vk\Vm}\bigr)_{\alp,\beta}^{\Vk,\Vm\in\ZNd}\in\xMN
\end{align*}
and a matrix $\MBA\in\xMN$ is symmetric positive definite if 
\begin{align*}
\scal{\MBA\MBu}{\MBv}_{\xXN} &= \scal{\MBu}{\MBA\MBv}_{\xXN},
&
\scal{\MBA\MBv}{\MBv}_{\xXN}>0
\end{align*}
holds for all $\MBu,\MBv\in\xXN$ such that $\MBv\neq\MB{0}$.
\subsection{Trigonometric polynomials}
\label{sec:trig_defs}
This section extends the results from \cite[Section~4.1]{VoZeMa2014FFTH} for vector-valued trigonometric polynomials defined on grids with an odd number of points \eqref{eq:N_odd} to the general case. In order to facilitate the introduction of the fully discrete spaces in Sections~\ref{sec:FD_spaces_odd} and~\ref{sec:FD_spaces_non_odd}, we also review the simplifications arising from the odd grid assumption \eqref{eq:N_odd}.

\begin{definition}[Trigonometric polynomials]\label{def:trigonometric_pol_E}
For $\VN\in\xN^d$,  approximation and interpolation spaces of $\xRd$-valued trigonometric polynomials are defined by
\begin{subequations}
\label{eq:trig_spaces}
\begin{align}
\label{eq:trig_space_Fourier}
\cTNd &= 
\Bigl\{\sum_{\Vk\in\ZNdr}{\hat{\MB{v}}^{\Vk}\varphi_{ \Vk }} : \hat{\MB{v}}^{\Vk} = \overline{(\hat{\MB{v}}^{-\Vk})} \in\xCd \Bigr\},
\\
\label{eq:trig_space_shape}
\cTNtd &= 
\Bigl\{\sum_{\Vk\in\ZNd}{\MB{v}^{\Vk}\varphi_{\VN,\Vk }} : \MB{v}^{\Vk}\in\xRd\Bigr\},
\end{align}
\end{subequations}
where a reduced and a full index sets stand for
\begin{align}
\label{eq:index_sets}
\ZNdr &= \biggl\{ \Vk \in \set{Z}^d : 
    -\frac{N_\alpha}{2} < k_\alp < \frac{N_\alpha}{2}\biggr\},
&
\ZNd  &= 
  \biggl\{ \Vk \in \set{Z}^d : 
    -\frac{N_\alpha}{2} \leq k_\alpha < \frac{N_\alpha}{2} \biggr\},
\end{align}
and the spaces $\cTNd$ and $\cTNtd$ are spanned by the Fourier and fundamental trigonometric polynomials, respectively:
\begin{subequations}
\label{eq:trig_basis}
\begin{align}
\label{eq:trig_basis_Fourier}
\varphi_{\Vk}(\Vx) &= \exp\left(2\pi\imu\sum_{\alp}\frac{k_\alp x_\alp}{Y_\alp}\right),
\\
\label{eq:trig_basis_shape}
\varphi_{\VN,\Vk}( \Vx )
&=
\frac{1}{\pVN}
\sum_{\Vm \in \ZNd}
\omega_{\VN}^{-\Vk\Vm} \varphi_{\Vm}(\Vx),
\end{align}
\end{subequations}
with the coefficients
\begin{align*}
\omega_{\VN}^{\Vk\Vm} =
\exp   \left(2 \pi \imu\sum_{\alp} \frac{k_\alp m_\alp}{N_\alp}  \right)\quad \text{for }\Vk,\Vm\in\xZ^d. 
\end{align*}
\end{definition}

\rev{}{To facilitate the following discussion, in the next lemma we recall some useful connection among both types of trigonometric basis polynomials and the Discrete Fourier Transform.}

\rev{}{\begin{lemma}
For $\Vk,\Vm\in\ZNd$, it holds
\begin{subequations}
\label{eq:properties_phi_E}
\begin{align}
\label{eq:DFT_orthogonality}
\sum_{\Vn\in\ZNd}\omega_{\VN}^{-\Vk\Vn}\omega_{\VN}^{\Vn\Vm} &= \meas{\VN}\del_{\Vk\Vm},
\\
\varphi_{\Vk}(\Vx_\VN^{\Vm}) &= \omega_\VN^{\Vk\Vm},
\label{eq:prop_phi_1_E}
\\
\varphi_{\VN,\Vk}(\Vx_\VN^{\Vm}) &= \del_{\Vk\Vm},
\label{eq:prop_phi_2_E}
\\
\scal{\varphi_{\VN,\Vk}}{\varphi_{\VN,\Vm}}_{\Lper{2}{\rev{}{\xCd}}} &= \frac{\del_{\Vk\Vm}}{\pVN}.
\label{eq:prop_phi_3_E}
\end{align}
\end{subequations}
\end{lemma}
\begin{proof}The identity \eqref{eq:DFT_orthogonality} represents the orthogonality of the Discrete Fourier Transform coefficients, e.g.~\cite{Burrus1985DFT/FFT} and \eqref{eq:prop_phi_1_E} follows by direct calculations. The equalities  \eqref{eq:prop_phi_2_E} and \eqref{eq:prop_phi_3_E} are based on the orthogonality of DFT \eqref{eq:DFT_orthogonality}, the latter  additionally employs the orthogonality of the Fourier trigonometric polynomials~\eqref{eq:Fourier_basis_functions_orthogonality}.
\end{proof}}
The remainder of this section is devoted to clarifying the connection between the two definitions of trigonometric polynomials~\eqref{eq:trig_spaces}, index sets~\eqref{eq:index_sets}, and basis functions~\eqref{eq:trig_basis}.

The \emph{approximation space} $\cTNd$ provides a finite-dimensional subspace to $\Lper{2}{\xRd}$ for the Galerkin method. Its conformity, i.e. $\cTNd \subset \Lper{2}{\xRd}$, is ensured once the Hermitian symmetry of the Fourier coefficients holds, compare~\eqref{eq:trig_space_Fourier} with~\eqref{eq:FT_def}. This condition is easily enforced for odd grids which are symmetric with respect to the origin, Figure~\ref{fig:grid}. For non-odd grids the highest~(Nyquist) frequencies $k_\alp = -N_\alp/2$ must be omitted, leading to the notion of the reduced index set $\ZNdr$.

\begin{figure}[htp]
\centering
\subfigure[\scriptsize Grid points]{
\includegraphics[scale=0.6]{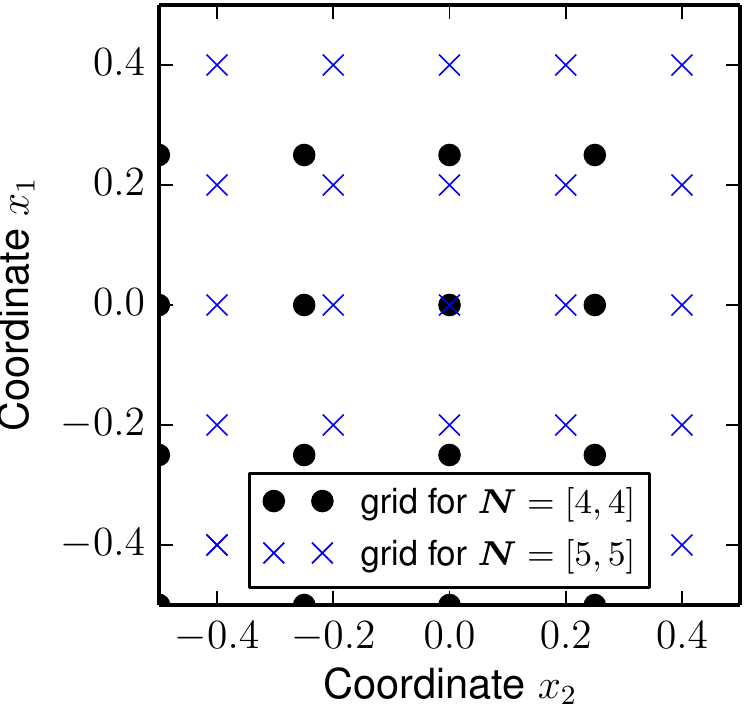}
\label{fig:grid}
}
\subfigure[\scriptsize Trigonometric polynomials]{
\includegraphics[scale=0.6]{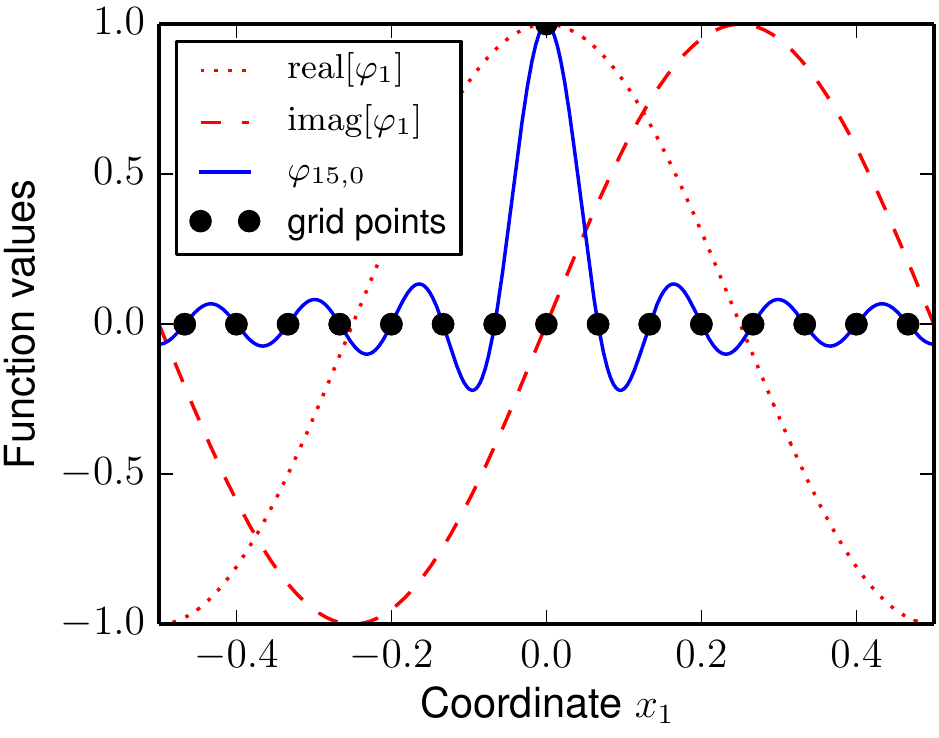}
\label{fig:trig_basis}
}
\caption{\rev{}{Examples of (a)~two-dimensional odd and even grid points \eqref{eq:grid_points} and (b)~one-dimensional complex-valued Fourier $\varphi_{1}$ and real-valued fundamental $\varphi_{15,0}$ trigonometric polynomials~\eqref{eq:trig_basis}}.}
\label{fig:trig_and_grid}
\end{figure}

The \emph{interpolation space} $\cTNtd$ will be used to perform the numerical quadrature in the Galerkin method and primarily works with data in the real instead of the Fourier domain. Its connection to the approximation space is established with the Discrete Fourier Transform~(DFT) and its inverse (iDFT)
\begin{align*}
\hat{\Vu}_{\VN}(\Vk)
=
\frac{1}{\meas{\VN}}
\sum_{\Vm\in\ZNd}\omega_\VN^{-\Vk\Vm}\Vu_{\VN}(\Vx_\VN^{\Vm}), \quad
\Vu_\VN(\Vx_\VN^{\Vk})
=
\sum_{\Vm\in\ZNd}\omega_\VN^{\Vk\Vm}\hat{\Vu}_{\VN}(\Vm)
\text{ for } \Vu_\VN\in\cTNd \text{ and }
\Vk \in \ZNd,
\end{align*}
where \rev{}{$\Vx_\VN^{\Vk}$ denotes the grid points}
\begin{align}\label{eq:grid_points}
\xNk
=
\sum_{\alp}\frac{Y_{\alp}k_{\alp}}{N_{\alp}}\cb{\alp}\quad
\text{for } 
\Vk \in \ZNd.
\end{align}
Indeed, expanding a function $\Vu_\VN:\puc\rightarrow\xCd$ into Fourier series 
\begin{align*}
\Vu_{\VN}(\Vx) & =  
\sum_{\Vk\in\ZNd}\hat{\Vu}_{\VN}(\Vk)\varphi_{\Vk}(\Vx)
= 
\frac{1}{\meas{\VN}}
\sum_{\Vk\in\ZNd} 
\sum_{\Vm\in\ZNd}
\omega_\VN^{-\Vk\Vm}
\Vu_{\VN}(\Vx_\VN^{\Vm})
\varphi_{\Vk}(\Vx)
\\
& =
\sum_{\Vm\in\ZNd}
\underbrace{\sum_{\Vk\in\ZNd} 
\frac{1}{\meas{\VN}}
\omega_\VN^{-\Vk\Vm}
\varphi_{\Vk}(\Vx)}_{\varphi_{\VN,\Vm}(\Vx)}
\Vu_{\VN}(\Vx_\VN^{\Vm})
\end{align*}
gives rise to the fundamental trigonometric polynomial $\varphi_{\VN,\Vm}$. In addition, these basis functions possess the Dirac delta property \eqref{eq:prop_phi_2_E}, Figure~\ref{fig:trig_basis}.

For further reference, these relations can be cast in the compact form  
\begin{align}\label{eq:trig_connection}
\Vu_{\VN} &=  \sum_{\Vk\in\ZNd}\hat{\MBu}_\VN^{\Vk}\varphi_{\Vk} = \sum_{\Vk\in\ZNd}{\MBu_{\VN}^{\Vk}\varphi_{\VN,\Vk}}
\quad\text{with }
\hat{\MBu}_\VN = \MBF_{\VN} \MBu_\VN\in\xCdN
\text{ and }
\MBu_\VN = \MBFi_{\VN}\hat{\MBu}_\VN\in\xCdN,
\end{align}
where $\hat{\MBu}_\VN^{\Vk} = \hat{\Vu}_\VN(\Vk)\in\xCd$  and $\MBu_\VN^{\Vk} = \Vu_\VN(\Vx_\VN^\Vk)\in\xCd$, and the matrices  
\begin{align}\label{eq:DFT}
\MBF_{\VN} &= \frac{1}{\pVN} \bigl( \del_{\alp\beta} \omega_{\VN}^{-\Vm\Vk} \bigr)_{\alp,\beta}^{\Vm,\Vk\in\ZNd} \in\xhMN,
&
\MBFi_{\VN} &= \bigl( \del_{\alp\beta} \omega_{\VN}^{\Vm\Vk} \bigr)_{\alp,\beta}^{\Vm,\Vk\in\ZNd} \in\xhMN
\end{align}
implement the vector-valued DFT and iDFT.

The relation between the two spaces of trigonometric polynomials depends on grid parity. For odd grids, $\ZNd\setminus\ZNdr = \emptyset$, and it follows from~\eqref{eq:trig_connection} that the spaces coincide: 
\begin{align*}
\cTNd = \cTNtd,
&&
\ZNdr = \ZNd
\qquad\text{ for odd grid assumption \eqref{eq:N_odd}.}
\end{align*}
This property is lost in general due to the Nyquist frequencies $\Vk\in\ZNd\setminus\ZNdr$, and only the following inclusions hold
\begin{align*}
 \cTNd \subseteq \cTNtd,
 && 
 \ZNdr \subseteq \ZNd.
\end{align*}
As a result, the interpolation space is non-conforming for non-odd grids, $\cTNtd \not\subset \Lper{2}{\xRd}$, because the fundamental trigonometric polynomials \eqref{eq:trig_basis_shape} become complex-valued off the grid points, despite being real-valued at the grid points due to the Dirac delta property. Thus, the interpolation space $\cTNtd$ admits an equivalent definition via Fourier coefficients
\begin{align*}
\cTNtd &= \biggl\{ \sum_{\Vk\in\ZNd} \hat{\MBv}_{\VN}^{\Vk} \varphi_{\Vk}:\hat{\MBv}_\VN \in \MBF_{\VN}(\xXN) \biggr\}.
\end{align*}


These arguments can be formalized by introducing suitable operators, which will be useful when dealing with the Galerkin approximations and their fully discrete versions later in Section~\ref{sec:gani}.
\begin{definition}[Operators]
\label{def:operators}
Using grid points $\Vx_\VN^\Vk$ for $\VN\in\xNd$ and $\Vk\in\ZNd$ according to \eqref{eq:grid_points}, the interpolation operator $\QN:C_{\per}(\puc;\xRd)\rightarrow \Lper{2}{\xCd}$, the truncation operator $\PN:\Lper{2}{\xRd}\rightarrow \Lper{2}{\xRd}$, and the discretization operator $\IN:\Cper{0}{\xCd}\rightarrow\xhXN$, are defined by
\begin{subequations}
\label{eq:operators}
\begin{align}
\label{eq:QN_E}
\QN[\Vu] &= \sum_{\Vk\in\ZNd} \Vu(\Vx^{\Vk}_{\VN}) \varphi_{\VN,\Vk},
\\
\label{eq:def_P_N_E}
\PN[\Vu] &= \sum_{\Vk\in\ZNdr} \widehat{\Vu}(\Vk) \varphi_{\Vk},
\\
\label{eq:def_IN_E}
\IN[\Vu] &= \bigl( u_{\alp}(\Vx_\VN^{\Vk}) \bigr)_{\alp=1,\dotsc,d}^{\Vk\in\ZNd}.
\end{align}
\end{subequations}
\end{definition}
The following lemma summarizes the relevant properties of operators~\eqref{eq:operators} and trigonometric polynomials \eqref{eq:trig_basis}. The proof generalizes the results from \cite{SaVa2000PIaPDE,Vondrejc2013PhD,VoZeMa2014FFTH} obtained under the odd grid assumption \eqref{eq:N_odd} to the general case; it is outlined here to keep the paper self-contained.
\begin{lemma}\label{lem:property_phi_IN_QN_PN_E}
\begin{enumerate}
\item 
The operator $\IN$
is an one-to-one  isometric map from $\cTNtd$ onto $\xXN$, i.e. for all $\Vu_\VN,\Vv_\VN\in\cTNtd$
\begin{align}\label{eq:from_L2_to_FD_E}
\scal{\Vu_\VN}{\Vv_\VN}_{\Lper{2}{\rev{}{\xCd}}}
=
\scal{\IN[\Vu_\VN]}{\IN[\Vv_\VN]}_{\xRdN}.
\end{align}
Moreover, for all $\Vu\in \Cper{0}{\xRd}$, we have
\begin{align}\label{eq:IN_QN_property_E}
 \IN\bigl[\QN[\Vu]\bigr] = \IN\bigl[\Vu\bigr].
\end{align}
\item The interpolation operator $\QN$
is a projection with the image $\cTNtd$,
\item The truncation operator $\PN$ is an orthogonal projection with the image $\cTNd$.
\end{enumerate}
\end{lemma}
\begin{proof}
\begin{enumerate}
 \item From \eqref{eq:prop_phi_2_E} we see that
trigonometric polynomials, e.g. $\Vu_\VN,\Vv_\VN\in\cTNtd$, are uniquely defined by their grid values, so that
\begin{align*}
\scal{\Vu_\VN}{\Vv_\VN}_{\Lper{2}{\rev{}{\xCd}}}
& = 
\sum_{\Vk,\Vm\in\ZNd} 
\scal{\Vu_\VN(\xNk)}{\Vv_\VN(\xNm)}_{\xRd}
\cdot
\scal{\bfunN{\Vk}}{\bfunN{\Vm}}_{\Lper{2}{\rev{}{\xCd}}}
\\
& = 
\sum_{\Vk,\Vm\in\ZNd} 
\scal{\Vu_\VN(\xNk)}{\Vv_\VN(\xNm)}_{\xRd}
\cdot 
\frac{\del_{\Vk\Vm}}{\meas{\VN}}
=
\scal{\IN[\Vu_\VN]}{\IN[\Vv_\VN]}_{\xRdN}.
\end{align*}
 \item follows from \eqref{eq:prop_phi_2_E} and the definition of the space $\cTNtd$.
 \item follows from orthogonality of Fourier trigonometric polynomials \eqref{eq:Fourier_basis_functions_orthogonality} and the definition of the space $\cTNd$.
\end{enumerate}
\end{proof}

\subsection{Fully discrete spaces --- odd grids}
\label{sec:FD_spaces_odd}
The focus of this section is on the fully discrete spaces storing the values of the trigonometric polynomials at grids with the odd number of points \eqref{eq:N_odd}. As first recognized in \cite{VoZeMa2014FFTH}, the remarkable property of such discretizations is that the structure of the continuous problem is translated into the discrete case in a conforming way, cf.~Figure~\ref{fig:subspaces_scheme_odd}.
\begin{definition}[Fully discrete projections]
\label{def:FD_projections}Let $\mhG\sub{\bullet}(\Vk)\in\xRdd$ for $\bullet\in\{\cU,\cE,\cJ\}$ and $\Vk\in\Zd$ be the Fourier coefficients from Definition~\ref{def:projection}.
We define block diagonal matrices $\MBhG\sub{\cU}_\VN, \MBhG\sub{\cE}_\VN$, and $ \MBhG\sub{\cJ}_\VN\in\xMN$  in the Fourier domain as
\begin{align}\label{eq:FD_proj_Fourier_O}
\bigl(\MhG\sub{\bullet}_\VN\bigr)_{\alp\beta}^{\Vk\Vm} &= \hG^\bullet_{\alp\beta}(\Vk)\del_{\Vk\Vm},
\end{align}
where $\Vk,\Vm\in\ZNd$ and $\bullet\in\{\cU,\cE,\cJ\}$. The real domain equivalents are obtained by similarity transformations using DFT \eqref{eq:DFT}, i.e.~
\begin{align*}
\MBG\sub{\bullet}_\VN = \MBFi_\VN \MBhG \sub{\bullet}_\VN \MBF_\VN.
\end{align*}
\end{definition}
\begin{lemma}\label{lem:matrix_orthogonality}
Matrices $\MBG\sub{\bullet}_\VN$ for $\bullet\in\{\cU,\cE,\cJ\}$  constitute the identity
\begin{align}\label{eq:projections_identity}
  \MBG\sub{\cU}_\VN + \MBG\sub{\cE}_\VN  + \MBG\sub{\cJ}_\VN = \MB{I}
\end{align}
and are mutually orthogonal projections on $\xRdN$. 
\end{lemma}
\begin{proof}
The resolution of identity \eqref{eq:projections_identity} follows from Definitions~\ref{def:projection} and~\ref{def:FD_projections}. The projection properties with their orthogonality are inherited from the continuous projections, cf.~Lemma~\ref{lem:Helmholtz_decomp} and \cite[Section~12.1]{Milton2002TC}.
\end{proof}
\begin{definition}[Finite dimensional subspaces]\label{def:FD_subspaces}
The previously defined projections provide us with the following subspaces of $\xXN$
\begin{align*}
\xUN &= \MBG_\VN\sub{\cU} [\xXN],
&
\xEN &= \MBG_\VN\sub{\cE}[\xXN],
&
\xJN &= \MBG_\VN\sub{\cJ}[\xXN],
\end{align*}
and their trigonometric counterparts
\begin{align}\label{eq:def_cUN_cEN_cJN}
\cUN &= \INi[\xUN],
&
\cEN &= \INi[\xEN],
&
\cJN &= \INi[\xJN].
\end{align}
\end{definition}
The relation of these subspaces to the Helmholtz decomposition \eqref{eq:Helmholtz_L2} is clarified by Figure~\ref{fig:subspaces_scheme_odd} and the following lemma.
\begin{figure}[htp]
\centering
\includegraphics[scale=0.8]{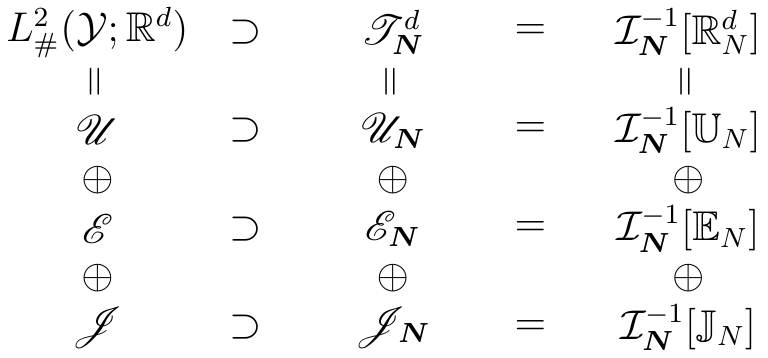}
\caption{The scheme of subspaces for odd grids}
\label{fig:subspaces_scheme_odd}
\end{figure}
\begin{lemma}\label{lem:FD_decomposition}
\begin{enumerate}
 \item Space $\xXN$ can be decomposed into three mutually orthogonal subspaces
\begin{align}
\label{eq:FD_Helmholtz_decomposition}
\xXN &= \xUN \oplus \xEN \oplus \xJN.
\end{align}
 \item 
 The scheme in Figure~\ref{fig:subspaces_scheme_odd} is valid and
\begin{align}\label{eq:connection_projections}
\mathcal{G}\sub{\bullet}[\cTNd] = \INi\bigl[\MBG\sub{\bullet}_\VN[\xRdN]\bigr]\quad\text{for }\bullet\in\{\cU,\cE,\cJ\}.
\end{align}
\end{enumerate}
\end{lemma}
\begin{proof}
The Helmholtz-like decomposition of trigonometric polynomials, the second column in Figure~\ref{fig:subspaces_scheme_odd}, is accomplished with the same set of projections $\mathcal{G}\sub{\bullet}$ for $\bullet\in\{\cU,\cE,\cJ\}$
as they satisfy
\begin{align*}
 \mathcal{G}\sub{\bullet}[\cTNd]\subset\cTNd\quad\text{for }\bullet\in\{\cU,\cE,\cJ\}.
\end{align*}
The connection of continuous projections and fully discrete projections in \eqref{eq:connection_projections} is a consequence of isometry of the discretization operator $\IN$ proven in Lemma~\ref{lem:property_phi_IN_QN_PN_E}, two representations of trigonometric polynomials \eqref{eq:trig_connection}, and the definition of the fully discrete projections \eqref{eq:FD_proj_Fourier_O} via continuous ones. The last column in Figure~\ref{fig:subspaces_scheme_odd} is then obvious.
\end{proof}
\begin{remark}
The previous proof yields an alternative characterization of the conforming subspaces
\begin{align*}
\cUN&=\cU\bigcap\cTNd,
&
\cEN&=\cE\bigcap\cTNd,
&
\cJN&=\cJ\bigcap\cTNd.
\end{align*}
Thus, $\cUN, \cEN$, and $\cJN$ represent the subspaces of constant, curl-free, and divergence-free vector-valued polynomials, while $\xUN=\IN[\cUN]$, $\xEN=\IN[\cEN]$, and $\xJN=\IN[\cJN]$ collect their values at the grid points.
\end{remark}
\subsection{Fully discrete spaces --- general grids}
\label{sec:FD_spaces_non_odd}
The framework of fully discrete spaces, introduced in previous section for odd grid assumption \eqref{eq:N_odd}, is extended here to the general grids.
Similarly to Section~\ref{sec:trig_defs}, the special attention is given to the Nyquist frequencies $\Vk\in\ZNd\setminus\ZNdr$ in order to obtain the conforming approximation spaces.
\begin{definition}[Fully discrete projections]
\label{def:FD_proj_E}
Let $\mhG\sub{\bullet}(\Vk)\in\xRdd$ for $\bullet\in\{\cU,\cE,\cJ\}$ and $\Vk\in\Zd$ be the Fourier coefficients from Definition~\ref{def:projection}.
We define the block diagonal matrices $\MBhG\sub{\cU}_\VN, \MBhG\sub{\cE}_{\VN,\Vo},\MBhG\sub{\cE}_{\VN,\mI},\MBhG\sub{\cJ}_{\VN,\Vo}$, and $\MBhG\sub{\cJ}_{\VN,\mI}\in\xMN$ in the Fourier domain as
\begin{align*}
(\MBhG\sub{\cU}_\VN)_{\alp\beta}^{\Vk\Vm} &= \hG^{\cU}_{\alp\beta}(\Vk)\del_{\Vk\Vm},
\\
(\MBhG\sub{\bullet}_{\VN,\Vo})_{\alp\beta}^{\Vk\Vm} &= 
\begin{cases}
\hG\sub{\bullet}_{\alp\beta}(\Vk)\del_{\Vk\Vm},
\\
0,
\end{cases}
&
(\MBhG\sub{\bullet}_{\VN,\mI})_{\alp\beta}^{\Vk\Vm} & = 
\begin{cases}
\hG\sub{\bullet}_{\alp\beta}(\Vk)\del_{\Vk\Vm} & \text{for }\Vk\in\ZNdr,
\\
\del_{\alp\beta}\del_{\Vk\Vm}&\text{for }\Vk\in\ZNd\setminus\ZNdr,
\end{cases}
\end{align*}
where $\Vk,\Vm\in\ZNd$ and $\bullet\in\{\cE,\cJ\}$.
The real domain equivalents are obtained by similarity transformations using the DFT matrices~\eqref{eq:DFT}, i.e.~
\begin{align*}
\MBG\sub{\cU}_\VN &= \MBFi_\VN\MBhG\sub{\cU}_\VN\MBF_\VN,
&
\MBG_{\VN,\Vo}\sub{\bullet} &= \MBFi_\VN\MBhG\sub{\bullet}_{\VN,\Vo}\MBF_\VN,
&
\MBG_{\VN,\mI}\sub{\bullet} &= \MBFi_\VN\MBhG\sub{\bullet}_{\VN,\mI}\MBF_\VN
\end{align*}
for $\bullet\in\{\cE,\cJ\}$.
\end{definition}
\begin{lemma}\label{lem:matrix_orthogonality_E}
The two triples of matrices $\{ \MBG\sub{\cU}_\VN, \MBG\sub{\cE}_{\VN,\Vo}, \MBG\sub{\cJ}_{\VN,\mI}\}$ and $\{\MBG\sub{\cU}_\VN, \MBG\sub{\cE}_{\VN,\mI}, \MBG\sub{\cJ}_{\VN,\Vo}\}$
constitute identities
\begin{align}\label{eq:FD_proj_id_E}
\MB{I} &= \MBG\sub{\cU}_\VN + \MBG\sub{\cE}_{\VN,\Vo}  + \MBG\sub{\cJ}_{\VN,\mI},
&
\MB{I} &= \MBG\sub{\cU}_\VN + \MBG\sub{\cE}_{\VN,\mI}  + \MBG\sub{\cJ}_{\VN,\Vo},
\end{align}
and each triple consists of mutually orthogonal projections on $\xRdN$.
\end{lemma}
\begin{proof}The resolution of identity \eqref{eq:FD_proj_id_E} follows from Definitions~\ref{def:projection} and~\ref{def:FD_proj_E}. The projection properties and their orthogonality are proven in the same way as in Lemma~\ref{lem:matrix_orthogonality} for $\Vk\in\ZNdr$, and are the direct consequence of Definition~\ref{def:FD_proj_E} for the Nyquist frequencies $\Vk\in\ZNd\setminus\ZNdr$.
\end{proof}
\begin{definition}[Finite dimensional subspaces] \label{def:FD_subspaces_E}
With the previously defined projections, we introduce the subspaces of $\xXN$
\begin{subequations}
\begin{align}
\label{eq:subspaces_xUN_xEN_xJN_E}
\xUN &= \MBG\sub{\cU}_\VN [\xXN],
&
\xEN &= \MBG\sub{\cE}_{\VN,\Vo}[\xXN],
&
\xJN &= \MBG\sub{\cJ}_{\VN,\Vo}[\xXN],
\\
\label{eq:subspaces_xENapp_xJNapp}
&
&
\xENapp &= \MBG\sub{\cE}_{\VN,\mI}[\xXN],
&
\xJNapp &= \MBG\sub{\cJ}_{\VN,\mI}[\xXN]
\end{align}
\end{subequations}
and their trigonometric counterparts
\begin{align*}
\cUN &= \INi[\xUN],
&
\cEN &= \INi[\xEN],
&
\cJN &= \INi[\xJN],
\\
&
&
\cENapp &= \INi[\xENapp],
&
\cJNapp &= \INi[\xJNapp].
\end{align*}
\end{definition}
Compared to the previous section, the relations among these subspaces 
are more intricate, see Figure~\ref{fig:subspaces_scheme} and the following lemma.
\begin{figure}[htp]
\centering
\includegraphics[scale=0.8]{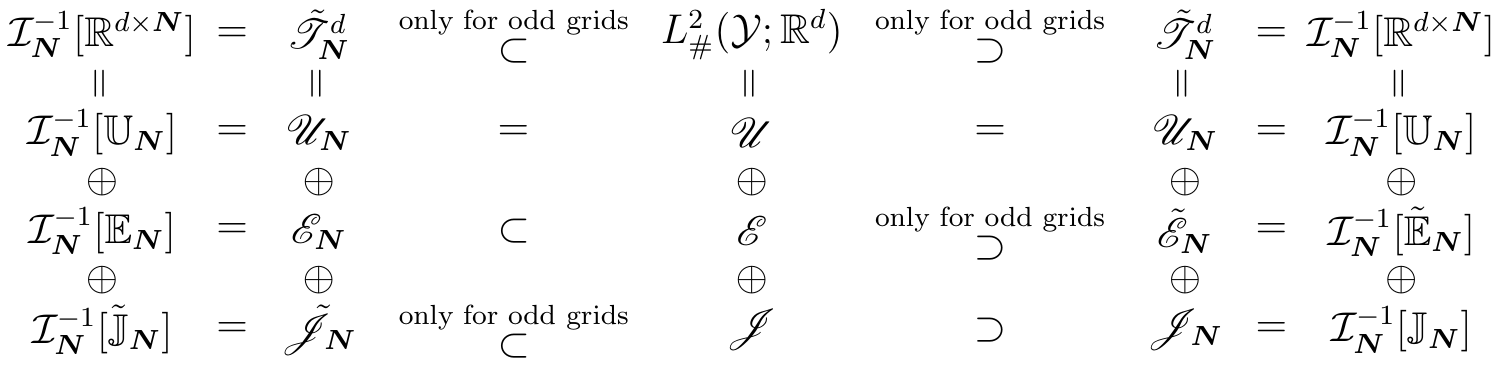}
\caption{The scheme of subspaces for general grids}
\label{fig:subspaces_scheme}
\end{figure}
\begin{lemma}\label{lem:FD_decomposition_E}
For the subspaces from  Definition \ref{def:FD_subspaces_E}, the following holds:
\begin{enumerate}
  \item Space $\xXN$ admits two alternative orthogonal decompositions
\begin{align}
\label{eq:FD_Helmholtz_decomposition_E}
\xXN &= \xUN \oplus \xEN \oplus \xJNapp,
&
\xXN &= \xUN \oplus \xENapp \oplus \xJN.
\end{align}
Moreover, the subspaces $\xENapp$ and $\xJNapp$ enlarge the original ones, i.e.
\begin{align*}
 \xEN &\subseteq \xENapp,
 &
 \xJN &\subseteq \xJNapp,
\end{align*}
and coincide only for odd grids \eqref{eq:N_odd}.
\item The scheme in Figure~\ref{fig:subspaces_scheme} is valid and
\begin{align}
\label{eq:connection_projections_E}
\mathcal{G}\sub{\cU}[\cTNd] &= \INi[\xUN],
&
\mathcal{G}\sub{\cE}[\cTNd] &= \INi[\xEN],
&
\mathcal{G}\sub{\cJ}[\cTNd] &= \INi[\xJN].
\end{align}
\end{enumerate}
\end{lemma}
\begin{proof}
Eq.~\eqref{eq:FD_Helmholtz_decomposition_E} is a consequence of resolutions of identity \eqref{eq:FD_proj_id_E}. The rest in (i) follows from Definition~\ref{def:FD_proj_E} of the fully discrete projections, once noticing that the pairs of matrices $\{ \MBG^{\cE}_{\VN,\Vo},\MBG^{\cE}_{\VN,\mI} \}$ and $\{ \MBG^{\cJ}_{\VN,\Vo},\MBG^{\cJ}_{\VN,\mI} \}$ coincide for odd grids \eqref{eq:N_odd} and differ only for the Nyquist frequencies $\Vk\in\ZNd\setminus\ZNdr$.

The special case of part (ii) for odd grids \eqref{eq:N_odd} has already been proven in Lemma~\ref{lem:FD_decomposition}; in such case the spaces of trigonometric polynomials $\cTNd$ and $\cTNtd$ coincide. 
Utilizing Lemma~\ref{lem:property_phi_IN_QN_PN_E} (i), we are left with
\begin{align*}
 \cU &= \cUN,
 &
 \cEN &\subsetneq \cE,
 &
 \cJN &\subsetneq \cJ.
\end{align*}
While the equality is evident, the inclusions follows from \eqref{eq:connection_projections_E} and from a property of continuous projections
\begin{align*}
 \mathcal{G}\sub{\bullet}[\cTNd] \subsetneq \cTNd \subsetneq \Lper{2}{\xRd}\quad\text{for }\bullet\in\{\cE,\cJ\}.
\end{align*}

Finally, the proof of \eqref{eq:connection_projections_E} follows from the connection of representations \eqref{eq:trig_connection} and from the fact that the Nyquist frequencies $\Vk\in\ZNd\setminus\ZNdr$ are left out in the definition of projections $\MBG\sub{\cE}_{\VN,\Vo}$ and $\MBG\sub{\cJ}_{\VN,\Vo}$, recall Definition~\ref{def:FD_proj_E}.
\end{proof}
\begin{remark}
The previous proof yields an alternative characterization of the conforming subspaces
\begin{align*}
\cEN &= \cE\bigcap\cTNd,
&
\cJN &= \cJ\bigcap\cTNd.
\end{align*}
\end{remark}

\section{Galerkin approximation with numerical integration}
\label{sec:gani}
This section deals with the discretization of \eqref{eq:homog_problem} by the Galerkin approximation with numerical integration (GaNi), a scheme which has been introduced and analyzed in \cite[Section~4.3]{VoZeMa2014FFTH} for the odd grids \eqref{eq:N_odd}.
Here, the method is generalized to the primal-dual setting and general grids, by utilizing the discretization strategy shown in Figure~\ref{fig:operators}.
\begin{figure}[htp]
\centering
\includegraphics[scale=0.8]{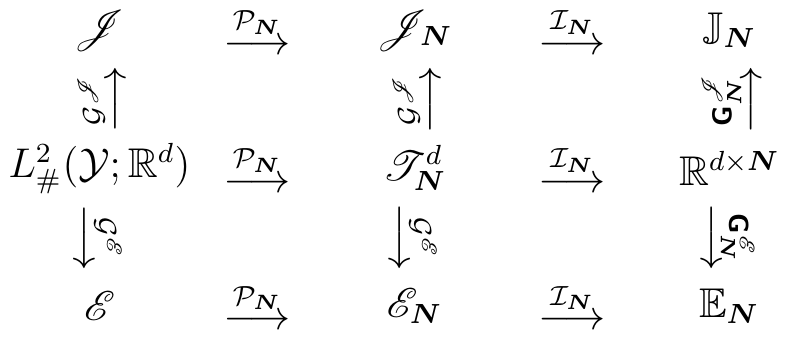}
\caption{Discretization strategy}
\label{fig:operators}
\end{figure}

The discretization consists in the approximation of bilinear forms \eqref{eq:bilinear_forms} using the interpolation operator \eqref{eq:QN_E} and the \rev{}{rectangular integration rule~\eqref{eq:from_L2_to_FD_E}, which will be referred to as the trapezoidal rule in order to adhere to the terminology used in spectral methods, e.g.~\cite[p.~94]{boyd_chebyshev_2001}, because the two integration rules coincide for periodic functions and regular grids. As a result, we obtain} the discretization-dependent forms $\bilfN{}{},\bilfNi{}{} :\cTNd\times\cTNd\rightarrow\xR$ given by
\begin{subequations}
\begin{align}
\label{eq:bilinear_forms_approx}
\bilfN{\Vu_\VN}{\Vv_\VN} &:= \scal{\QN[\TA\Vu_\VN]}{\Vv_\VN}_{\Lper{2}{\rev{}{\xCd}}},
&
\bilfNi{\Vu_\VN}{\Vv_\VN} &:= \scal{\QN[\TA^{-1}\Vu_\VN]}{\Vv_\VN}_{\Lper{2}{\rev{}{\xCd}}}.
\end{align}
\end{subequations}
\begin{definition}[Galerkin approximation with numerical integration (GaNi)]\label{def:GaNi}
Let the material coefficients satisfy \eqref{eq:A} and  $\TA\in\Cper{0}{\xRddspd}$. Then, the approximate primal and dual homogenized matrices $\AeffN, \BeffN \in \xR^{d\times d}$ are defined as
\begin{subequations}\label{eq:GaNi}
\begin{align}
\label{eq:GaNi_primal}
\scal{\AeffN\VE}{\VE}_{\xRd} &= \min_{\Ve_{\VN}\in \cEN } \bilfN{\VE+\Ve_{\VN}}{\VE+\Ve_{\VN}} =  \bilfN{\VE+\Ve_{\VN}\mac{\VE}}{\VE+\Ve_{\VN}\mac{\VE}},
\\
\label{eq:GaNi_dual}
\scal{\BeffN\VJ}{\VJ}_{\xRd} &= \min_{\Vj_{\VN}\in \cJN} \bilfNi{\VJ+\Vj_{\VN}}{\VJ+\Vj_{\VN}} = \bilfNi{\VJ+\Vj_{\VN}\mac{\VJ}}{\VJ+\Vj_{\VN}\mac{\VJ}},
\end{align}
\end{subequations}
for arbitrary $\VE,\VJ\in\xRd$. 
\end{definition}
\rev{}{Before we discuss the properties of solutions to GaNi in Remark~\ref{rem:gani_solvability}, we first introduce the fully discrete versions of the bilinear forms \eqref{eq:bilinear_forms_approx}.}
\begin{lemma}\label{lem:FD_of_GaNi}
Under assumptions of the Definition~\ref{def:GaNi}, we have
\begin{subequations}
\label{eq:gani_2_FD}
\begin{align}
\bilfN{\Vu_\VN}{\Vv_\VN} &=
\bilfDN{\MBu_\VN}{\MBv_\VN} := \scal{\MBA_\VN\MBu_\VN}{\MBv_\VN}_{\xXN},
\\
\bilfNi{\Vu_\VN}{\Vv_\VN} &=
\bilfDNi{\MBu_\VN}{\MBv_\VN} := \scal{\MBB_\VN\MBu_\VN}{\MBv_\VN}_{\xXN},
\end{align}
\end{subequations}
where
\begin{align*}
\MBu_\VN&:=\IN[\Vu_\VN]\in\xRdN,
&
\MBv_\VN&:=\IN[\Vv_\VN]\in\xRdN,
\end{align*}
and the components of the matrices
$\MBA_\VN,\MBB_\VN\in\xMN$ are defined as
\begin{align}\label{eq:FD_A_GaNi}
\MBA^{\Vk\Vm}_{\VN} &=
\TA(\Vx_\VN^{\Vk})\del_{\Vk\Vm},
&
\MBB_{\VN}^{\Vk\Vm} &=
\TAi(\Vx_\VN^{\Vk})\del_{\Vk\Vm}
\end{align}
for $\Vk,\Vm\in\ZNd$. Moreover, 
\begin{align}\label{eq:FD_inverse_matrix_equality}
\MBA_{\VN} = \MBB_{\VN}^{-1}.
\end{align}
\end{lemma}
\begin{proof}The proof is a consequence of Lemma~\ref{lem:property_phi_IN_QN_PN_E} (i), particularly Eqs.~\eqref{eq:from_L2_to_FD_E} and \eqref{eq:IN_QN_property_E}, together with the definition of the operator $\IN$ in \eqref{eq:def_IN_E}.
\end{proof}
\rev{}{\begin{remark}\label{rem:gani_solvability}
The existence of approximate solutions to GaNi, $\Ve_{\VN}\mac{\VE}$, and their convergence to the solution of continuous problem, $\Ve\mac{\VE}$, follows from exactly the same arguments as for the odd grids \eqref{eq:N_odd} stated in \cite[Proposition~8]{VoZeMa2014FFTH}. Indeed, since the approximate bilinear forms \eqref{eq:bilinear_forms_approx} are defined locally, recall~\eqref{eq:gani_2_FD}, the
symmetry, positive definiteness, and boundedness follow from the assumption on coefficients \eqref{eq:A} and the identity $\|\MBv_\VN\|_{\xRdN} = \|\Vv_\VN\|_{\Lper{2}{\xRd}}$ for $\Vv_\VN=\INi[\MBv_\VN]\in\cTNd$; the existence and uniqueness of $\Ve_{\VN}\mac{\VE}$ is then inferred from the Lax-Milgram lemma. Its convergence to $\Ve\mac{\VE}$ follows from the first Strang lemma \cite{Strang1972varcrime}, which depends not only on the approximation properties of trigonometric polynomials $\inf_{\Vu_\VN\in\cTNd}\|\Ve\mac{\VE}-\Vu_\VN\|_{\Lper{2}{\xRd}}$, but also on the
consistency error of approximate bilinear forms $a_\VN$.
\begin{align}
\label{eq:consist_error}
\sup_{\Vv_\VN\in\cTNd} \frac{\bigl|\bilf{\Vu_\VN}{\Vv_\VN}-\bilfN{\Vu_\VN}{\Vv_\VN}\bigr|}{\|\Vv_\VN\|_{\Lper{2}{\xRd}}},
\end{align}
the so-called variational crime. The resulting estimate
\begin{align}\label{eq:converg_est}
\|\Ve\mac{\VE} - \Ve_\VN\mac{\VE}\| \leq C \left(\max_{\alp}\frac{Y_
\alp}{N_\alp} \right)^s
\end{align}
can be proven for smooth material coefficients with rate $s$ and constant $C$ independent on grid size $\VN$. For nonsmooth data, the regularization approach can be used, see \cite[Section~3.6]{Necas2012direct} and \cite[Secdtion~3.5, pp.~115--117]{Vondrejc2013PhD}, or the arguments can be adjusted to the Riemann integrable coefficients following the results by Schneider~\cite{Schneider2014convergence}. The treatment of the dual solutions $\Vj_\VN\mac{\VJ}$ is established by analogous arguments.
\end{remark}}
\begin{remark}
Recall that the dual formulation \eqref{eq:GaNi_dual} involves inverse coefficients $\TAi$. Interestingly, this property is maintained in the fully discrete formulation~\eqref{eq:FD_inverse_matrix_equality}, so that the assumptions of Proposition~\ref{lem:transform2dual} are met, leading to the duality results in Propositions~\ref{lem:FD_odd_problem} and \ref{lem:FD_even_problem}.
\end{remark}
\begin{remark}\label{rem:GaNi_vs_Ga}
The GaNi scheme coincides with the original Moulinec-Suquet method \cite{Moulinec1998NMC,Moulinec1994FFT} as shown in~\cite[Section~5.3]{VoZeMa2014FFTH} for the variational formulation and in \cite{ZeVoNoMa2010AFFTH,VoZeMa2012LNSC} for the Lippmann-Schwinger equation. 
The reason for using the trapezoidal integration rule in~\eqref{eq:bilinear_forms_approx} is that it can be applied to general coefficients, but the associated numerical scheme may cause a non-monotonous convergence of the approximate solutions, see Section~\ref{sec:results_homogenized_coef_odd}. We will show in Section~\ref{sec:evaluation_bounds} that the quadrature can be avoided for a wide class of coefficients, albeit at a higher computational cost. This procedure provides the Galerkin scheme without numerical integration, proposed theoretically in~\cite[Section~4.2]{VoZeMa2014FFTH}, and studied separately in \cite{Vondrejc2015FFTimproved}.
\end{remark}
The previous lemma, particularly \eqref{eq:gani_2_FD}, enables us to define the homogenization problem in the fully discrete setting that represents the matrix formulation of the GaNi.
\begin{corollary}[Fully discrete formulations of the GaNi]
\label{lem:FD_GaNi}
Under the assumptions of Definition~\ref{def:GaNi}, the primal and the dual homogenized matrices $\AeffN,\BeffN\in\xRdd$ satisfy 
\begin{subequations}\label{eq:FD_GaNi}
\begin{align}
\label{eq:FD_GaNi_prim}
\scal{\AeffN \VE}{\VE}_{\xRd} &= \min_{\MBe_\VN\in\xEN}
\bilfDN{\VE+\MBe_\VN}{\VE+\MBe_\VN} = \bilfDN{\VE+\MBe_\VN\mac{\VE}}{\VE+\MBe_\VN\mac{\VE}},
\\
\label{eq:FD_GaNi_dual}
\scal{\BeffN\VJ}{\VJ}_{\xRd} &=
\min_{\MBj_\VN\in\xJN}
\bilfDNi{\VJ+\MBj_\VN}{\VJ+\MBj_\VN} = \bilfDNi{\VJ+\MBj_\VN\mac{\VJ}}{\VJ+\MBj_\VN\mac{\VJ}}
\end{align}
\end{subequations}
for arbitrary $\VE,\VJ\in\xRd$.
Moreover, the discrete minimizers \rev{}{$\MBe_\VN\mac{\VE}$ and $\MBj\mac{\VJ}_\VN$ exist, are unique, and are connected to $\tVe_\VN\mac{\VE}$ and $\tVj_\VN\mac{\VJ}$ from~\eqref{eq:GaNi} via}
\begin{align*}
\IN\bigl[\tVe_\VN\mac{\VE} \bigr] &= \MBe_\VN\mac{\VE},
&
\IN\bigl[\tVj_\VN\mac{\VJ} \bigr] &= \MBj_\VN\mac{\VJ}.
\end{align*}
\end{corollary}
\begin{remark}
\label{rem:identify_xRd_cUN}
The discrete bilinear forms $\bilfDN{}{}$, $\bilfDNi{}{}$ are defined on $\xRdN\times\xRdN$, rendering the terms $\VE+\MBe_\VN$ and $\VJ+\MBj_\VN$ formally ill-defined. The sums need to be understood with the help of the isometric isomorphism $\IN$ from \eqref{eq:def_IN_E} that identifies $\xRd$ or $\cU$ with $\xUN$, e.g.
\begin{align*}
\bilfDN{\VE+\MBe_\VN}{\VE+\MBe_\VN} &= \bilfDN{\IN[\VE]+\MBe_\VN}{\IN[\VE]+\MBe_\VN}
\quad\text{with }
(\IN[\VE]+\MBe_\VN)^\Vk_\alp = E_\alp + \M{e}^\Vk_{\VN,\alp}.
\end{align*}
\end{remark}
\subsection{Duality for odd grids}
\label{sec:duality_FD_odd}
In this section, the perturbation duality theorem, Proposition~\ref{lem:transform2dual}, is applied to the fully discrete formulation of the GaNi \eqref{eq:FD_GaNi}.
For discretization with odd number of grid points \eqref{eq:N_odd}, it leads to a surprising result: the discrete formulations are mutually dual, so that the duality of continuous formulations \eqref{eq:homog_problem} is preserved under the discretization.
\begin{proposition}\label{lem:FD_odd_problem}
Assuming odd grids \eqref{eq:N_odd}, the following holds for the fully discrete homogenization problem \eqref{eq:FD_GaNi}:
\begin{enumerate}
\item The primal and the dual homogenized matrices are mutually inverse
\begin{align*}
\AeffN = \BeffN^{-1}.
\end{align*}
\item The primal and the dual discrete minimizers $\MBe^{(\alp)}_\VN\in\xEN$, $\MBj^{(\alp)}_\VN\in\xJN$ are related via
\begin{align}
\label{eq:FD_superposition_odd_A}
\cb{\beta} + \MBj^{(\beta)}_\VN &= \MBA_\VN\sum_{\alp} E_\alp ( \cb{\alp} + \MBe^{(\alp)}_\VN ),
\end{align}
where $\VE = \AeffN^{-1} \cb{\beta}$.
\end{enumerate}
\end{proposition}
\begin{proof}
The proof is a direct consequence of Proposition~\ref{lem:transform2dual} for
\begin{align*}
\begin{array}{ccccccc}
\cH & = &\rcU & \oplus & \rcE & \oplus & \rcJ \\
\rotatebox{90}{=} &  & \rotatebox{90}{=} &  & \rotatebox{90}{=} &  & \rotatebox{90}{=} \\
\xXN & = &\xUN & \oplus & \xEN & \oplus & \xJN \\
\end{array}
\end{align*}
and
\begin{align*}
\bilfG{}{} &= \bilfDN{}{},
&
\bilfGi{}{} &= \bilfDNi{}{},
&
\mathring{\TA}_\eff &= \AeffN,
&
\mathring{\TB}_\eff &= \BeffN,
&
\rVe\mac{\VE} &= \MBe_\VN\mac{\VE},
&
\rVj\mac{\VJ} &= \MBj_\VN\mac{\VJ}.
\end{align*}
\end{proof}
\subsection{Duality for general grids}
\label{sec:duality_FD_non_odd}
For general grids, the fully discrete formulations \eqref{eq:FD_GaNi} lack the mutual duality as the fully discrete subspaces may not exhaust the whole $\xRdN$, i.e.
\begin{align*}
 \xUN \oplus \xEN \oplus \xJN \subseteq \xRdN\quad\text{and the equality holds only for odd grids \eqref{eq:N_odd}},
\end{align*}
cf. Figure~\ref{fig:subspaces_scheme}. However, Proposition~\ref{lem:FD_even_problem} below shows that the formulations for matrices $\AeffN$ and $\BeffN$ from~\eqref{eq:FD_GaNi} are in duality with
\begin{subequations}
\label{eq:FD_GaNi_tilde}
\begin{align}
\label{eq:FD_GaNi_prim_tilde}
\scal{\tilde{\TB}_{\eff,\VN}\VJ}{\VJ}_{\xRd} &=  \min_{\MBj_\VN\in \xJNapp}  \bilfDNi{\VJ+\MBj_\VN}{\VJ+\MBj_\VN} = \bilfDNi{\VJ+\tMBj_\VN\mac{\VJ}}{\VJ+\tMBj_\VN\mac{\VJ}},
\\
\label{eq:FD_GaNi_dual_tilde}
\scal{\tilde{\TA}_{\eff,\VN}\VE}{\VE}_{\xRd} &= 
\min_{\MBe_\VN\in \xENapp} \bilfDN{\VE+\MBe_\VN}{\VE+\MBe_\VN} = \bilfDN{\VE+\tMBe_\VN\mac{\VE}}{\VE+\tMBe_\VN\mac{\VE}},
\end{align}
\end{subequations}
when using the dual spaces $\xENapp$ and $\xJNapp$ from \eqref{eq:subspaces_xENapp_xJNapp}.
\begin{proposition}\label{lem:FD_even_problem}
The following holds for the fully discrete homogenization problems \eqref{eq:FD_GaNi} and \eqref{eq:FD_GaNi_tilde}:
\begin{enumerate}
  \item The homogenized matrices from the fully discrete formulations \eqref{eq:FD_GaNi_prim} and \eqref{eq:FD_GaNi_dual} coincide with those in \eqref{eq:FD_GaNi_prim_tilde} and \eqref{eq:FD_GaNi_dual_tilde}, respectively
  \begin{align}\label{eq:gani_eff_equality}
   \TA_{\eff,\VN} &= \tilde{\TB}^{-1}_{\eff,\VN},
   &
   \TB^{-1}_{\eff,\VN} &= \tilde{\TA}_{\eff,\VN}.
  \end{align}
  \item The discrete minimizers $\MBe^{(\beta)}_\VN\in\xEN$ and $\MBj^{(\alp)}_\VN\in\xJN$ of 
  \eqref{eq:FD_GaNi_prim} and \eqref{eq:FD_GaNi_dual}
  are related to the minimizers $\tMBe^{(\alp)}_{\VN}\in\xENapp$ and $\tMBj^{(\alp)}_{\VN}\in\xENapp$ of
  \eqref{eq:FD_GaNi_prim_tilde} and \eqref{eq:FD_GaNi_dual_tilde}
  via
\begin{align}
\label{eq:FD_superposition_even_A}
\cb{\beta} + \MBe^{(\beta)}_\VN &= \MBA_\VN^{-1}\sum_{\alp} J_\alp ( \cb{\alp} + \tMBj^{(\alp)}_{\VN} ),
&
\cb{\beta} + \MBj^{(\beta)}_\VN &= \MBA_\VN\sum_{\alp} E_\alp ( \cb{\alp} + \tMBe^{(\alp)}_{\VN} ),
\end{align}
with $\VE := \BeffN \cb{\beta}$ and $\VJ := \AeffN \cb{\beta}$.
\item  The primal and the dual homogenized matrices satisfy
\begin{align}\label{eq:Leff_FFTH_ineq}
 \BeffN^{-1} \preceq \AeffN.
\end{align}
\end{enumerate}
\end{proposition}
\begin{proof}The proof of parts (i) and (ii) is a consequence of Proposition~\ref{lem:transform2dual}. The equivalence between \eqref{eq:FD_GaNi_prim} and \eqref{eq:FD_GaNi_prim_tilde} is shown by
\begin{align*}
\begin{array}{ccccccc}
\cH & = &\rcU & \oplus & \rcE & \oplus & \rcJ \\
\rotatebox{90}{=} &  & \rotatebox{90}{=} &  & \rotatebox{90}{=} &  & \rotatebox{90}{=} \\
\xXN & = &\xUN & \oplus & \xEN & \oplus & \xJNapp \\
\end{array}
\end{align*}
and
\begin{align*}
\bilfG{}{} &= \bilfDN{}{},
&
\bilfGi{}{} &= \bilfDNi{}{},
&
\mathring{\TA}_\eff &= \AeffN,
&
\mathring{\TB}_\eff &= \tilde{\TB}_{\eff,\VN},
&
\rVe\mac{\VE} &= \MBe_\VN\mac{\VE},
&
\rVj\mac{\VJ} &= \tilde{\MBj}_{\VN} \mac{\VJ}.
\end{align*}
The equivalence between
\eqref{eq:FD_GaNi_dual} and \eqref{eq:FD_GaNi_dual_tilde} follows from
\begin{align*}
\begin{array}{ccccccc}
\cH & = &\rcU & \oplus & \rcE & \oplus & \rcJ \\
\rotatebox{90}{=} &  & \rotatebox{90}{=} &  & \rotatebox{90}{=} &  & \rotatebox{90}{=} \\
\xXN & = &\xUN & \oplus & \xENapp & \oplus & \xJN \\
\end{array}
\end{align*}
and
\begin{align*}
\bilfG{}{} &= \bilfDN{}{},
&
\bilfGi{}{} &= \bilfDNi{}{},
&
\mathring{\TA}_\eff &= \tilde{\TA}_{\eff,\VN}
&
\mathring{\TB}_\eff &= \TB_{\eff,\VN},
&
\rVe\mac{\VE} &= \tilde{\MBe}_{\VN}\mac{\VE},
&
\rVj\mac{\VJ} &= \MBj_{\VN}\mac{\VJ}.
\end{align*}

The proof of the duality gap (iii) is based on the inclusion $\xEN\subseteq\xENapp$, recall Eq.~\eqref{eq:gani_eff_equality} in Lemma~\ref{lem:FD_decomposition_E}~(i), and the following inequality
\begin{align*}
\scal{\BeffN^{-1}\VE}{\VE}_\xRd &= \min_{\MBe_\VN\in \xENapp} \bilfDN{\VE+\MBe_\VN}{\VE+\MBe_\VN} 
\\
&\leq \min_{\MBe_\VN\in \xEN} \bilfDN{\VE+\MBe_\VN}{\VE+\MBe_\VN} = \scal{\AeffN\VE}{\VE}_\xRd,
\end{align*}
holding for an arbitrary $\VE\in\xRd$.
\end{proof}

\section{Evaluation of upper-lower bounds on homogenized properties}
\label{sec:evaluation_bounds}
As the GaNi scheme \eqref{eq:GaNi}, or its fully discrete relative \eqref{eq:FD_GaNi}, deliver conforming approximations to the minimizers of the homogenization problem~\eqref{eq:homog_problem}, i.e. $\Ve_\VN\mac{\alp}\in\cE_\VN\subset\cE$ and $\Vj_\VN\mac{\alp}\in\cJ_\VN\subset\cJ$, they can be utilized within the upper-lower bounds structure of Section~\ref{sec:upper-lower_bounds_general}. Details of these developments are gathered here with the emphasis on the evaluation of the bounds in a computationally efficient way. Recall that the GaNi scheme is defined with the approximate bilinear forms $a_\VN$ and $a_\VN^{-1}$, \eqref{eq:bilinear_forms_approx}, whereas the upper-lower bounds are obtained via bilinear forms of the continuous homogenization problem \eqref{eq:homog_problem},
\begin{subequations}
\label{eq:GaNi_bounds}
\begin{align}
\scal{\ol{\TA}_{\eff,\VN}\VE}{\VE}_{\xRd} &= \bilf{\VE+\tVe_{\VN}\mac{\VE}}{\VE+\tVe_{\VN}\mac{\VE}},
\\
\scal{\ol{\TB}_{\eff,\VN}\VJ}{\VJ}_{\xRd} &= \bilfi{\VJ+\tVj_{\VN}\mac{\VJ}}{\VJ+\tVj_{\VN}\mac{\VJ}}
\end{align}
\end{subequations}
and the mean of guaranteed bounds $\oul{\TA}_{\eff,\VN}$ with the guaranteed error $\mD_\VN$ reads as
\begin{subequations}
\begin{align}
\label{eq:GaNi_bounds_mean}
\oul{\TA}_{\eff,\VN} &= \frac{1}{2} \left( \ol{\TA}_{\eff,\VN} + \ol{\TB}_{\eff,\VN}^{-1} \right),
\\
\label{eq:GaNi_error}
\mD_\VN &= \frac{1}{2} \left( \ol{\TA}_{\eff,\VN} - \ol{\TB}_{\eff,\VN}^{-1} \right).
\end{align}
\end{subequations}

For an easier orientation among the matrices, we refer to their scheme in Figure~\ref{fig:scheme_homog}; the inequality on the last line is proven in Propositions~\ref{lem:FD_odd_problem} and \ref{lem:FD_even_problem}.
Notice that the effective matrices $\AeffN$ and $\BeffN$ of the GaNi \eqref{eq:GaNi} or \eqref{eq:FD_GaNi} are generally in no relation, in the sense of the L\"{o}wner partial order, to the homogenized matrix $\Aeff$ and to a posteriori upper-lower bounds $\oAeffN$ and $\oBeffN$, as confirmed with numerical experiments in Section~\ref{sec:numerical_experiments}.
\begin{figure}[htp]
\centering
\includegraphics[scale=.8]{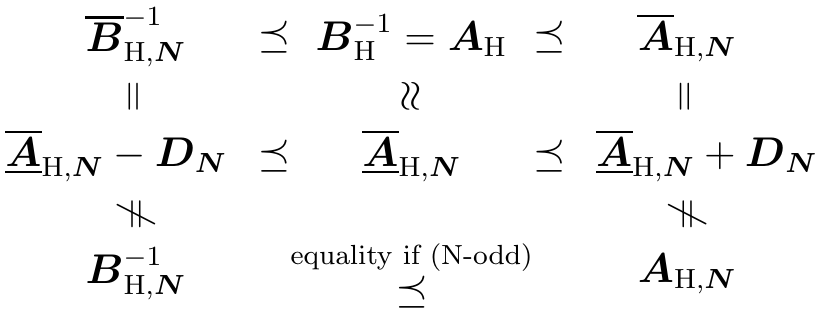}
\caption{Relations among homogenized matrices}
\label{fig:scheme_homog}
\end{figure}

Computation of the  bounds involves integrals of the type
\begin{align}\label{eq:integral2eval}
\scal{\TA \Vu_{\VN} }{ \Vv_{\VN} }
_{\Lper{2}{\xRd}}\quad\text{ for }\TA\in\Lper{\infty}{\xRdd}\text{ and }\Vu_\VN,\Vv_\VN\in\cTNd,
\end{align}
recall \eqref{eq:GaNi_bounds}. Notice that, due to the definition of spaces $\cEN$ and $\cJN$ in \eqref{eq:def_cUN_cEN_cJN}, the minimizers $\Ve_\VN\mac{\alp},\Vj_\VN\mac{\alp}$ always belong to $\cTNd$ \rev{}{that are defined on odd grids~\eqref{eq:N_odd}, recall~\eqref{eq:trig_space_Fourier}, we may consider only odd grids in this section without the loss of generality. This is because the integration now must be performed exactly, instead of approximately by the trapezoidal rule, in order to obtain guaranteed error bounds.}

\rev{}{In particular, we} show in Lemma~\ref{lem:Ga_full} that the term in \eqref{eq:integral2eval} can be evaluated in an analogous way to the GaNi, recall Corollary~\ref{lem:FD_GaNi}, but the resulting matrix becomes fully populated, rendering the estimates very costly.
Fortunately, we recover the block diagonal structure when defining the fully discrete quadratic forms on the double grid, Lemma~\ref{lem:integral_eval_double_grid}.
\begin{lemma}\label{lem:Ga_full}
For odd grids \eqref{eq:N_odd}, the integral \eqref{eq:integral2eval} equals to
\begin{align}\label{eq:integral2evalfull}
\scal{\TA\Vu_\VN}{\Vv_\VN}_{\Lper{2}{\xRd}} &= 
\scal{\hat{\MBA}_{\mathrm{full}}
\hat{\MBu}_\VN}{\hat{\MBv}_\VN}_{\xhXN} = \scal{\MBA_{\mathrm{full}}\MBu_\VN}{\MBv_\VN}_{\xXN},
\end{align}
where vectors $\MBu_\VN,\MBv_\VN\in\xXN$ and $\hat{\MBu}_\VN,\hat{\MBv}_\VN\in\xhXN$ are defined via
\begin{align*}
\MBu_\VN &= \IN[\Vu_\VN],
&
\MBv_\VN &= \IN[\Vv_\VN],
&
\hat{\MBu}_\VN &= \MBF_\VN\MBu_\VN,
&
\hat{\MBv}_\VN &= \MBF_\VN\MBv_\VN,
\end{align*}
and matrices $\hat{\MBA}_{\mathrm{full}}\in\xhMN$ and $\MBA_{\mathrm{full}}\in\xMN$ follow from
\begin{align*}
\bigl( \hat{\MBA}_{\mathrm{full}} \bigr)^{\Vl\Vk}
&= \frac{1}{\meas{\puc}} \int_{\puc} \TA(\Vx) \varphi_\Vk(\Vx)  \varphi_{-\Vl}(\Vx) \D{\Vx}
\quad\text{for }\Vk,\Vl\in\ZNd,
&
\MBA_{\mathrm{full}} &= \MBF_\VN \hat{\MBA}_{\mathrm{full}} \MBFi_\VN.
\end{align*}
\end{lemma}
\begin{proof}To obtain the first expression in \eqref{eq:integral2evalfull}, we represent the vectors in \eqref{eq:integral2eval} with their Fourier series $\Vu_{\VN} =
\sum_{\Vk\in\ZNd}\hat{\MBu}_\VN^\Vk\varphi_\Vk$, $\Vv_{\VN} =
\sum_{\Vl\in\ZNd}\hat{\MBv}_\VN^\Vl\varphi_\Vk$. Substitution into \eqref{eq:integral2eval} yields
\begin{align*}
\scal{\TA\Vu_\VN}{\Vv_\VN}_{L^2_{\per}} 
&= \sum_{\alp,\beta} \sum_{\Vk,\Vl\in\ZNd}
\frac{\hat{\MBu}_{\VN,\beta}^{\Vk} \hat{\MBv}_{\VN,\alp}^{\Vl}}{\meas{\puc}} \int_{\puc} \TA_{\alp\beta} \varphi_\Vk  \varphi_\Vl \D{\Vx} = \scal{\hat{\MBA}_{\mathrm{full}}
\hat{\MBu}_{\VN}}{\hat{\MBv}_{\VN}}_{\xhXN}.
\end{align*}
To obtain the last expression in \eqref{eq:integral2evalfull}, we map the Fourier coefficients with DFT matrix \eqref{eq:DFT} to obtain $\hat{\MBu}_\VN = \MBF_\VN\MBu_\VN$ and $\hat{\MBv}_\VN = \MBF_\VN\MBv_\VN$, cf. \eqref{eq:trig_connection}, 
from which we calculate
\begin{align*}
\scal{\hat{\MBA}_{\mathrm{full}} \hat{\MBu}_{\VN}}{\hat{\MBv}_{\VN}}_{\xhXN} &=
\scal{\hat{\MBA}_{\mathrm{full}} \MBF_\VN\MBu_{\VN}}{\MBF_\VN\MBv_{\VN}}_{\xhXN} =
\frac{1}{\pVN}\scal{\MBFi_\VN\hat{\MBA}_{\mathrm{full}} \MBF_\VN\MBu_{\VN}}{\MBv_{\VN}}_{\xhXN}
\\
&= \scal{\MBFi_\VN\hat{\MBA}_{\mathrm{full}} \MBF_\VN\MBu_{\VN}}{\MBv_{\VN}}_{\xXN},
\end{align*}
where we have utilized $\MBF_\VN^\star = \frac{1}{\pVN}\MBF_\VN^{-1}$.
\end{proof}
\begin{remark}\label{not:sub_N}
The sparse quadrature involves a projection to a finer grid denoted as 
\begin{align*}
\MBu_\VN = \IM[\Vu_\VN] \in \xR^{d\times\VM}
\quad\text{for }
\VM,\VN\in\xRd
\text{ such that }
M_\alp>N_\alp.
\end{align*}
Here, we decided to use the same subscript $\VN$ for the trigonometric polynomial $\Vu_\VN$ and its discrete representation $\MBu_\VN$ in order to highlight their polynomial degree and to avoid a profusion of notation. The actual dimension of $\MBu_\VN$ is understood implicitly from the context, so that the terms like $\scal{\MBA_\VM\MBu_\VN}{\MBu_\VN}_{\xR^{d\times\VM}}$ with $\MBA_\VM\in\bigl[\xR^{d\times\VM}\bigr]^2$ remain well-defined. 
\end{remark}
\begin{lemma}[Double-grid quadrature]
\label{lem:integral_eval_double_grid}
For odd grids \eqref{eq:N_odd}, the integral \eqref{eq:integral2eval} equals to
\begin{align*}
\scal{\TA\Vu_\VN}{\Vv_\VN}_{L^2_{\per}} &=
\scal{\MBA_\tVN\MBu_\VN}{\MBv_\VN}
_{\xRdtN}
\end{align*}
where $\MBu_\VN = \ItN[\Vu_\VN]$, $\MBv_\VN=\ItN[\Vu_\VN]\in\xXtN$, and
$\MBA_\tVN\in\xMtN$ has the components
\begin{align}\label{eq:MBA_GA}
\MBA_\tVN^{\Vk \Vm} &=
\del_{\Vk\Vm}\sum_{\Vn\in\ZNNd}
\omega_{\tVN}^{\Vk\Vn}
\hat{\TA}(\Vn)\in\xRdd.
\end{align}
\end{lemma}
\begin{proof}Because the product of two trigonometric polynomials $\Vu_\VN\Vv_{\VN}\hide{=\bigl(u_{\VN,\alp}v_{\VN,\alp}\bigr)_{\alp}}\in\cTtNd$  has bounded frequencies, we can express it as
$$u_{\VN,\beta}v_{\VN,\alp} = \sum_{\Vk\in\ZNNd} u_{\VN,\beta}(\Vx_\tVN^\Vk)
v_{\VN,\alp}(\Vx_\tVN^\Vk) \varphi_{\tVN,\Vk} = \sum_{\Vk\in\ZNNd} \M{u}_{\VN,\beta}^\Vk
\M{v}_{\VN,\alp}^\Vk \varphi_{\tVN,\Vk}.$$
Substitution into \eqref{eq:integral2eval} and direct calculations reveal
\begin{align*}
\scal{\TA\Vu_\VN}{\Vv_\VN}_{L^2_{\per}} 
&= \sum_{\alp,\beta}
\sum_{\Vk\in\ZNNd} \M{u}_{\VN,\beta}^\Vk \M{v}_{\VN,\alp}^\Vk \frac{1}{\meas{\puc}}\int_\puc A_{\alp\beta}(\Vx)\varphi_{\tVN,\Vk}(\Vx)\D{\Vx}
\\
&= \sum_{\alp,\beta} \sum_{\Vk\in\ZNNd} \left( \sum_{\Vl\in\ZNNd}
\frac{\omega_{\tVN}^{-\Vk\Vl}}{\ptVN \meas{\puc}}
\int_\puc A_{\alp\beta}(\Vx)\varphi_{\Vl}(\Vx)\D{\Vx} 
\right)
\M{u}_{\VN,\beta}^\Vk \M{v}_{\VN,\alp}^\Vk.
\end{align*}
The statement of the lemma follows by substitution of $\Vl$ with $-\Vn$.
\end{proof}
To evaluate the matrix in \eqref{eq:MBA_GA}, we need to determine the Fourier coefficients $[\hat{A}_{\alp\beta}(\Vn)]^{\Vn\in\ZNNd}$.
In the present section, these are elaborated in detail for the matrix-inclusion composites, characterized by the coefficients in the form
\begin{align}\label{eq:A_with_inclusions}
\TA(\Vx) &= \TA\incl{0}+\sum_{j=1}^J f\incl{j} (\Vx-\Vx\incl{j}) \TA\incl{j}
\end{align}
where $\TA\incl{0}\in\xRdd$ represents the  coefficients of the matrix phase, matrices $\TA\incl{j}\in\xRdd$ with functions $f\incl{j}\in L^{\infty}_\per(\puc)$ for $j=0,\dotsc,J$
quantify the distribution of coefficients within inclusions, centered at $\Vx\incl{j}$, along with their geometry~(in short, the functions $f\incl{j}$ will be referred to as inclusion topologies).
\begin{lemma}\label{lem:bounds_for_matrix-inclusion}
The matrix
\eqref{eq:MBA_GA} for coefficients \eqref{eq:A_with_inclusions} is given by
\begin{align}\label{eq:doubleFD_GA_single_incl}
\MBA_{\tVN}^{\Vk\Vm} 
&= \del_{\Vk\Vm} \left[ \TA\incl{0}  + 
\sum_{j=1}^J \TA\incl{j} 
\left( \sum_{\Vn\in\ZtNd} \omega_{\tVN}^{\Vk\Vn}
\varphi_{-\Vn}(\Vx\incl{j})
\hat{f}\incl{j}(\Vn) \right)
\right]\in\xRdd,
\end{align}
where $\hat{f}\incl{j}(\Vn)$ for $j\in\{1,\dotsc,J\}$ and $\Vn\in\ZNd$ denote the Fourier coefficients \eqref{eq:FT_def} of inclusion topologies $f\incl{j}$.
\end{lemma}
\begin{proof}
Using basic properties of the Fourier trigonometric polynomials, namely
$\int_{\puc}\varphi_\Vn(\Vx)\D{\Vx} = \meas{\puc}\del_{\V{0}\Vn}$ and
$\varphi_\Vn(\Vx + \Vx\incl{j}) = \varphi_\Vn(\Vx) \varphi_\Vn(\Vx\incl{j})$,
we deduce
\begin{align*}
\MBA_{\tVN}^{\Vk\Vm} &= \del_{\Vk\Vm} \sum_{\Vn\in\ZNNd}
\omega_{\tVN}^{\Vk\Vn}
\hat{\TA}(\Vn)
\\
&= \del_{\Vk\Vm} \sum_{\Vn\in\ZtNd}
\frac{\omega_{\tVN}^{\Vk\Vn}}{\meas{\puc}} \int_{\puc} \left[
\TA\incl{0} + \sum_{j=1}^J
\TA\incl{j}f\incl{j}(\Vx-\Vx\incl{j})\right]\varphi_{-\Vn}(\Vx) \D{\Vx}
\\
&= \del_{\Vk\Vm} \left[ \TA\incl{0}  + \sum_{j=1}^J
\TA\incl{j} \sum_{\Vn\in\ZtNd}
\frac{\omega_{\tVN}^{\Vk\Vn}}{\meas{\puc}} \int_\puc f\incl{j}(\Vx)
\varphi_{-\Vn}(\Vx) \varphi_{-\Vn}(\Vx\incl{j}) \D{\Vx} \right].
\end{align*}
\end{proof}
\begin{remark}\label{rem:inclusion}
An example of the inclusion topology from \eqref{eq:A_with_inclusions} is provided by a rectangle/cuboid of side lengths $0<h_\alp\leq Y_\alp$ centered at the origin, i.e.
\begin{align*}
\rect_{\Vh}(\Vx) &=
\begin{cases}
1&\text{if }|x_\alp|<\frac{h_\alp}{2}\text{ for all }\alp
\\
0&\text{otherwise}
\end{cases},
&
\widehat{\rect}_{\Vh}(\Vm) &
= \frac{1}{\meas{\puc}} \prod_{\alp} h_{\alp}  \sinc\left(
\frac{h_\alpha m_{\alp}}{Y_{\alp}}\right),
\end{align*}
where 
\begin{align*}
 \sinc (x) = 
\begin{cases}
  1&\text{for } x=0
\\
\frac{\sin(\pi x)}{\pi x}&\text{for }x\neq 0
\end{cases}.
\end{align*}
This topology is utilized in numerical examples in Section~\ref{sec:numerical_experiments} and corresponds to pixel or voxel-wise definition of material coefficients, which are commonly produced by imaging techniques such as tomography or microscopy.
However, \rev{}{a straightforward implementation based on~\eqref{eq:integral2eval} and \eqref{eq:A_with_inclusions} is computationally expensive, but its efficiency can be substantially increased by FFT algorithm as explained in \cite{Vondrejc2015FFTimproved}.}
Other examples of inclusion topologies, such as spherical and bilinear, can be found in \cite[pages~137--138]{Vondrejc2013PhD}.
\end{remark}
\begin{remark}[Types of numerical integration]The trapezoidal integration used in GaNi scheme \eqref{eq:GaNi} leads to the algorithm defined by Moulinec and Suquet \cite{Moulinec1994FFT}. 
In \cite[Section~13.3.2]{Nemat-Nasser1999}, the exact integration formula leading to the fully populated matrix according to Lemma~\ref{lem:Ga_full} was  used for the Hashin-Shtrikman functional with piece-wise constant material coefficients. Later, the Fourier coefficients of individual inclusions have been incorporated as the so-called shape functions in~\cite{Bonnet2007,Monchiet2012polarization,Monchiet2013conduct} to enhance FFT-based homogenization schemes. Our results thus explain their good performance and introduce the numerical quadrature on double grid even in a more general setting.
\hide{also in Remark~\ref{rem:GaNi_vs_Ga}}
\end{remark}

\section{Computational aspects}
\label{sec:computational_issues}
Here, we discuss computational aspects related to the determination of upper-lower bounds. Section~\ref{sec:CG_solution} deals with the calculation of minimizers by the Conjugate gradients algorithms, while Section~\ref{sec:algorithm} gathers remarks on algorithm development and implementation issues.
\subsection{Conjugate gradients}
\label{sec:CG_solution}
Restricting our attention to the primal problem \eqref{eq:FD_GaNi_prim}, we are left with the minimization of a quadratic function over a subspace
\begin{align}\label{eq:CG_minimization}
\MBe\mac{\VE}_\VN = \argmin_{\MBe_\VN\in\xEN}\bilfDN{\VE+\MBe_\VN}{\VE+\MBe_\VN}.
\end{align}
This problem is suitable for the Conjugate Gradients~(CG) method, as it involves symmetric and positive definite forms.

According to \cite[Section~5.3]{VoZeMa2014FFTH}, the  problem \eqref{eq:CG_minimization} is equivalent to the solution of a linear system. Indeed, the minimizer satisfy the stationarity condition
\begin{align*}
\bilfDN{\MBe\mac{\VE}_\VN}{\MBv} = -\bilfDN{\VE}{\MBv}\quad\forall\MBv\in\xEN.
\end{align*}
Using $\MBG_{\VN,\Vo}\sub{\cE}$, an orthogonal (symmetric) projection on $\xEN$ from Definition~\ref{def:FD_proj_E}, we proceed to
\begin{align*}
\bilfDN{\MBe\mac{\VE}_\VN}{\MBG_{\VN,\Vo}\sub{\cE}\MBv} &= -\bilfDN{\VE}{\MBG_{\VN,\Vo}\sub{\cE}\MBv}
\quad\forall\MBv\in\xRdN,
\\
\scal{\MBG_{\VN,\Vo}\sub{\cE} \MBA_\VN \MBe\mac{\VE}_\VN}{\MBv}_{\xRdN} &= -\scal{\MBG_{\VN,\Vo}\sub{\cE} \MBA_\VN \VE}{\MBv}_{\xRdN}
\quad\forall\MBv\in\xRdN.
\end{align*}
Because the space of test functions was enlarged to $\xRdN$, we pass to a linear system
\begin{align*}
\underbrace{\MBG_{\VN,\Vo}\sub{\cE}\MBA_\VN}_{\MB{C}}\underbrace{\MBe\mac{\VE}_\VN}_{\MB{x}} &= \underbrace{-\MBG_{\VN,\Vo}\sub{\cE}\MBA_\VN\VE}_{\MB{b}}
\quad\text{for }
\MBe\mac{\VE}_\VN\in\xEN
\end{align*}
with $\MBA_\VN$ defined in \eqref{eq:FD_A_GaNi}. Thus, the minimization of \eqref{eq:CG_minimization} can be performed by CG applied to the linear system
\begin{align}\label{eq:linear_system_primal}
 \MB{C}\MB{x} = \MB{b}\quad\text{for }
 \MB{C}=\MBFi_\VN\MBhG\sub{\cE}\MBF_\VN\MBA_\VN 
\end{align}
with an initial approximation $\MB{x}_{(0)}\in\xEN$, \cite{VoZeMa2012LNSC}. 
By analogous arguments, the minimizers of the dual problem \eqref{eq:FD_GaNi_dual} satisfy the linear systems
\begin{align}\label{eq:linear_system_dual}
\underbrace{\MBG_{\VN,\Vo}\sub{\cJ}\MBA_\VN^{-1}}_{\MB{C}}\underbrace{\MBj\mac{\VJ}_\VN}_{\MB{x}} &= \underbrace{-\MBG_{\VN,\Vo}\sub{\cJ}\MBA_\VN^{-1}\VJ}_{\MB{b}}
\quad\text{for }
\MBj\mac{\VJ}_\VN\in\xJN
\end{align}
that are solvable by CG
with an initial approximation $\MB{x}_{(0)}\in\xJN$. 
\rev{}{\begin{remark}
Because of the involvement of the projection $\MBG_{\VN,\Vo}\sub{\cE}$, the matrix $\MB{C}$ in \eqref{eq:linear_system_primal} is singular on $\xRdN$, but regular on a subspace $\xEN$. The convergence of CG is ensured by~\cite[Lemma~13]{VoZeMa2014FFTH}, showing that the Krylov spaces generated by an arbitrary $\MB{x}_{(0)}\in\xEN$ remain in $\xEN$, and so are all the iterates generated by CG. The analogous arguments hold also for the discrete dual formulation~\eqref{eq:linear_system_dual}.
\end{remark}}
\subsection{Implementation issues}
\label{sec:algorithm}
\begin{algorithm}
For coefficients $\TA\in L^{\infty}_{\mathrm{per}}(\puc;\xR^{d\times d})$, the evaluation of upper-lower bounds on homogenized matrix consists of the following steps.
\begin{enumerate}
\item Set the number of grid points $\VN$ and assemble matrices $\MBA_\VN$, $\MBA_\VN^{-1}$, $\MBhG_{\Vo}\sub{\cE}$, $\MBhG_{\Vo}\sub{\cJ}\in\xMN$ according to Definition~\ref{def:FD_projections} and Eq.~\eqref{eq:FD_A_GaNi}.
\item For $\alp=1,\dotsc,d$, find discrete primal and dual minimizers $\MBe^{(\alp)}_\VN\in\xEN$, $\MBj^{(\alp)}_\VN\in\xJN$ as solutions to linear systems
\eqref{eq:linear_system_primal} and \eqref{eq:linear_system_dual} 
for $\VE = \VJ = \cb{\alp}$.
\item Evaluate upper-lower bounds \eqref{eq:GaNi_bounds} according to Lemmas~\ref{lem:integral_eval_double_grid} and~\ref{lem:bounds_for_matrix-inclusion}.
\end{enumerate}
\end{algorithm}
\begin{remark}
The matrices in step (i) are block diagonal leading to a substantial reduction in memory requirements. In step (ii), the solution of linear systems requires only matrix-vector multiplications involving sequential application of matrices $\MBA_\VN,\MBF_\VN,\MBhG\sub{\cE}$, and $\MBFi_\VN$. The computational cost is dominated by \rev{}{multiplications with} DFT matrices $\MBF_\VN$ and $\MBFi_\VN$ that are performed only in $\mathcal{O}(|\VN|\log|\VN|)$ operations by the FFT algorithm.
\end{remark}
\begin{remark}[Convergence criteria]
Regarding step (ii), initial approximations to CG are set to the zero vector and the convergence criterion is based on the norm of residuum, i.e.  $\|\MB{r}_{(i)}\|_{\xRdN}\leq \ep \|\VE\|_2$ with  $\MB{r}_{(i)}=\rev{}{ -\MBG_{\VN,\Vo}\sub{\cE}\MBA_\VN(\MBx_{(i)}+\VE)}$ and $\MBx_{(i)}$ denoting $i$-th iterate.
The tolerance is set to $\ep=10^{-8}$ in order to ensure that the overall error is dominated by the discretization error instead of the algebraic one.
The norm for residuum $\|\MB{r}_{(i)}\|_{\xRdN}$, due to Parseval's theorem, equals to $\|\INi[\MB{r}_{(i)}]\|_{\Lper{2}{\xRd}}$, the $L^2_\per$-norm of corresponding trigonometric polynomial. The dual case is treated in an analogous way.
\end{remark}
\begin{remark}[Divergence-free convergence criterion]
\label{rem:div_norm}
The most commonly used termination criterion in FFT-based algorithms is based on the divergence-free condition for the dual fields, $\MBA_\VN\MBe_\VN\mac{\alp}\in \xJN$ with $\MBA_\VN$ from \eqref{eq:FD_A_GaNi}, \cite{Moulinec1994FFT,Moulinec1998NMC,Michel2000CMB,Moulinec2014comparison}. Our analysis reveals that this criterion is reasonable only for the odd grids~\eqref{eq:N_odd}, namely
\begin{align*}
\MBe_\VN\mac{\alp}\in\xEN \Longleftrightarrow \MBA_\VN\MBe_\VN\mac{\alp}\in\xJN,
\end{align*}
cf.~Proposition~\ref{lem:FD_odd_problem}.
Such property is lost for general grids when either minimizers or dual fields are conforming only up to the Nyquist frequencies $\Vk\in\ZNd\setminus\ZNdr$, so that
\begin{align*}
\MBe_\VN\mac{\alp}\in\xEN 
\Longrightarrow
\MBA_\VN\MBe_\VN\mac{\alp}\in\xJNapp,
&&
\text{or}
&&
\tMBe_\VN\mac{\alp}\in\xENapp
\Longrightarrow
\MBA_\VN \tMBe_\VN\mac{\alp} \in \xJN,
\end{align*}
recall Proposition~\ref{lem:FD_even_problem}. This observation is in agreement with \cite[Section~2.4.2]{Moulinec1998NMC}, where the projection $\MBG_{\VN,\mI}^{\cE}$ from~Definition~\ref{def:FD_proj_E} was utilized to obtain divergence-free fields.
\end{remark}
\begin{remark}
The matrix \eqref{eq:doubleFD_GA_single_incl}  needed in step (ii) can be assembled in an efficient way. The Fourier coefficients \eqref{eq:FT_def} of each inclusion topology  $[\hat{f}\incl{j}(\Vm)]^{\Vm\in\ZtNd}$ for $j=1,\dotsc,J$ are evaluated in the closed form and shifted by distance $\Vx\incl{j}$ to account for its position; the shift corresponds to element-wise multiplication by the matrix $
[\varphi_{-\Vm}(\Vx\incl{j})]^{\Vm\in\ZtNd}$. The sum over $\Vk\in\ZtNd$ can be performed with the FFT algorithm.
\end{remark}
\begin{remark}[Avoiding the solution of dual formulation]
For odd grids \eqref{eq:N_odd}, the dual discrete minimizers $\MBj^{(\alp)}_\VN$ can be obtained from Eq.~\eqref{eq:FD_superposition_odd_A} if the original minimizers $\MBe^{(\alp)}_\VN$ are the exact solutions to the corresponding linear systems, see Section~\ref{sec:CG_solution}.
In reality, the linear systems are solved only approximately, so that $\MBj^{(\alp)}_\VN\notin\xJN$. This non-conformity can be corrected by the projection operator $\MBG_{\Vo}\sub{\cJ}$ and, when $\MBA_\VN$ is badly conditioned, by performing several CG iterations for the dual formulation, recall \eqref{eq:FD_GaNi_dual} and \eqref{eq:linear_system_dual}.
\end{remark}
\begin{remark}[Arbitrary accurate bounds]
\label{rem:conv_eN2e}
\rev{}{Thanks to the estimates on guaranteed error \eqref{eq:estimate_guaranteed_error} and the convergence results \cite{Vondrejc2013PhD,VoZeMa2014FFTH,Schneider2014convergence} discussed in Remark~\ref{rem:gani_solvability}, the two-sided bounds on homogenized properties can be made arbitrarily accurate for sufficiently fine discretizations.}
\end{remark}

\section{Numerical experiments}
\label{sec:numerical_experiments}
This section is dedicated to numerical experiments supporting our theoretical
results, especially on the primal-dual structure and convergence of homogenized
matrices. The calculations in
Sections~\ref{sec:results_homogenized_coef_odd}
and~\ref{sec:results_homogenized_coef_even} are performed on a
two-dimensional cell with a square inclusion first, in order to demonstrate 
the difference between odd and non-odd discretization grids and to study the
behavior of upper-lower bounds as a function of grid spacing and contrast in coefficients.
Section~\ref{sec:engineering_example} deals with the determination of
effective thermal conductivity of an alkali-activated fly ash foam described
with a high-resolution bitmap. All results in this section were obtained with
an open-source Python library FFTHomPy available at \url{https://github.com/vondrejc/FFTHomPy}.

In Sections~\ref{sec:results_homogenized_coef_odd} and~\ref{sec:results_homogenized_coef_even}, we consider problems with coefficients defined on the periodic cell $\puc = (-1,1)\times(-1,1) \subset \xR^2$ via
\begin{align*}
\TA(\Vx) = [1 + \rho f(\Vx)] \mI \quad\text{for }\Vx\in\puc,
\end{align*}
where $\mI\in\xR^{2\times 2}$ is the identity matrix, $f: \puc\rightarrow \xR$ is the topology function introduced in Remark~\ref{rem:inclusion}, and $\rho\in\{10,10^3\}$ is the phase contrast. Three types of square inclusions are considered, namely
\begin{subequations}
\begin{align}
\label{eq:S}
f(\Vx) &= 
\begin{cases}
1&\text{if }|x_\alpha|<\frac{3}{5}\text{ for all }\alpha\\
0&\text{otherwise}
\end{cases},
\\
\label{eq:S1}
f(\Vx) &= 
\begin{cases}
1&\text{if }|x_\alpha|<\frac{3}{4}\text{ for all }\alpha\\
0&\text{otherwise}
\end{cases},
\\
\label{eq:S2}
f(\Vx) &= 
\begin{cases}
1&\text{if }|x_\alpha|\leq\frac{3}{4}\text{ for all }\alpha \\
0&\text{otherwise}
\end{cases}.
\end{align}
\end{subequations}
The square \eqref{eq:S} is discretized with odd number of points $\VN = (n,n)$ for $n\in\{5\cdot 3^j:j = 0,1,\dotsc,6\}$, see Figure~\ref{fig:square_S}, while squares \eqref{eq:S1} and \eqref{eq:S2} with even number of points, $n\in\{2^j:j=2,3,\dotsc,10\}$, Figure~\ref{fig:square_S12}. Because \rev{}{all inclusions are symmetric with respect to the origin and the material phases are isotropic, the homogenized matrices are proportional to identity $\mI$ and only one diagonal component needs to be plotted in what follows}.

\begin{figure}[htp]
\centering
\subfigure[\scriptsize Topology \eqref{eq:S} and odd grids]{
\includegraphics[scale=0.6]{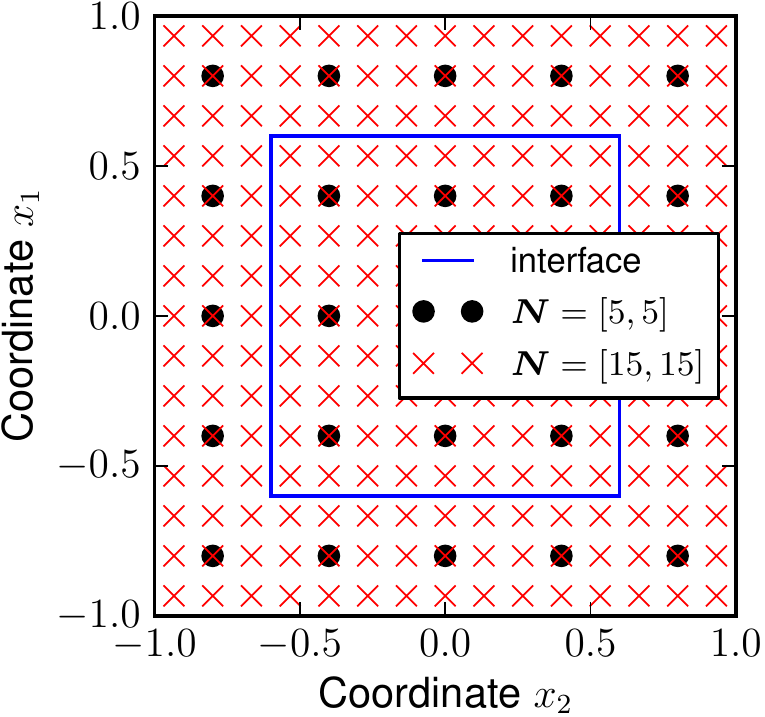}
\label{fig:square_S}}
\subfigure[\scriptsize Topologies \eqref{eq:S1}, \eqref{eq:S2} and even grids]{
\includegraphics[scale=0.6]{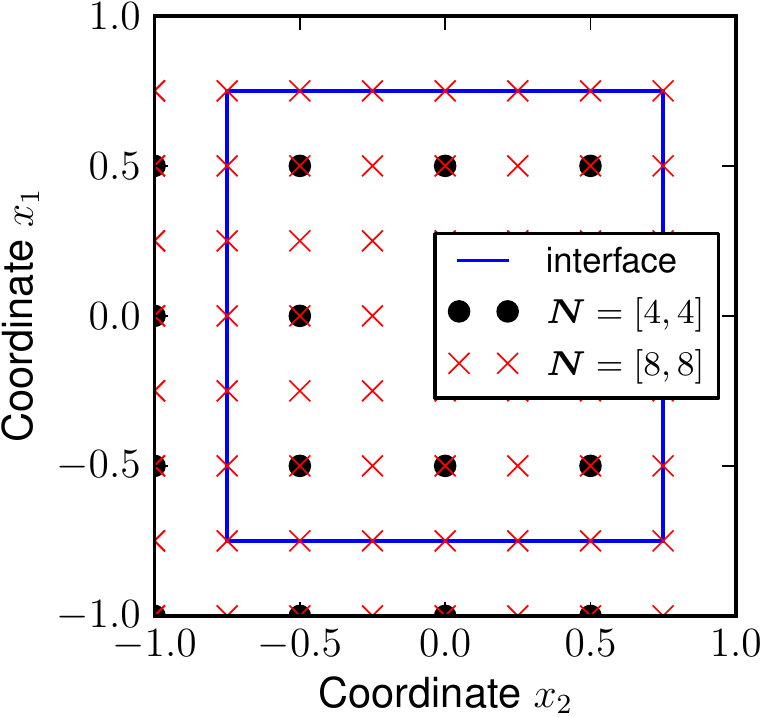}
\label{fig:square_S12}}
\caption{Cells with odd and even number of grid points}
\label{fig:square_SS12}
\end{figure}
\subsection{Homogenized matrices for odd discretization}
\label{sec:results_homogenized_coef_odd}
For odd grids \eqref{eq:N_odd}, the approximate homogenized matrices $\AeffN,\BeffN$ calculated from GaNi, recall \eqref{eq:GaNi}, are mutually inverse $\AeffN = \BeffN^{-1}$ as stated in Proposition~\ref{lem:FD_odd_problem}. 
The inequality $\oBeffN^{-1}\preceq\oAeffN$ of upper-lower bounds, stated in Lemma~\ref{lem:bounds}~(i), is satisfied  and the guaranteed error \eqref{eq:GaNi_error} converges to zero according to Lemma~\ref{lem:estimates} \rev{}{and Remark~\ref{rem:conv_eN2e}}. 
By the same arguments, the approximate homogenized matrices $\AeffN=\BeffN^{-1}$ from GaNi and the mean of guaranteed bounds $\ouAeffN$, \eqref{eq:GaNi_bounds_mean}, converge to $\Aeff$.
Since the inclusion shape is sampled with the grid points accurately, matrices $\AeffN=\BeffN^{-1}$ from GaNi approximate the homogenized properties better than the mean of guaranteed bounds $\ouAeffN$, especially for a small number of grid points.
\begin{figure}[htp]
\centering
\subfigure[\scriptsize Phase contrast $\rho = 10$]{
\includegraphics[scale=0.6]{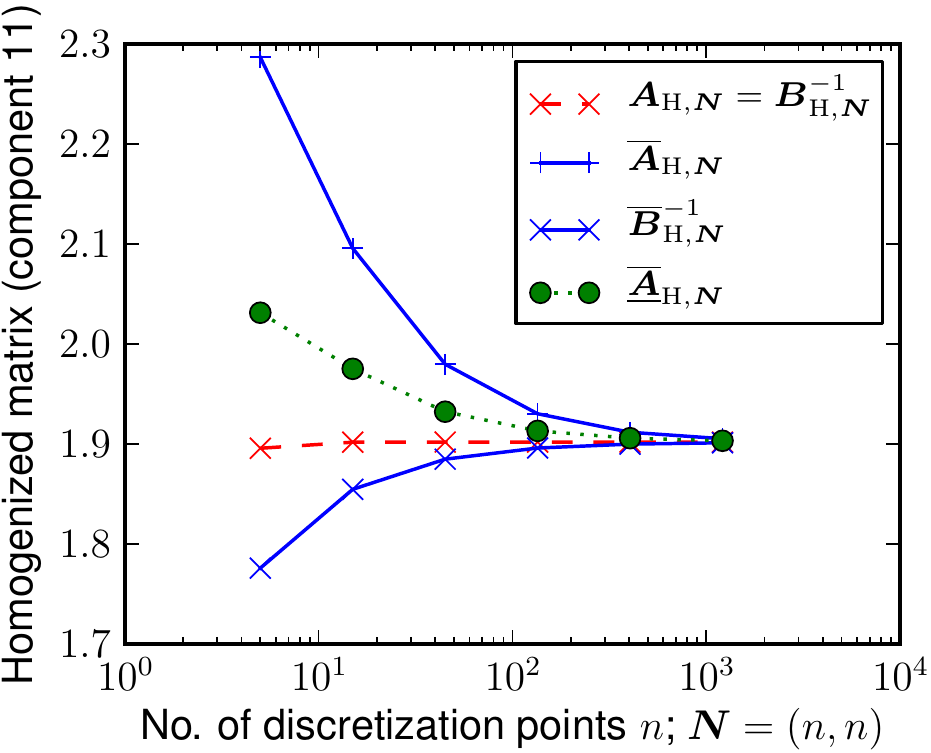}
}
\subfigure[\scriptsize Phase contrast $\rho = 10^3$]{
\includegraphics[scale=0.6]{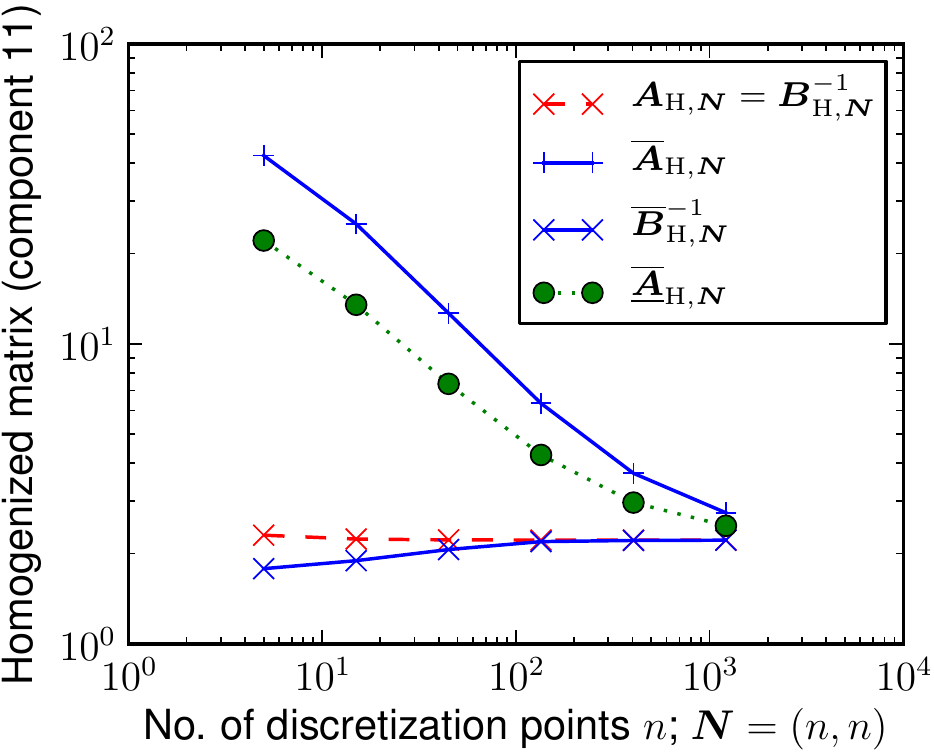}
}
\caption{Homogenized matrices for cell \eqref{eq:S} and odd grids}
\label{fig:odd_discretization}
\end{figure}

In Figure~\ref{fig:convergence}, we plot analogous results to Figure~\ref{fig:odd_discretization} for a refined sequence of grid points $\VN=(n,n)$ with $n\in\{5,7,9,\dotsc,145\}$. The results reveal that the convergence of guaranteed error \eqref{eq:GaNi_error}, Remark~\ref{rem:conv_eN2e}, is not monotone with an increasing number of grid points, despite the hierarchy of approximation spaces
\begin{align}\label{eq:two_grids}
\cEN \subseteq\cE_{\VM}\subset\cE\text{ and }\cJN \subseteq\cJ_{\VM}\subset\cJ\quad\text{for }N_\alp\leq M_\alp;
\end{align}
\rev{}{following from the fact that an increase in $\VN$ adds new basis functions into $\cTNd$, see~\eqref{eq:trig_space_Fourier}, similarly to the $p$-version of FEM.}
We attribute this behavior to the numerical
integration in approximate bilinear forms $\bilfN{}{}$ and $\bilfNi{}{}$
in~\eqref{eq:bilinear_forms_approx}\rev{}{, see the discussion in Remark~\ref{rem:gani_solvability}}, so that the solutions corresponding to two
discretizations $\V{N}$ and $\V{M}$ from~\eqref{eq:two_grids} are determined for
different sampling of material coefficients $\V{A}$. This ``variational
crime''~\cite{Strang1972varcrime} \rev{}{due to the inconsistency error \eqref{eq:consist_error}} results in the non-monotonous convergence of
the approximate solutions; their convergence is nevertheless assured
by~\cite[Proposition~8]{VoZeMa2014FFTH}. Moreover, no oscillations have been
observed for the Galerkin method without numerical
integration~\cite{Vondrejc2015FFTimproved}.
\begin{figure}[htp]
\centering
\subfigure[\scriptsize Phase contrast $\rho = 10$]{
\includegraphics[scale=0.6]{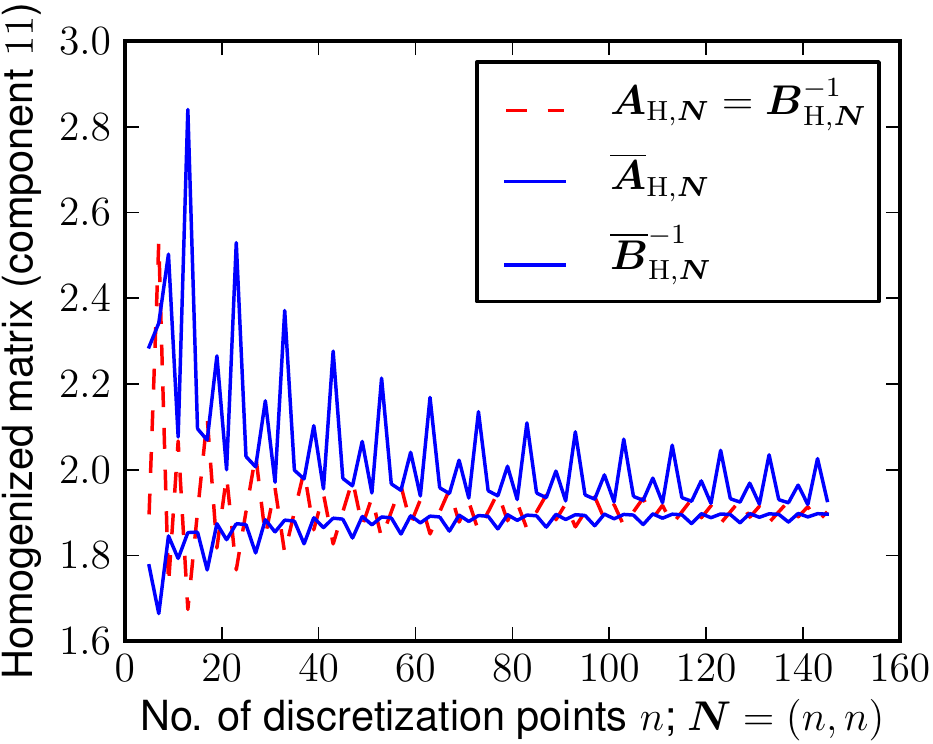}
}
\subfigure[\scriptsize Phase contrast $\rho = 10^3$]{
\includegraphics[scale=0.6]{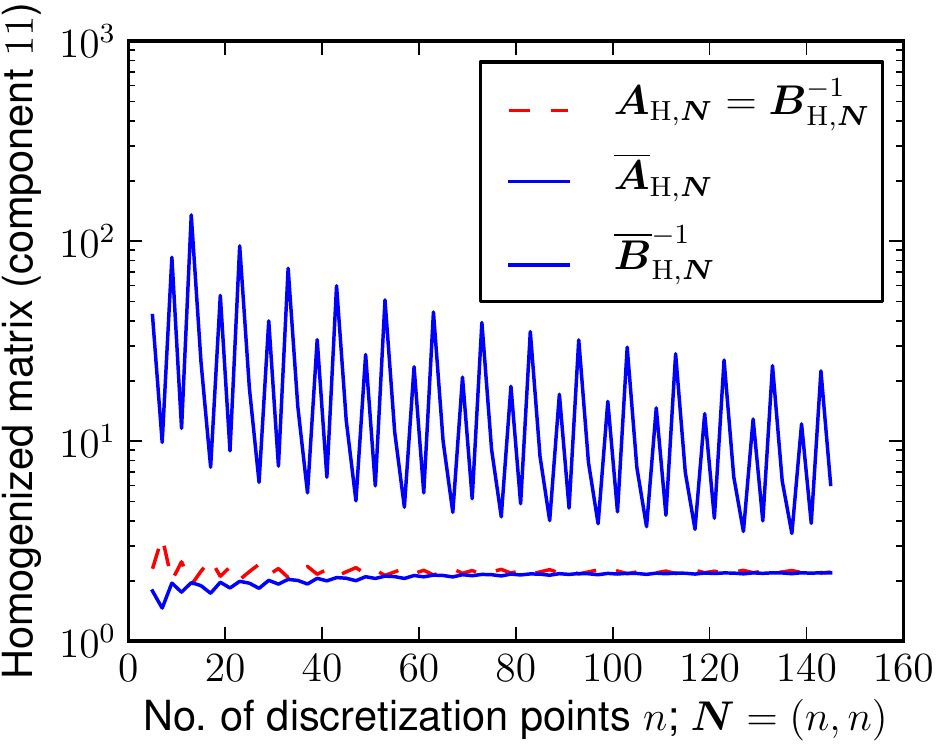}
}
\caption{Homogenized matrices for topology \eqref{eq:S} and a refined sequence of odd grids}
\label{fig:convergence}
\end{figure}

\subsection{Homogenized matrices for even discretization}
\label{sec:results_homogenized_coef_even}
\rev{}{Before presenting the results for even discretizations with two phase contrasts $\rho\in\{10,10^3\}$, we first clarify the reason for using the two inclusion topologies~\eqref{eq:S1} and~\eqref{eq:S2} that --- being different only at the matrix-inclusion interface --- are indistinguishable in homogenized properties and in the distribution of local fields. This no longer holds for the GaNi-based discretizations, once some of the grid points are located exactly at the interface, recall Figure~\ref{fig:square_SS12}(b). As a consequence of the trapezoidal integration rule, these points become associated with the coefficients of the matrix phase, \eqref{eq:S1}, or inclusion, \eqref{eq:S2}, rendering the two approximate solutions different. In addition, we expect these effects to be further amplified with the well-known Gibbs oscillations along the material interfaces, e.g.~\cite{Brisard2010FFT,willot2013fourier,willot_fourier-based_2015}, intrinsic to trigonometric polynomials-based approximations.}
\begin{figure}[htp]
\centering
\subfigure[\scriptsize Cell \eqref{eq:S1}]{
\includegraphics[scale=0.6]{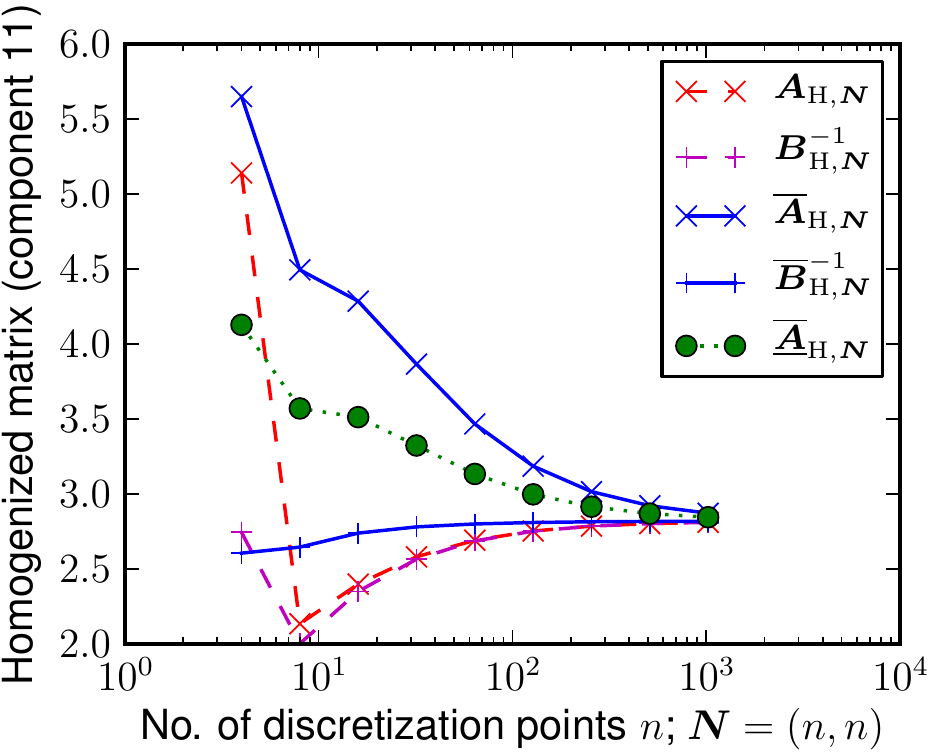}
}
\subfigure[\scriptsize Cell \eqref{eq:S2}]{
\includegraphics[scale=0.6]{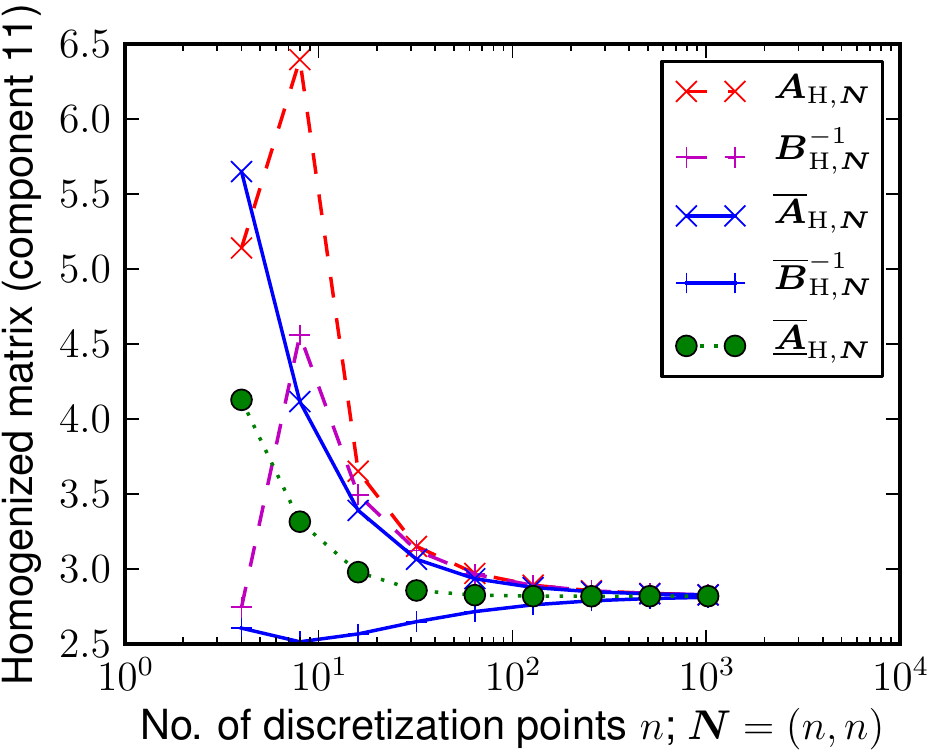}
}
\caption{Homogenized matrices for cells \eqref{eq:S1} and \eqref{eq:S2}, even grids, and phase contrast $\rho = 10$}
\label{fig:even_r10}
\end{figure}

In particular, Figures~\ref{fig:even_r10} and~\ref{fig:even_r1000} show that the
approximate homogenized matrices $\AeffN$ and $\BeffN^{-1}$ from GaNi are
different for even grids, nevertheless they still satisfy $\BeffN^{-1}\preceq\AeffN$, in agreement with Theorem~\ref{lem:FD_even_problem}.
Moreover, the duality
gap decreases as the effect of the Nyquist frequencies
diminishes with an increasing number of grid points,
and both matrices converge to the homogenized matrix $\Aeff$. The
same holds for the upper-lower bounds $\oAeffN$, $\oBeffN^{-1}$
and their mean $\ouAeffN$.
\begin{figure}[htp]
\centering
\subfigure[\scriptsize Cell \eqref{eq:S1}]
{
\includegraphics[scale=0.6]{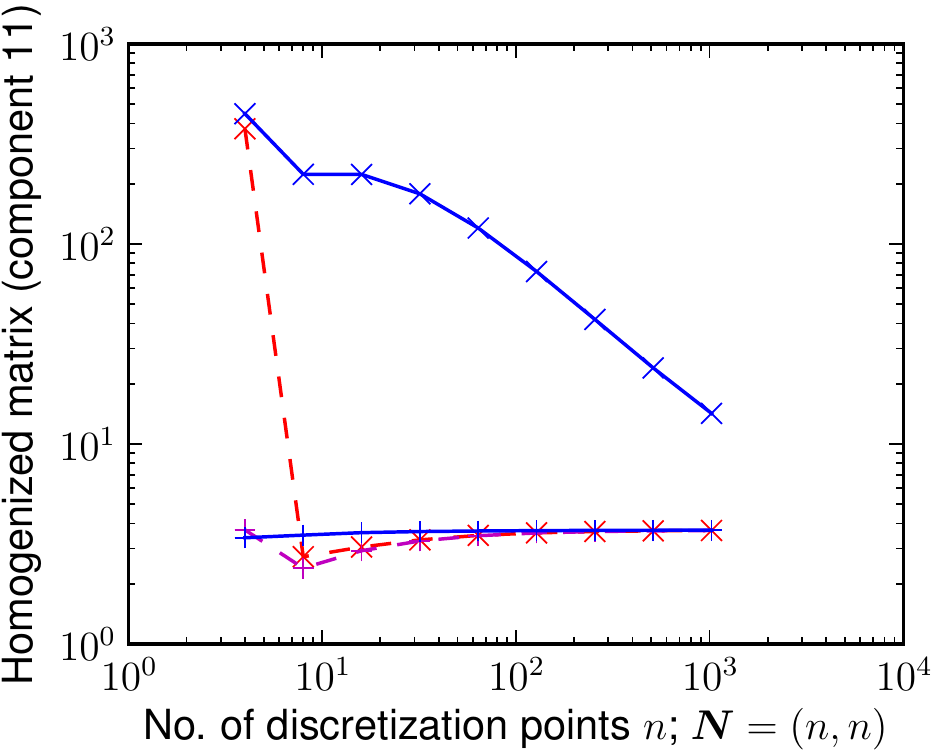}
}
\subfigure[\scriptsize Cell \eqref{eq:S2}]{
\includegraphics[scale=0.6]{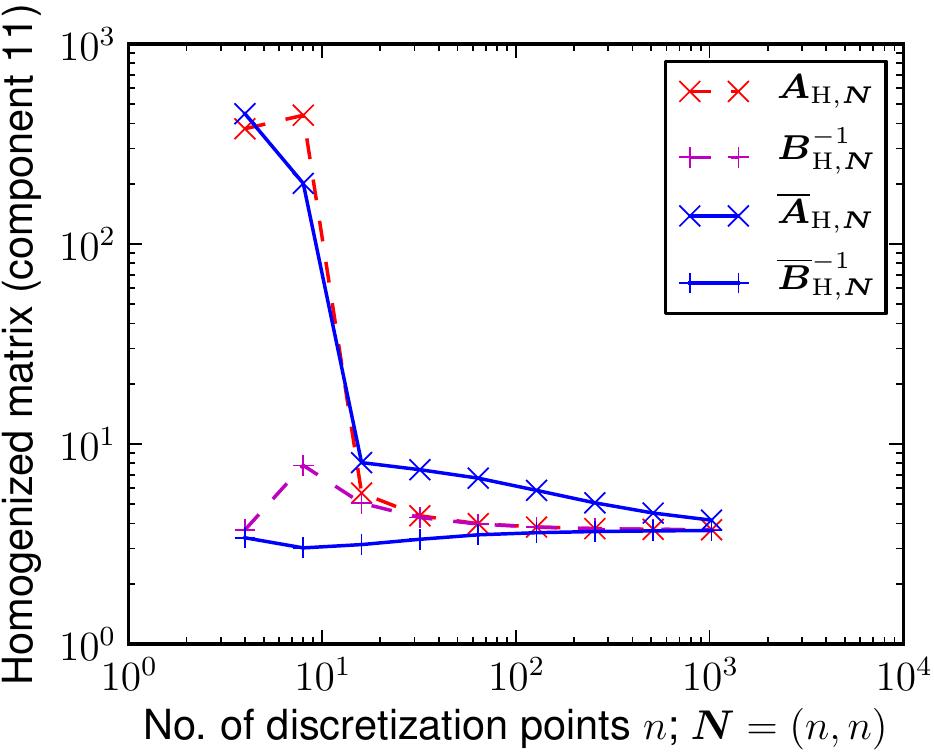}
}
\caption{Homogenized matrices for cells \eqref{eq:S1} and \eqref{eq:S2}, even grids, and phase contrast $\rho = 10^3$}
\label{fig:even_r1000}
\end{figure}

For both topologies \eqref{eq:S1} or \eqref{eq:S2}, the matrices $\AeffN$ and
$\BeffN^{-1}$ from GaNi may provide inaccurate
prediction of homogenized properties as they, in some cases, fall outside the
upper-lower bounds $\oAeffN$ and $\oBeffN^{-1}$. The mean of guaranteed bounds $\ouAeffN$, or one of the
upper-lower bounds $\oAeffN$ or $\oBeffN^{-1}$ if the worst case scenario
is needed, always provides admissible values.

Finally, in Figure~\ref{fig:even_comparison}, the upper-lower bounds $\oAeffN$
and $\oBeffN^{-1}$ are compared for both topologies \eqref{eq:S1} and
\eqref{eq:S2}, which differ only at the interface.
A significant difference is observed especially for the upper bound and the higher phase ratio $\rho=10^3$, \rev{}{which we attribute mainly to the Gibbs oscillations}.
\hide{bounds: 14.2121896338 4.15603873377 3.70469137375 3.68893890758}

\begin{figure}[htp]
\centering
\subfigure[\scriptsize Phase contrast $\rho = 10$]{
\includegraphics[scale=0.6]{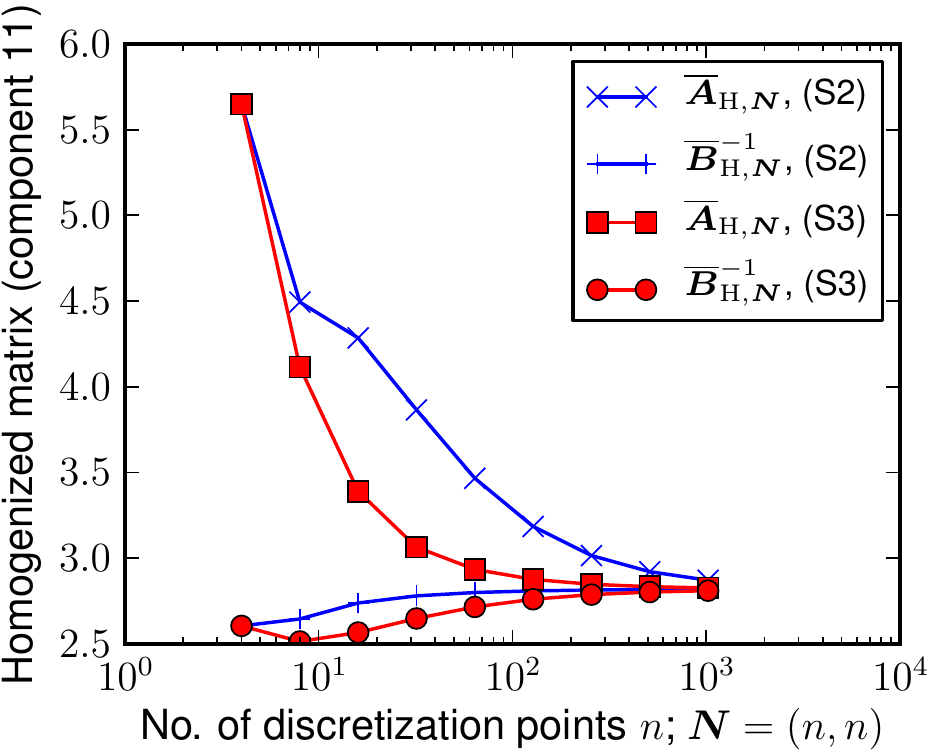}
}
\subfigure[\scriptsize Phase contrast $\rho = 10^3$]{
\includegraphics[scale=0.6]{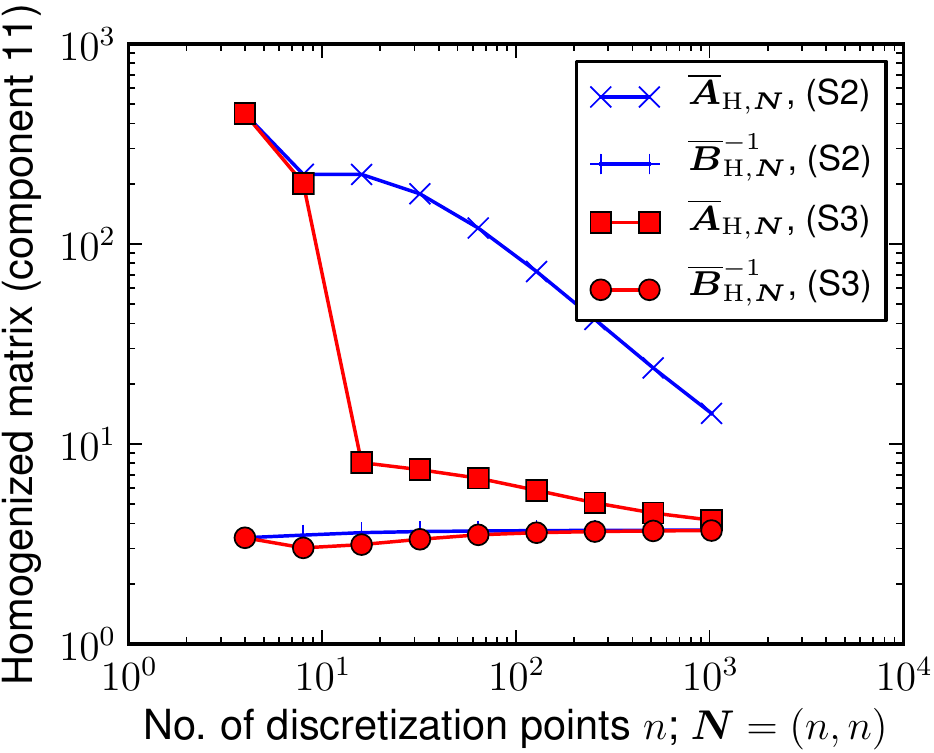}
}
\caption{Comparison of upper-lower bounds for cells \eqref{eq:S1} and \eqref{eq:S2}}
\label{fig:even_comparison}
\end{figure}
\subsection{Alkali-actived ash foam}
\label{sec:engineering_example}

We are concerned with the determination of effective thermal conductivity of an
alkali-activated ash foam, characterized with the $1,200 \times 1,200$ bitmap
shown in Figure~\ref{fig:topo_original}. The spatial distribution of the
material coefficients
\begin{align}
\label{eq:A_numerical_experiment}
\TA(\Vx) = 
\left[0.49 f(\Vx) + 0.029 \bigl(1 - f(\Vx)\bigr) \right]\mI
\quad\text{for }\Vx\in\puc
\end{align}
is defined with the help of the pixel-wise constant fly ash phase
characteristic function $f : \puc \rightarrow \{ 0, 1\}$ and the thermal
conductivities of fly ash ($0.49$~Wm$^{-2}$K$^{-1}$) and air in the
pores~($0.029$~Wm$^{-2}$K$^{-1}$)~\cite{Hlavacek2014flyash}. The
determination of all primal-dual homogenized matrices
\begin{subequations}
\label{eq:flyash_original}
\begin{align}
\label{eq:flyash_orig_gani}
\AeffN &= 
\begin{bmatrix}
0.1379997 & -0.0003841 \\
-0.0003841 & 0.1287957 \\
\end{bmatrix},
&
\BeffN^{-1} &= 
\begin{bmatrix}
0.1379880 & -0.0003840 \\
-0.0003840 & 0.1287849 \\
\end{bmatrix},
\\
\label{eq:flyash_orig_bounds}
\oAeffN &= 
\begin{bmatrix}
0.1409687 & -0.0004043 \\
-0.0004043 & 0.1319070 \\
\end{bmatrix},
&
\oBeffN^{-1} &= 
\begin{bmatrix}
0.1283959 & -0.0004628 \\
-0.0004628 & 0.1200422 \\
\end{bmatrix},
\\
\label{eq:flyash_orig_mean}
\ouAeffN &= 
\begin{bmatrix}
0.1346823 & -0.0004335 \\
-0.0004335 & 0.1259746 \\
\end{bmatrix},
&
\mD_{\VN} &= 
\begin{bmatrix}
0.0062864 & 0.0000292 \\
0.0000292 & 0.0059324 \\
\end{bmatrix},
\end{align}
\end{subequations}
involves
solutions of two linear systems with $2.88 \times 10^6$ unknowns and two right
hand sides and evaluation of lower-upper bounds by the double-grid quadrature, Section~\ref{sec:algorithm}, which took about fifteen minutes on a conventional
laptop with Intel\textcopyright Core\texttrademark i5-4200M CPU @ 2.5 GHz $\times$ 2 processor and $8$ GB of RAM. 

As in the previous section, the homogenized matrices
of the GaNi scheme \eqref{eq:flyash_orig_gani} slightly differ
because of the algebraic error due to iterative
solution of linear systems and the effect of the
Nyquist frequencies, but still satisfy
$\BeffN^{-1}\preceq\AeffN$ in agreement
with Proposition~\ref{lem:FD_even_problem}. The guaranteed error $\mD_{\VN}$,
however, remains rather large, which we attribute again to inaccuracy of
local fields in the vicinity of interfaces~\cite{willot2013fourier}.

\begin{figure}
\subfigure[Original]{\includegraphics[scale=.5]{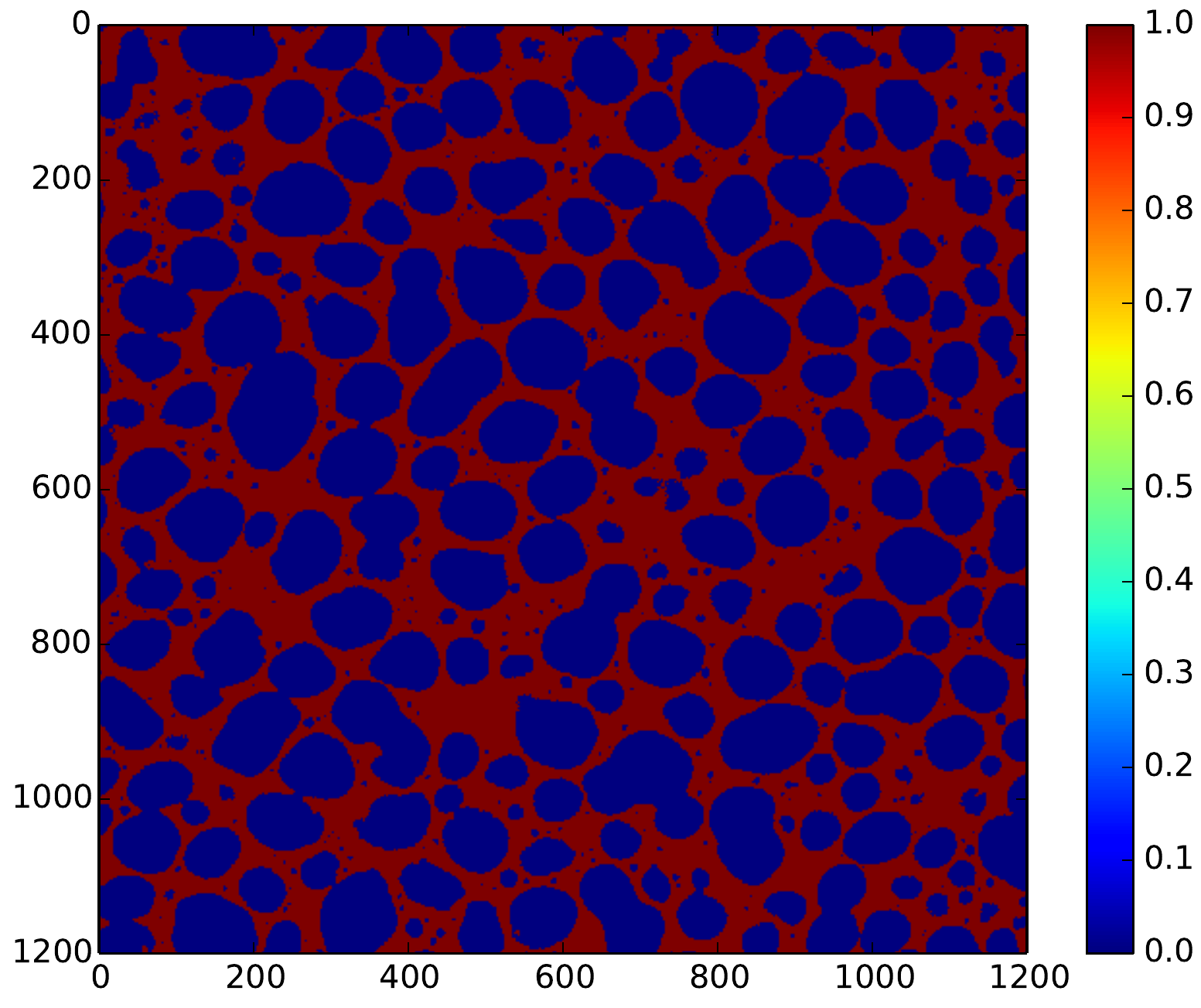}
\label{fig:topo_original}
}
\subfigure[Smoothed]{\includegraphics[scale=.5]{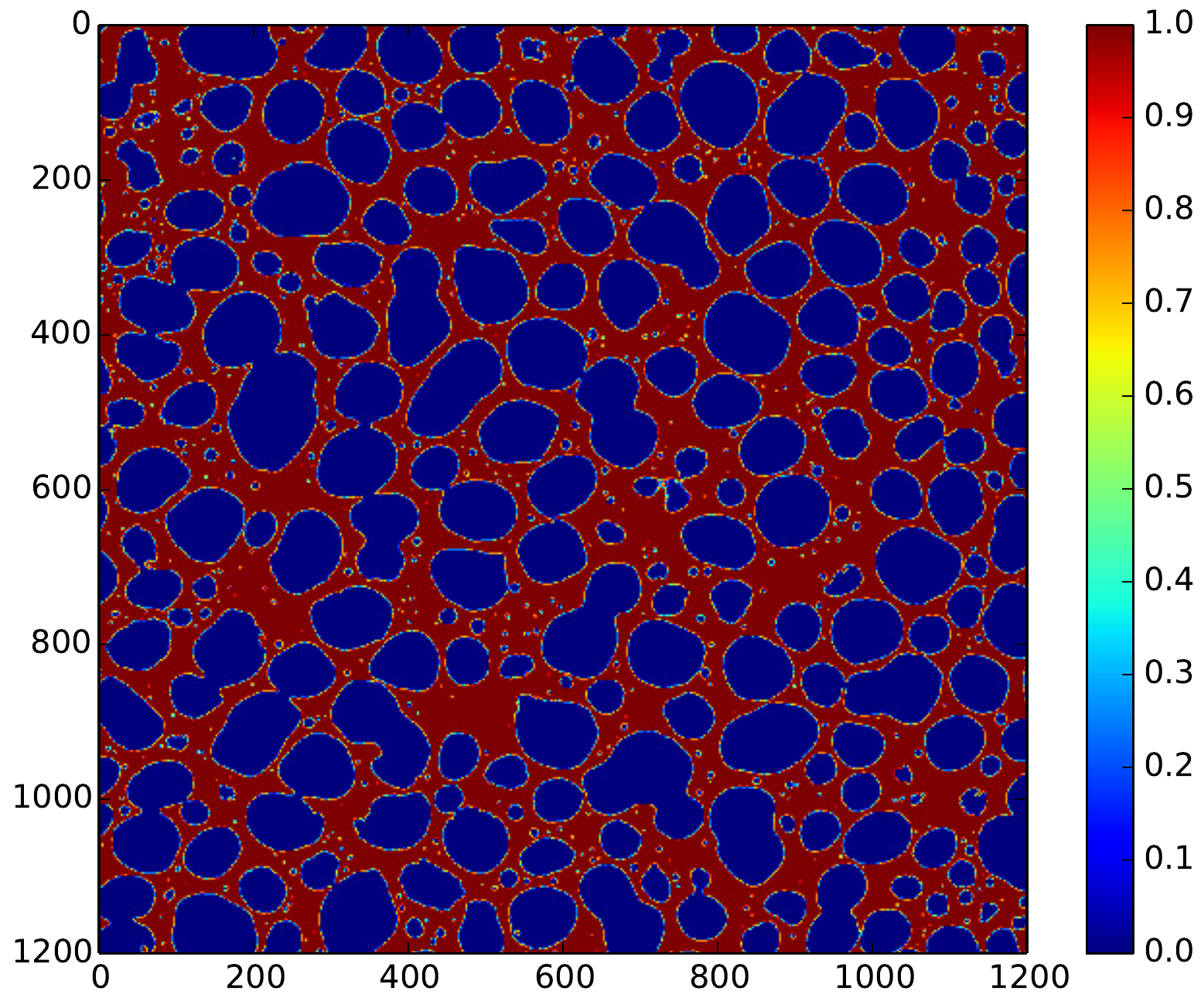}
\label{fig:topo_smooth}
}
\caption{Fly-ash
phase characteristic function. Courtesy of Petr Hlav\'{a}\v{c}ek, CTU in
Prague.}
\end{figure}

We have demonstrated in
\cite[pp.~142--145]{Vondrejc2013PhD} that
the solution accuracy can be substantially improved when smoothing the
coefficients\rev{}{, which also alleviates the Gibbs oscilations~\cite{gelebart_filtering_2015}}. For this purpose, we replace the grid values of the fly ash characteristic
function with a local average
\begin{align*}
f_\textrm{smoothed}(\Vx_\VN^\Vk) = & \frac{4}{16} f(\Vx_\VN^\Vk) + \frac{2}{16}
[f(\Vx_\VN^{\Vk+(1,0)})+f(\Vx_\VN^{\Vk-(1,0)})+f(\Vx_\VN^{\Vk+(0,1)})+f(\Vx_\VN^{\Vk-(0,1)})]
\\
+ & \frac{1}{16}
[f(\Vx_\VN^{\Vk+(1,1)})+f(\Vx_\VN^{\Vk-(1,1)})+f(\Vx_\VN^{\Vk+(-1,1)})+f(\Vx_\VN^{\Vk+(1,-1)})],
\end{align*}
which keeps the data almost unchanged, see
Figure~\ref{fig:topo_smooth}. The corresponding homogenized
properties then read as
\begin{subequations}
\label{eq:flyash_smooth}
\begin{align}
\label{eq:flyash_smooth_gani}
\AeffN &= 
\begin{bmatrix}
0.1418904 & -0.0004085 \\
-0.0004085 & 0.1328166 \\
\end{bmatrix},
&
\BeffN^{-1} &= 
\begin{bmatrix}
0.1418902 & -0.0004085 \\
-0.0004085 & 0.1328162 \\
\end{bmatrix},
\\
\label{eq:flyash_smooth_bounds}
\ol{\TA}_{\eff,\VN} &= 
\begin{bmatrix}
0.1422750 & -0.0004159 \\
-0.0004159 & 0.1332668 \\
\end{bmatrix},
&
\oBeffN^{-1} &= 
\begin{bmatrix}
0.1408147 & -0.0004152 \\
-0.0004152 & 0.1318124 \\
\end{bmatrix},
\\
\label{eq:flyash_smooth_mean}
\ouAeffN &= 
\begin{bmatrix}
0.1415448 & -0.0004155 \\
-0.0004155 & 0.1325396 \\
\end{bmatrix},
&
\mD_{\VN} &= 
\begin{bmatrix}
0.0007302 & -0.0000004 \\
-0.0000004 & 0.0007272 \\
\end{bmatrix}.
\end{align}
\end{subequations}

Notice that, as a result of smoothing, the error \eqref{eq:flyash_smooth_mean} decreases by an order of magnitude~(even more accurate results can be obtained for the Galerkin method with exact integration~\cite{Vondrejc2015FFTimproved}), while the eigenvalues of the new homogenized matrices \eqref{eq:flyash_smooth_gani} and \eqref{eq:flyash_smooth_bounds} increase. This behavior occurs because we decided to smooth the primal coefficients $\TA$; the extreme case would correspond to the Voigt bound where the coefficients are replaced with the mean value $\mean{\TA}$. By analogy, smoothing of the dual coefficients $\TA^{-1}$ decreases the homogenized properties in the direction of the Reuss bound $\mean{\TA^{-1}}^{-1}$. 

\section{Conclusion}
\label{sec:conclusion}
We have presented a method for the reliable determination of homogenized matrices
arising from the cell problem \eqref{eq:homog_problem} discretized
with the Galerkin approximation with numerical integration (GaNi),
introduced recently in \cite{VoZeMa2014FFTH} by the authors for uniform grids with an odd number of
points. The method employs trigonometric polynomials as the
approximation space and delivers conforming minimizers that are used
to evaluate guaranteed upper-lower bounds on the homogenized matrix. Our most important findings are summarized as follows:
\begin{itemize}
\item A generalization of GaNi for a non-odd number of grid points is
provided as a method for delivering conforming approximations of minimizers.
\item Primal and dual formulations are investigated in discretized and
fully discrete forms. Interestingly, duality is completely preserved for an
odd number of grid points. For non-odd discretization, the structure is
violated due to Nyquist frequencies. Our advice is to use odd grids whenever possible.
\item The idea of upper-lower bounds on a homogenized properties, independently proposed by Dvo\v{r}\'{a}k \cite{Dvorak1993master,Dvorak1995RNM} and Wi\c{e}ckowski \cite{Wieckowski1995DFEM} for the Finite Element Method (FEM), 
has been successfully applied within the framework of FFT-Galerkin methods. Moreover, thanks to convergence result in
\cite[Proposition~8]{VoZeMa2014FFTH}, these bounds can be made arbitrarily
accurate. Unlike the FEM, it results in primal and dual problems with the same structure. Therefore, our developments can be easily generalized beyond the scalar elliptic problems considered in this work, as done recently by Monchiet~\cite{Monchiet2015} for the case of linear elasticity. 
\item Our theoretical findings are confirmed by
numerical examples in Section~\ref{sec:numerical_experiments} for both odd and even discretization as well as by analysis of a real-world material system.
\end{itemize}

\appendix
\section{Primal-dual formulations}
\label{sec:appendix_duality}
This appendix is dedicated to the proof of
Proposition~\ref{lem:transform2dual} summarizing duality
arguments for both continuous \eqref{eq:homog_problem} and discrete
homogenization problems \eqref{eq:FD_GaNi} and \eqref{eq:FD_GaNi_tilde}.
Although several related results are
available in the literature,
e.g.~\cite{suquet1982dual,Jikov1994HDOIF,cherkaev2000variational}, we have failed to find them in a form compatible with our homogenization setting. Our expositions combine Dvořák's results \cite{Dvorak1993master,Dvorak1995RNM} with Ekeland and Temam's general
duality theory~\cite[page~46--51]{ekeland1976convex}.
In particular, the following lemma adjusts the arguments of \cite[Proposition~2.1 and Remark 2.3]{ekeland1976convex} to the current framework, connects the primal and the dual formulations, and provides a way to prove Proposition~\ref{lem:transform2dual}.
\begin{lemma}[Perturbation duality theorem]
\label{lem:perturbation_duality}
Consider Hilbert spaces $\rcE$ and $\cH$ such
that $\rcE\subset\cH$, and let $F:\cH\rightarrow\xR$ be a convex
functional. With $\varPhi:\cH\times \cH\rightarrow \xR$ and
$\varPhi^*:\rcE\times \cH\rightarrow\xR$ we denote a perturbed
functional with its Fenchel's conjugate, i.e.
\begin{align}\label{eq:perturbed_fenchel_conjugate}
\varPhi(\Vu,\Vv) &= F(\Vu+\Vv).
&
\varPhi^*(\Vu^*, \Vv^* ) &= \max_{\Vu\in \rcE,\Vv\in
\cH}\left[\scal{\Vu^*}{\Vu}_{\rcE} + \scal{\Vv^*}{\Vv}_{\cH} -
\varPhi(\Vu,\Vv) \right].
\end{align}
Then, the extremal values of the primal problem
\begin{align*}
\min_{\rVe\in \rcE}F(\rVe) = \min_{\rVe\in \rcE}\varPhi(\rVe, 0)
\end{align*}
and the dual problem 
\begin{align*}
\max_{\Vv^*\in \cH} -\varPhi^*(0; \Vv^* )
\end{align*}
coincide, i.e.
\begin{align*}
\min_{\rVe\in \rcE}\varPhi(\rVe, 0) = \max_{\Vv^*\in \cH} -\varPhi^*(0; \Vv^*
).
\end{align*}
\end{lemma}
\begin{proof}[Proof of Proposition~\ref{lem:transform2dual}]
In order to utilize Lemma~\ref{lem:perturbation_duality}, we define a functional $F:\cH\rightarrow\xR$ and its
perturbation $\varPhi:\cH \times \cH\rightarrow\xR$ for a
given $\VE\in\xRd$ as
\begin{align*}
F(\rVe) &= \frac{1}{2}\bilfG{\VE+\rVe}{\VE+\rVe}
&
\varPhi(\rVe,\Vv) &= \frac{1}{2} 
\bilfG{\VE+\rVe+\Vv}{\VE+\rVe+\Vv}.
\end{align*}
Because the linear operator $\rTA$ is coercive, \rev{}{the  functional 
$F$} is convex and Lemma~\ref{lem:perturbation_duality} can be
employed. The primal formulation is then equivalent to the dual formulation
\begin{align}
\label{eq:duality_homog_identity}
 \scal{\rTA_\eff \VE}{\VE}_{\xRd} = \min_{\rVe\in \rcE} 2F(\rVe) = 2 \min_{\rVe\in \rcE} \varPhi(\rVe,\MB{0}) = 2 \max_{\Vv^*\in \cH} -\varPhi^*(\MB{0},\Vv^*),
\end{align}
where $\MB{0}\in\rcE\subset\cH$ is the zero vector and
$\varPhi^*:\rcE\times\cH\rightarrow\xR$ is the Fenchel conjugate function according to~\eqref{eq:perturbed_fenchel_conjugate}$_2$.

Now, we investigate \eqref{eq:duality_homog_identity} to retrieve the dual formulation \eqref{eq:general_homog_dual}. Using substitution
 $\Vv' = \VE+\rVe+\Vv$
where $\Vv'$ covers the whole space $\cH$, we deduce
\begin{align*}
\scal{\rTA_\eff \VE}{\VE}_{\xRd} &= 2 \max_{\Vv^*\in \cH} -\varPhi^*(\MB{0},\Vv^*) 
= 2 \max_{\Vv^*\in \cH} \left[ - \max_{\substack{\rVe\in \rcE\\\Vv'\in \cH}}  \left( \scal{\Vv^*}{\Vv'-\VE-\rVe}_{\cH} - \frac{1}{2}\scal{\rTA \Vv'}{\Vv'}_{\cH} \right) \right]
\\
&= 2 \max_{\Vv^*\in \cH} \left[ \scal{\Vv^*}{\VE}_{\cH} + \min_{\rVe\in \rcE} \scal{\Vv^*}{\rVe}_{\cH}\right.
\left. -\max_{\Vv'\in \cH}\left( \scal{\Vv^*}{\Vv'}_{\cH} 
 - \frac{1}{2}\scal{\rTA \Vv'}{\Vv'}_{\cH} \right)  \right].
\end{align*}
We focus on the maximizer $\tilde{\Vv}'$ of the last equation, which satisfies
 $\rTA \tilde{\Vv}' = \Vv^*$.
\rev{}{By the assumptions of} Proposition~\ref{lem:transform2dual}, the operator $\rTA$ is
invertible; hence, we obtain $\tilde{\Vv}' = \rTA^{-1}\Vv^*$, and 
the inner $\max$-term simplifies to $$\max_{\Vv'\in \cH}\left( \scal{\Vv^*}{\Vv'}_{\cH} - \frac{1}{2}\scal{\rTA \Vv'}{\Vv'}_{\cH} \right) = \frac{1}{2}\scal{\rTA^{-1} \Vv^*}{\Vv^*}_{\cH}.$$
The inner $\min$-term equals to the negative value of the indicator function of $\rcU\oplus\rcJ$ as
$$\min_{\rVe\in \rcE} \scal{\Vv^*}{\rVe}_{\cH} =
\begin{cases}
0 &\text{for }\Vv^* \in \rcU\oplus\rcJ,
\\
-\infty &\text{otherwise}.
\end{cases}
$$
This term can be omitted when restricting optimization to $\rcU\oplus\rcJ$.
We proceed to
\begin{align}
\scal{\rTA_\eff \VE}{\VE}_{\xRd}
&= 2 \max_{\Vv^*\in \rcU\oplus\rcJ} \left[ \scal{\Vv^*}{\VE}_{\cH} - \frac{1}{2}\scal{\rTA^{-1} \Vv^*}{\Vv^*}_{\cH} \right]
\nonumber
\\
\label{eq:last_equation_0}
&= 2 \max_{\VJ\in \xRd} \left[ \scal{\VJ}{\VE}_{\xRd} - \min_{\rVj \in \rcJ} \frac{1}{2}\scal{\rTA^{-1} (\VJ + \rVj)}{\VJ + \rVj}_{\cH} \right],
\end{align}
where we have utilized the decomposition  $\Vv^* = \VJ+\rVj\in\rcU\oplus\rcJ$ with $\VJ\in\rcU$ and $\rVj\in\rcJ$.

The $\min$-term in the last equation \eqref{eq:last_equation_0} already matches the dual formulation \eqref{eq:general_homog_dual},
namely
\begin{align*}
 \scal{\rTB_{\eff} \VJ}{\VJ }_{\xRd} = \min_{\rVj\in \rcJ}  \scal{\rTA^{-1} (\VJ + \rVj)}{\VJ + \rVj}_{\cH}.
\end{align*}
Now, we will show that $\rTA_{\eff}=\rTB_{\eff}^{-1}$ as claimed in~\eqref{eq:duality_mutually_inverse_property}. Notice first that the matrix $\rTB_{\eff}$ is invertible as it is symmetric and coercive because the linear operator $\rTA^{-1}$ is symmetric and coercive (see Remark~\ref{rem:Aeff_spd} for similar arguments).
Next, the dual formulation \eqref{eq:last_equation_0} simplifies to
\begin{align}\label{eq:dual_pom}
\scal{\rTA_\eff \VE}{\VE}_{\xRd} 
&= \max_{\VJ\in \xRd} \left[ 2\scal{\VJ}{\VE}_{\xRd} -  \scal{\rTB_{\eff} \VJ}{\VJ}_{\xRd} \right]
\end{align}
and the maximum is attained for $\VJ=\rTB_\eff^{-1}\VE$ that, when substituted back to \eqref{eq:dual_pom}, provides the desired identity \eqref{eq:duality_mutually_inverse_property}.

The relation between minimizers in 
\eqref{eq:connection_primal_dual_fields} using the identity between homogenized matrices \eqref{eq:duality_mutually_inverse_property} must still be proven.
Indeed, the stationarity condition for the primal formulation
\eqref{eq:general_homog_primal}, i.e.
$\bilfG{\VE+\rVe\mac{\VE}}{\Vv} = 0$
for all $\Vv\in\rcE$,
reveals
\begin{align}\label{eq:pom_dual_effective}
 \rTA(\VE+\rVe\mac{\VE})\in\rcU\oplus\rcJ
\qquad\text{and}\qquad
 \scal{\rTA_\eff\VE}{\VE}_{\xRd} &= \bilfG{\VE+\rVe\mac{\VE}}{\VE+\rVe\mac{\VE}} = \bilfG{\VE+\rVe\mac{\VE}}{\VE},
\end{align}
holding for arbitrary $\VE\in\xRd$. Combining both in  \eqref{eq:pom_dual_effective} and setting $\VJ=\mathring{\TA}_\eff\VE$, we obtain
$\rTA(\VE+\rVe\mac{\VE}) = \VJ + \breve{\Vj}\mac{\VJ}\quad\text{with }\breve{\Vj}\mac{\VJ}\in\rcJ$.

The relation \eqref{eq:connection_primal_dual_fields} will be
established once showing that
$\breve{\Vj}\mac{\VJ} = \rVj\mac{\VJ}$.
For $\VJ=\mathring{\TA}_\eff\VE$, the extremal values in \eqref{eq:general_homog} coincide, so that
\begin{align*}
\bilfGi{\VJ+\rVj\mac{\VJ}}{\VJ+\rVj\mac{\VJ}}
&= \bilfG{\VE+\rVe\mac{\VE}}{\VE+\rVe\mac{\VE}}
= \scal{\rTA(\VE+\rVe\mac{\VE})}{\VE+\rVe\mac{\VE}}_{\cH}
\\
&= \scal{\rTA^{-1}\rTA(\VE+\rVe\mac{\VE})}{\rTA(\VE+\rVe\mac{\VE})}_{\cH} 
= \bilfGi{\VJ+\breve{\Vj}\mac{\VJ}}{\VJ+\breve{\Vj}\mac{\VJ}},
\end{align*}
and hence 
$\breve{\Vj}\mac{\VJ} = \rVj\mac{\VJ}$ holds because $\rVe\mac{\VE}$ is
the unique minimizer.
\end{proof}

\section*{Acknowledgments}
The authors are thankful to Jaroslav Haslinger for bringing Jan Dvo\v{r}\'{a}k's works \cite{Dvorak1993master,Dvorak1995RNM} to our attention.
This work was supported by the Czech Science Foundation through project  No.~P105/12/0331. J.~Vond\v{r}ejc received the support from project EXLIZ -- CZ.1.07/2.3.00/30.0013, which is
co-financed by the European Social Fund and the state budget of the Czech Republic, and J.~Zeman from the European
Regional Development Fund under the IT4Innovations Center of Excellence, project No.~CZ.1.05/1.1.00/02.0070.


\begin{thebibliography}{10}
\expandafter\ifx\csname url\endcsname\relax
  \def\url#1{\texttt{#1}}\fi
\expandafter\ifx\csname urlprefix\endcsname\relax\def\urlprefix{URL }\fi
\expandafter\ifx\csname href\endcsname\relax
  \def\href#1#2{#2} \def\path#1{#1}\fi

\bibitem{OBBH2003research-directions}
J.~T. Oden, T.~Belytschko, I.~Babu\v{s}ka, T.~J.~R. Hughes, {Research
  directions in computational mechanics}, Computer Methods in Applied Mechanics
  and Engineering 192~(7) (2003) 913--922.

\bibitem{VoZeMa2014FFTH}
J.~Vondřejc, J.~Zeman, I.~Marek,
\href{http://dx.doi.org/10.1016/j.camwa.2014.05.014}{{An FFT-based Galerkin method for homogenization of periodic media}}, Computers
  \& Mathematics with Applications 68~(3) (2014) 156--173.

\bibitem{Moulinec1994FFT}
H.~Moulinec, P.~Suquet, {A fast numerical method for computing the linear and
  nonlinear mechanical properties of composites}, Comptes rendus de
  l'Acad\'{e}mie des sciences. S\'{e}rie II, M\'{e}canique, physique, chimie,
  astronomie 318~(11) (1994) 1417--1423.

\bibitem{Montagnat2014multiscale}
M.~Montagnat, O.~Castelnau, P.~D. Bons, S.~H. Faria, O.~Gagliardini,
  F.~Gillet-Chaulet, F.~Grennerat, A.~Griera, R.~A. Lebensohn, H.~Moulinec,
  J.~Roessiger, P.~Suquet, {Multiscale modeling of ice deformation behavior},
  Journal of Structural Geology 61 (2014) 78--108.

\bibitem{Sliseris2014}
J.~Sliseris, H.~Andr\"{a}, M.~Kabel, B.~Dix, B.~Plinke, O.~Wirjadi, G.~Frolovs,
  {Numerical prediction of the stiffness and strength of medium density
  fiberboards}, Mechanics of Materials 79 (2014) 73--84.

\bibitem{Stein2014fatigue}
C.~A. Stein, A.~Cerrone, T.~Ozturk, S.~Lee, P.~Kenesei, H.~Tucker, R.~Pokharel,
  J.~Lind, C.~Hefferan, R.~M. Suter, A.~R. Ingraffea, A.~D. Rollett, {Fatigue
  crack initiation, slip localization and twin boundaries in a nickel-based
  superalloy}, Current Opinion in Solid State and Materials Science 18~(4)
  (2014) 244--252.

\bibitem{Bensoussan1978per_structures}
G.~Papanicolau, A.~Bensoussan, J.~Lions, {Asymptotic analysis for periodic
  structures}, Vol.~5, North Holland, 1978.

\bibitem{nguetseng1989general}
G.~Nguetseng, {A general convergence result for a functional related to the
  theory of homogenization}, SIAM Journal on Mathematical Analysis 20~(3)
  (1989) 608--623.

\bibitem{allaire1992homogenization}
G.~Allaire, {Homogenization and two-scale convergence}, SIAM Journal on
  Mathematical Analysis 23~(6) (1992) 1482--1518.

\bibitem{cioranescu2008unfolding}
D.~Cioranescu, A.~Damlamian, G.~Griso,
  \href{http://epubs.siam.org/doi/abs/10.1137/080713148}{{The periodic
  unfolding method in homogenization}}, SIAM Journal on Mathematical Analysis
  40~(4) (2008) 1585--1620.

\bibitem{Flaherty1973}
J.~E. Flaherty, J.~B. Keller,
{Elastic behavior of
  composite media}, Communications on Pure and Applied Mathematics 26~(4)
  (1973) 565--580.

\bibitem{Panasenko1988}
G.~P. Panasenko, {Numerical solution of cell problems in averaging theory},
  \{USSR\} Computational Mathematics and Mathematical Physics 28~(1) (1988)
  183--186.

\bibitem{Garboczi:1998:FEFD}
E.~Garboczi, {Finite
  element and finite difference programs for computing the linear electric and
  elastic properties of digital images of random materials.}, Tech. Rep.
  NISTIR 6269, Building and Fire Research Laboratory, National Institute of
  Standards and Technology, Gaithesburg, Maryland 2089 (1998).

\bibitem{Guedes1990}
J.~M. Guedes, N.~Kikuchi,
  \href{http://www.sciencedirect.com/science/article/pii/004578259090148F}{{Preprocessing
  and postprocessing for materials based on the homogenization method with
  adaptive finite element methods}}, Computer Methods in Applied Mechanics and
  Engineering 83~(2) (1990) 143--198.

\bibitem{Michel1999}
J.-C.~C. Michel, H.~Moulinec, P.~Suquet,
  \href{http://www.sciencedirect.com/science/article/pii/S0045782598002278}{{Effective
  properties of composite materials with periodic microstructure: a
  computational approach}}, Computer Methods in Applied Mechanics and
  Engineering 172~(1--4) (1999) 109--143.

\bibitem{Geers2010}
M.~G.~D. Geers, V.~G. Kouznetsova, W.~A.~M. Brekelmans,
  \href{http://www.sciencedirect.com/science/article/pii/S0377042709005536}{{Multi-scale
  computational homogenization: Trends and challenges}}, Journal of
  Computational and Applied Mathematics 234~(7) (2010) 2175--2182.

\bibitem{Eischen1993BEM}
J.~Eischen, S.~Torquato,
  \href{http://scitation.aip.org/content/aip/journal/jap/74/1/10.1063/1.354132}{{Determining
  elastic behavior of composites by the boundary element method}}, Journal of
  applied physics 74~(1) (1993) 159--170.

\bibitem{Kaminski1999BEM}
M.~Kamiński,
  \href{http://www.sciencedirect.com/science/article/pii/S0955799799000296}{{Boundary
  element method homogenization of the periodic linear elastic fiber
  composites}}, Engineering Analysis with Boundary Elements 23~(10) (1999)
  815--823.

\bibitem{Prochazka2001BEM}
P.~Proch\'{a}zka, {Homogenization of linear and of debonding composites using
  the \{BEM\}}, Engineering Analysis with Boundary Elements 25~(9) (2001)
  753--769.

\bibitem{Greengard1998}
L.~Greengard, J.~Helsing,
  \href{http://www.sciencedirect.com/science/article/pii/S0022509697000410}{{On
  the numerical evaluation of elastostatic fields in locally isotropic
  two-dimensional composites}}, Journal of the Mechanics and Physics of Solids
  46~(8) (1998) 1441--1462.

\bibitem{Greengard2006}
L.~Greengard, J.~Lee, {Electrostatics and heat conduction in high contrast
  composite materials}, Journal of Computational Physics 211~(1) (2006) 64--76.

\bibitem{Helsing2011effective}
J.~Helsing, {The effective conductivity of arrays of squares: large random unit
  cells and extreme contrast ratios}, Journal of Computational Physics 230~(20)
  (2011) 7533--7547.

\bibitem{Brisard2010FFT}
S.~Brisard, L.~Dormieux, {FFT-based methods for the mechanics of composites: A
  general variational framework}, Computational Materials Science 49~(3) (2010)
  663--671.

\bibitem{Brisard2012FFT}
S.~Brisard, L.~Dormieux,
  \href{http://www.sciencedirect.com/science/article/pii/S0045782512000059}{{Combining
  Galerkin approximation techniques with the principle of Hashin and Shtrikman
  to derive a new FFT-based numerical method for the homogenization of
  composites}}, Computer Methods in Applied Mechanics and Engineering 217--220
  (2012) 197--212.

\bibitem{hashin1962elast}
Z.~Hashin, S.~Shtrikman,
  \href{http://www.sciencedirect.com/science/article/pii/0022509662900042}{{On
  some variational principles in anisotropic and nonhomogeneous elasticity}},
  Journal of the Mechanics and Physics of Solids 10~(4) (1962) 335--342.

\bibitem{Schneider2014convergence}
M.~Schneider, \href{http://doi.wiley.com/10.1002/mma.3259}{{Convergence of
  FFT-based homogenization for strongly heterogeneous media}}, Mathematical
  Methods in the Applied Sciences.

\bibitem{Eyre1999FNS}
D.~J. Eyre, G.~W. Milton, {A fast numerical scheme for computing the response
  of composites using grid refinement}, The European Physical Journal Applied
  Physics 6~(1) (1999) 41--47.

\bibitem{Vinogradov2008AFFT}
V.~Vinogradov, G.~W. Milton, {An accelerated FFT algorithm for thermoelastic
  and non-linear composites}, International Journal for Numerical Methods in
  Engineering 76~(11) (2008) 1678--1695.

\bibitem{ZeVoNoMa2010AFFTH}
J.~Zeman, J.~Vondřejc, J.~Nov\'{a}k, I.~Marek,
  \href{http://linkinghub.elsevier.com/retrieve/pii/S0021999110003931}{{Accelerating
  a FFT-based solver for numerical homogenization of periodic media by
  conjugate gradients}}, Journal of Computational Physics 229~(21) (2010)
  8065--8071.

\bibitem{Brown2002DFT}
C.~M. Brown, W.~Dreyer, W.~H. M\"{u}ller,
  \href{http://rspa.royalsocietypublishing.org/content/458/2024/1967.short}{{Discrete
  Fourier transforms and their application to stress—strain problems in
  composite mechanics: a convergence study}}, Proceedings of the Royal Society
  A: Mathematical, Physical and Engineering Sciences 458~(2024) (2002)
  1967--1987.

\bibitem{Michel2000CMB}
J.~Michel, H.~Moulinec, P.~Suquet, {A computational method based on augmented
  Lagrangians and fast Fourier transforms for composites with high contrast},
  CMES: Computer Modeling in Engineering \& Sciences 1~(2) (2000) 79--88.

\bibitem{Monchiet2012polarization}
V.~Monchiet, G.~Bonnet, {A polarization-based FFT iterative scheme for
  computing the effective properties of elastic composites with arbitrary
  contrast}, International Journal for Numerical Methods in Engineering 89~(11)
  (2012) 1419--1436.

\bibitem{Monchiet2013conduct}
V.~Monchiet, G.~Bonnet,
  \href{http://www.emeraldinsight.com/doi/abs/10.1108/HFF-10-2011-0207}{{A
  polarization-based fast numerical method for computing the effective
  conductivity of composites}}, International Journal of Numerical Methods for
  Heat \& Fluid Flow 23~(7) (2013) 1256--1271.

\bibitem{willot2013fourier}
F.~Willot, B.~Abdallah, Y.-P. Pellegrini,
  \href{http://onlinelibrary.wiley.com/doi/10.1002/nme.4641/full}{{Fourier-based
  schemes with modified Green operator for computing the electrical response of
  heterogeneous media with accurate local fields}}, International Journal for
  Numerical Methods in Engineering 98~(7) (2014) 518--533.

\bibitem{VoZeMa2012LNSC}
J.~Vondřejc, J.~Zeman, I.~Marek, {Analysis of a Fast Fourier Transform Based
  Method for Modeling of Heterogeneous Materials}, Lecture Notes in Computer
  Science 7116 (2012) 512--522.

\bibitem{voigt1910lehrbuch}
W.~Voigt, {Lehrbuch der kristallphysik}, Vol.~34, BG Teubner, 1910.

\bibitem{Reuss1929}
A.~Reuss, {Berechnung der Flie\ss grenze von Mischkristallen auf Grund der
  Plastizit\"{a}tsbedingung f\"{u}r Einkristalle}, ZAMM-Journal of Applied
  Mathematics and Mechanics/Zeitschrift f\"{u}r Angewandte Mathematik und
  Mechanik 9~(1) (1929) 49--58.

\bibitem{hashin1963variational}
Z.~Hashin, S.~Shtrikman,
  \href{http://www.sciencedirect.com/science/article/pii/0022509663900607}{{A
  variational approach to the theory of the elastic behaviour of multiphase
  materials}}, Journal of the Mechanics and Physics of Solids 11~(2) (1963)
  127--140.

\bibitem{Milton2002TC}
G.~W. Milton, {The Theory of Composites}, Vol.~6 of Cambridge Monographs on
  Applied and Computational Mathematics, Cambridge University Press, Cambridge,
  UK, 2002.

\bibitem{torquato2002random}
S.~Torquato, {Random heterogeneous materials: microstructure and macroscopic
  properties}, Vol.~16, Springer, 2002.

\bibitem{cherkaev2000variational}
A.~Cherkaev, {Variational methods for structural optimization}, Vol. 140,
  Springer, 2000.

\bibitem{dvorak2012micromechanics}
G.~J. Dvorak, {Micromechanics of Composite Materials}, Vol. 186, Springer,
  2012.

\bibitem{Dvorak1993master}
J.~Dvoř\'{a}k, {Optimization of Composite Materials}, Ph.D. thesis, Charles
  University (Jun. 1993).

\bibitem{Dvorak1995RNM}
J.~Dvoř\'{a}k,
  \href{http://citeseerx.ist.psu.edu/viewdoc/summary?doi=10.1.1.45.1190}{{A
  Reliable Numerical Method for Computing Homogenized Coefficients}}, Tech.
  Rep. 0045, Charles University in Prague, Faculty of Mathematics and Physics
  (1995).

\bibitem{haslinger1995optimum}
J.~Haslinger, J.~Dvoř\'{a}k, {Optimum composite material design},
  RAIRO-Mathematical Modelling and Numerical Analysis-Modelisation Mathematique
  et Analyse Numerique 29~(6) (1995) 657--686.

\bibitem{Wieckowski1995DFEM}
Z.~Wieckowski, {Dual Finite Element Methods in Mechanics of Composite
  Materials}, Journal of Theoretical and Applied Mechanics 2~(33) (1995)
  233--252.

\bibitem{kabel2012precisebounds}
M.~Kabel, H.~Andr\"{a},
  \href{http://math2market.de/Publications/2013ReportFraunhoferITWM\_Nr224.pdf}{{Fast
  numerical computation of precise bounds of effective elastic moduli}} (2012)
  1--16.

\bibitem{bignonnet2014fft}
F.~Bignonnet, L.~Dormieux,
  \href{http://onlinelibrary.wiley.com/doi/10.1002/nag.2278/full}{{FFT-based bounds on the permeability
  of complex microstructures}}, International Journal for Numerical and
  Analytical Methods in Geomechanics 38~(16) (2014) 1707--1723.

\bibitem{SaVa2000PIaPDE}
J.~Saranen, G.~Vainikko, {Periodic Integral and Pseudodifferential Equations
  with Numerical Approximation}, Springer Monographs Mathematics, Berlin,
  Heidelberg, 2002.

\bibitem{horn2012matrix}
R.~A. Horn, C.~R. Johnson, {Matrix analysis}, 2nd Edition, Cambridge University
  Press, 32 Avenue of the Americas, New York, NY 10013-2473, USA, 2013.

\bibitem{rudin1986real}
W.~Rudin, {Real and complex analysis}, 3rd Edition, McGraw-Hill, New York,
  1986.

\bibitem{Jikov1994HDOIF}
V.~V. Jikov, S.~M. Kozlov, O.~A. Oleinik, {Homogenization of Differential
  Operators and Integral Functionals}, Springer-Verlag, 1994.

\bibitem{Cioranescu1999Intro2Homog}
D.~Cioranescu, P.~Donato, {An Introduction to Homogenization}, Oxford Lecture
  Series in Mathematics and Its Applications, Oxford University Press, 1999.

\bibitem{Moulinec1998NMC}
H.~Moulinec, P.~Suquet,
  \href{http://linkinghub.elsevier.com/retrieve/pii/S0045782597002181}{{A
  numerical method for computing the overall response of nonlinear composites
  with complex microstructure}}, Computer Methods in Applied Mechanics and
  Engineering 157~(1--2) (1998) 69--94.

\bibitem{Burrus1985DFT/FFT}
C.~S. Burrus, T.~W. Parks, {DFT/FFT and Convolution Algorithms}, Wiley, New
  York, 1985.

\bibitem{Vondrejc2013PhD}
J.~Vondřejc,
  \href{http://mech.fsv.cvut.cz/wiki/images/4/49/PhD\_dissertation\_Vondrejc\_2013.pdf}{{FFT-Based
  Method for Homogenization of Periodic Media: Theory and Applications}}, Ph.D.
  thesis, Czech Technical University in Prague (Jan. 2013).

\bibitem{boyd_chebyshev_2001}
J.~P. Boyd, Chebyshev and {Fourier} Spectral Methods: {Second} Revised Edition,
  Courier Corporation, 2001.

\bibitem{Strang1972varcrime}
G.~Strang, {Variational crimes in the finite element method}, The mathematical
  foundations of the finite element method with applications to partial
  differential equations (1972) 689--710.

\bibitem{Necas2012direct}
J.~Ne\v{c}as, {Direct methods in the theory of elliptic equations}, Springer,
  2012.

\bibitem{Vondrejc2015FFTimproved}
J.~Vondřejc, \href {http://arxiv.org/abs/1412.2033}
  {\path{arXiv:1412.2033}}.

\bibitem{Nemat-Nasser1999}
S.~Nemat-Nasser, M.~Hori,
  \href{http://tocs.ulb.tu-darmstadt.de/93002114.pdf}{{Micromechanics: overall
  properties of heterogeneous materials}}, North Holland, Amsterdam, 1999.

\bibitem{Bonnet2007}
G.~Bonnet,
  \href{http://www.sciencedirect.com/science/article/pii/S0022509606001918}{{Effective
  properties of elastic periodic composite media with fibers}}, Journal of the
  Mechanics and Physics of Solids 55~(5) (2007) 881--899.

\bibitem{Moulinec2014comparison}
H.~Moulinec, F.~Silva, \href{http://doi.wiley.com/10.1002/nme.4614}{{Comparison
  of three accelerated FFT-based schemes for omputing the mechanical response
  of composite materials}}, International Journal for Numerical Methods in
  Engineering 97~(13) (2014) 960--985.

\bibitem{willot_fourier-based_2015}
F.~Willot,
  \href{http://www.sciencedirect.com/science/article/pii/S1631072114002149}{Fourier-based
  schemes for computing the mechanical response of composites with accurate
  local fields}, Comptes Rendus Mécanique 343~(3) (2015) 232--245.

\bibitem{Hlavacek2014flyash}
P.~Hlav\'{a}\v{c}ek, V.~\v{S}milauer, F.~\v{S}kv\'{a}ra, L.~Kopeck\'{y},
  R.~\v{S}ulc,
  \href{http://linkinghub.elsevier.com/retrieve/pii/S0955221914004403}{{Inorganic
  foams made from alkali-activated fly ash: Mechanical, chemical and physical
  properties}}, Journal of the European Ceramic Society 35~(2) (2015) 703--709.

\bibitem{gelebart_filtering_2015}
L.~Gélébart, F.~Ouaki,
  \href{http://www.sciencedirect.com/science/article/pii/S002199911500203X}{Filtering
  material properties to improve {FFT}-based methods for numerical
  homogenization}, Journal of Computational Physics 294 (2015) 90--95.

\bibitem{Monchiet2015}
V.~Monchiet,
  \href{http://www.sciencedirect.com/science/article/pii/S0045782514003661}{{Combining
  FFT methods and standard variational principles to compute bounds and
  estimates for the properties of elastic composites}}, Computer Methods in
  Applied Mechanics and Engineering 283 (2015) 454--473.

\bibitem{suquet1982dual}
P.~Suquet, {Une m\'{e}thode duale en homog\'{e}n\'{e}isation: application aux
  milieux \'{e}lastiques}, Journal de M\'{e}canique th\'{e}orique et
  Appliqu\'{e}e (Special issue) (1982) 79--98.

\bibitem{ekeland1976convex}
I.~Ekeland, R.~Temam, {Convex Analysis and Variational Problems}, SIAM, 1976.

\end{thebibliography}

\end{document}